\documentclass[11pt]{amsart}
\usepackage{amsmath}
\usepackage{amsfonts}
\usepackage{amssymb}
\usepackage{graphicx}
\usepackage{hyperref}
\usepackage{mathrsfs}
\usepackage{pb-diagram}
\usepackage{epstopdf}
\usepackage{amsmath,amsfonts,amsthm,enumerate,amscd,latexsym,curves}
\usepackage{bbm}
\usepackage[mathscr]{eucal}
\usepackage{epsfig,epsf,epic}
\usepackage{caption}
\usepackage{subcaption}
\usepackage{multirow}
\usepackage{color}
\usepackage[all]{xy}
\usepackage{fourier}
\usepackage{comment}

\newtheorem{thm}{Theorem}[section]
\newtheorem{lemma}[thm]{Lemma}

\newtheorem{prop}[thm]{Proposition}
\newtheorem{defn}[thm]{Definition}

\newtheorem{remark}[thm]{Remark}

\numberwithin{equation}{section}

\newcommand{\Z}{\mathbb{Z}}

\newcommand{\R}{\mathbb{R}}
\newcommand{\C}{\mathbb{C}}

\newcommand{\bP}{\mathbb{P}}
\newcommand{\quat}{\mathbb{H}}

\newcommand{\bS}{\mathbb{S}}

\newcommand{\cS}{\mathcal{S}}

\newcommand{\consti}{\mathbf{i}\,}

\newcommand{\WT}[1]{\widetilde{#1}}

\newcommand{\bD}{\mathcal{D}}

\newcommand{\pt}{\mathrm{pt}}

\newcommand{\cA}{\mathcal{A}}

\newcommand{\End}{\mathrm{End}}

\newcommand{\Stab}{\mathrm{Stab}}

\newcommand{\one}{\mathbf{1}}

\newcommand{\Fuk}{\mathrm{Fuk}}

\newcommand{\bL}{\mathbb{L}}

\newcommand{\cO}{\mathcal{O}}
\newcommand{\Per}{\mathrm{Per}}

\begin{document}

\title{Mirror of Atiyah flop in symplectic geometry and stability conditions}
\author[Fan]{Yu-Wei Fan}
\address{Department of Mathematics\\Harvard University}
\email{ywfan@math.harvard.edu}
\author[Hong]{Hansol Hong}
\address{Center of Mathematical Sciences and Applications\\Harvard University}
\email{hhong@cmsa.fas.harvard.edu, hansol84@gmail.com}
\author[Lau]{Siu-Cheong Lau}
\address{Department of Mathematics and Statistics\\ Boston University}
\email{lau@math.bu.edu}
\author[Yau]{Shing-Tung Yau}
\address{Department of Mathematics\\Harvard University}
\email{yau@math.harvard.edu}

\begin{abstract}
We study the mirror operation of the Atiyah flop in symplectic geometry.  We formulate the operation for a symplectic manifold with a Lagrangian fibration.  Furthermore we construct geometric stability conditions on the derived Fukaya category of the deformed conifold and study the action of the mirror Atiyah flop on these stability conditions.

\end{abstract}

\maketitle
\section{Introduction}
Flop is a fundamental operation in birational geometry.  By the work of Koll\'ar \cite{Kollar}, any birational transformation of compact threefolds with nef canonical classes and $Q$-factorial terminal singularities can be decomposed into flops.

Atiyah flop is the most well-known among many different kinds of flops.  It contracts a $(-1,-1)$ curve and resolves the resulting conifold singularity by a small blow-up, producing a $(-1,-1)$ curve in another direction, see Figure \ref{fig:Atiyah-flop}.

\begin{figure}[h]
\begin{center}
\includegraphics[scale=0.3]{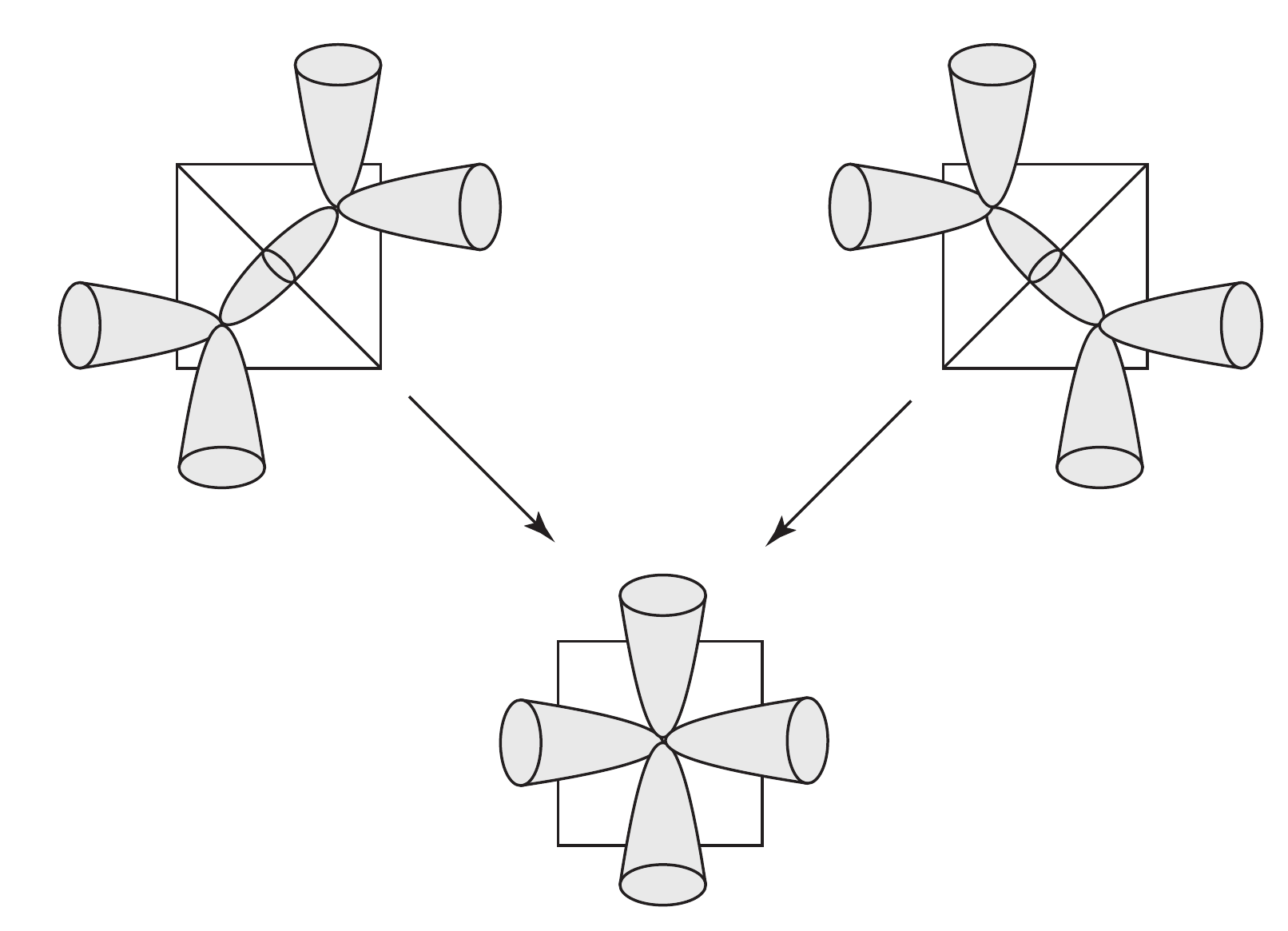}
\caption{The Atiyah flop.}
\label{fig:Atiyah-flop}
\end{center}
\end{figure}

In mirror symmetry, complex and symplectic geometries are dual to each other.  Flop is an important operation in complex geometry.  It is natural to ask whether there is a mirror operation in symplectic geometry.  In this paper we focus on the mirror of Atiyah flop.

SYZ mirror symmetry of a conifold singularity is well-known by the works of \cite{Gross-eg,CLL,CM2,AAK,CPU,KL}.  A conifold singularity is given by $u_1v_1=u_2v_2$ in $\C^4$.  There are two different choices of anti-canonical divisors which turn out to be mirror to each other, namely $D_1=\{u_2v_2=1\}$ and $D_2=\{(u_2-1)(v_2-1)=0\}$.  Consider the resolved conifold $\cO_{\bP^1}(-1)\oplus \cO_{\bP^1}(-1)$, with the divisor $D_2$ deleted.  Its SYZ mirror is given by the deformed conifold $\{(u_1,v_1,u_2,v_2,z) \in \C^4 \times \C^\times: u_1v_1=z+q, u_2v_2=z+1\}$.  Here $q$ is the K\"ahler parameter of the resolved conifold, namely $q=e^{-A}$ where $A$ is the symplectic area of the $(-1,-1)$ curve in the resolved conifold.  The deformed conifold contains a Lagrangian sphere whose image in the $z$-coordinate projection is the interval $[-1,-q] \subset \C$.  The Lagrangian sphere is mirror to the holomorphic sphere in the resolved conifold.

Now take the Atiyah flop.  The K\"ahler moduli of the resolved conifold is the punctured real line $\R-\{0\}$, consisting of two K\"ahler cones $\R_+$ and $\R_-$ of the resolved conifold and its flop respectively.  $A$ serves as the standard coordinate and flop takes $A \in \R_+$ to $-A \in \R_-$.  Thus the Atiyah flop amounts to switching $A$ to $-A$, or equivalently $q$ to $q^{-1}$.  As a result, the SYZ mirror changes from $\{u_1v_1=z+q, u_2v_2=z+1\}$ to $\{u_1v_1=z+q^{-1}, u_2v_2=z+1\}$.

However the above two manifolds are symplectomorphic to each other, and hence they are just equivalent from the viewpoint of symplectic geometry.  Unlike Atiyah flop in complex geometry, the mirror operation does not produce a new symplectic manifold.  It is not very surprising since symplectic geometry is much softer than complex geometry.

In contrast to complex geometry, the mirror flop is just a symplectomorphism rather than a new symplectic manifold.  First observe that this symplectomorphism is non-trivial (Section \ref{sec:flop-sympl}).

\begin{prop} \label{prop:sympl-intro}
Given a symplectic threefold $(X,\omega)$ and a Lagrangian three-sphere $S \subset X$, we have another symplectic threefold $(X^\dagger,\omega^\dagger)$ with a corresponding Lagrangian three-sphere $S^\dagger \subset X^\dagger$, together with a symplectomorphism $f^{(X,S)}:(X,\omega) \to (X^\dagger,\omega^\dagger)$.  It has the property that $f^{(X^\dagger,S^\dagger)} \circ f^{(X,S)} = \tau_S^{-1}$, where $\tau_S$ is the Dehn twist along the Lagrangian sphere $S$.
\end{prop}

We shall regard $X$ and $X^\dagger$ as the same symplectic manifold using the above symplectomorphism $f^{(X,S)}$.

We need to endow a symplectic threefold with additional geometric structures in order to make it more rigid, so that the effect of the mirror flop can be seen.  In the above local case, $\{u_1v_1=z+q, u_2v_2=z+1\}$ and $\{u_1v_1=z+q^{-1}, u_2v_2=z+1\}$ simply have different complex structures.  However in general requiring the existence of a complex structure on a symplectic manifold would be too restrictive.  Friedman \cite{Friedman} and Tian \cite{Tian} showed that there are topological obstructions to complex smoothing of conifold points; Smith-Thomas-Yau \cite{STY} found the mirror statement for topological obstructions to K\"ahler resolution of conifold points.  

In this paper, we consider two kinds of geometric structures, namely Lagrangian fibrations, and Bridgeland stability conditions on the derived Fukaya category.  First consider a symplectic threefold $X$ equipped with a Lagrangian fibration $\pi: X \to B$.  Let $S \subset X$ be a Lagrangian sphere.  We assume that $\pi$ around $S$ is given by a local model of Lagrangian fibration on the deformed conifold, where $S$ is taken as the vanishing sphere under a conifold degeneration, see Definition \ref{def:con_Lag}.  We call such a fibration to be conifold-like around $S$.  Then we make sense of the mirror flop by doing a local surgery around $S$ and obtain another Lagrangian fibration $\pi^\dagger: X \to B$.  ($X$ and $X^\dagger$ have been identified by the above symplectomorphism $\rho_{X,S}$.)

\begin{thm} \label{thm:fib-intro}
Given a symplectic threefold $(X,\omega)$ with a Lagrangian fibration $\pi:X\to B$ which is conifold-like around a Lagrangian three-sphere $S \subset X$, there exists another Lagrangian fibration $\pi^\dagger:X \to B$ with the following properties.
\begin{enumerate}
\item $\pi^\dagger$ is also conifold-like around $S$.
\item The images of $S$ under $\pi$ and $\pi^\dagger$ are the same, denoted by $\underline{S}$.  They are one-dimensional affine submanifolds in $B$ away from discriminant locus.
\item $\pi^\dagger = \pi$ outside a tubular neighborhood of $S$.  In particular the affine structures on $B$ induced from $\pi$ and $\pi^\dagger$ are identical away from a neighborhood of $\underline{S}$.
\item The induced orientations on $\underline{S}$ from $\pi$ and $\pi^\dagger$ are opposite to each other.
\end{enumerate}
\end{thm}

We call the change from $\pi$ to $\pi^\dagger$ to be the A-flop of a Lagrangian fibration along $S$.  As a compact example, consider the Shoen's Calabi-Yau, which admits a conifold-like Lagrangian fibration around certain Lagrangian spheres by the work of Gross \cite{Gross-BB} and Casta\~no-Bernard and Matessi \cite{CM2}.  Then we can apply the A-flop to obtain other Lagrangian fibrations.

More generally we can consider the effect of A-flop along $S$ on Lagrangian submanifolds other than Lagrangian torus fibers.  Given a Lagrangian submanifold $L \subset X$ which has $T^2$-symmetry around $S$ (see Definition \ref{def:T^2-type}), we can construct another Lagrangian submanifold $L^\dagger$ (which also has $T^2$-symmetry around $S$) which we call to be the A-flop of $L$, with the property that $(L^\dagger)^\dagger$ equals to the inverse Dehn twist of $L$ along $S$.

Then we can take A-flop of special Lagrangian submanifolds with respect to a certain holomorphic volume form (if it exists).  Formally we start with a Bridgeland stability condition $(Z,\cS)$  \cite{Bridgeland}  on the derived Fukaya category, where $Z$ is a homomorphism of the K group to $\C$, and $\cS$ is a collection of objects in the derived Fukaya category which are said to be stable.  A stability condition $(Z,\cS)$ is said to be geometric if there exists a holomorphic volume form $\Omega$ such that $Z$ is given by the period $\int_\cdot \Omega$ and $\cS$ is a collection of graded special Lagrangians with respect to $\Omega$.  A-flop should be understood as a change of stability conditions $(Z,\cS) \mapsto (Z^\dagger,\cS^\dagger)$.  

In this paper we realize the above for the local deformed conifold in Section \ref{sec:fordefnc}.  We obtain the following theorem in Section \ref{subsec:bridgelandfuk}.

\begin{thm}\label{thm:aflopstab}
Let $X$ be the deformed conifold $\{u_1v_1=z+q, u_2v_2=z+1, z\not= 0\}$ (where $q\not=1$).  Equip $X$ with the holomorphic volume form $\Omega = d z \wedge d u_1 \wedge d u_2$.  There exists a collection $\cS$ of graded special Lagrangians which defines a geometric stability condition $(Z,\cS)$ on $X$.  Moreover the flop $(Z^\dagger,\cS^\dagger)$ also defines a geometric stability condition with respect to $(f^{(X,S)})^*\Omega_{X^\dagger}$ where $f^{(X,S)}: X \to X^\dagger = \{u_1v_1=z+1, u_2v_2=z+1/q: z\not= 0\}$ is the symplectomorphism in Proposition \ref{prop:sympl-intro}  (and $\Omega_{X^\dagger} = d z \wedge d u_1 \wedge d u_2$ on $X^\dagger$).
\end{thm}

Stability conditions for the derived Fukaya category were constructed for the $A_n$ case by Thomas \cite{Thomas}, for certain local Calabi-Yau threefolds associated to quadratic differentials by Bridgeland-Smith \cite{BS,Smith}, and for punctured Riemann surfaces with quadratic differentials by Haiden-Katzarkov-Kontsevich \cite{HKK}.  In this paper we construct stability conditions on the derived Fukaya category of the deformed conifold by applying the mirror functor construction in \cite{CHL,CHL2}; in the mirror side we use the results of Nagao-Nakajima \cite{NN} about stability conditions on the noncommutative resolved conifold (see Theorem \ref{thm:nncl}).

\begin{thm}[see Theorem \ref{thm:loccatid}]
The mirror construction in \cite{CHL2} applied to the deformed conifold $X$ produces the noncommutative resolved conifold $\cA$ given by Equation \eqref{eq:mir-A}.  In particular, there is a natural equivalence of triangulated categories
\begin{equation}\label{eqn:mainequiv}
\Psi: D^b \mathcal{F} \to D_{\rm nil}^b \mathrm{mod} \, \cA
 \end{equation}
where $\mathcal{F}$ is a subcategory of $\Fuk(X)$ generated by Lagrangians spheres, and $D_{\rm nil}^b \mathrm{mod} \, \cA$ is a subcategory of $D^b \mathrm{mod}\, \cA$ consisting of modules with nilpotent cohomology.
\end{thm}

The relation between the mirror construction in \cite{CHL2} and the SYZ construction is summarized in Figure \ref{fig:con-MS-mod}.  The SYZ construction uses Lagrangian torus fibration coming from degeneration to the large complex structure limit.  The noncommutative mirror construction in \cite{CHL2} uses Lagrangian vanishing spheres coming from degeneration to the conifold point.

\begin{figure}[h]
\begin{center}
\includegraphics[scale=0.5]{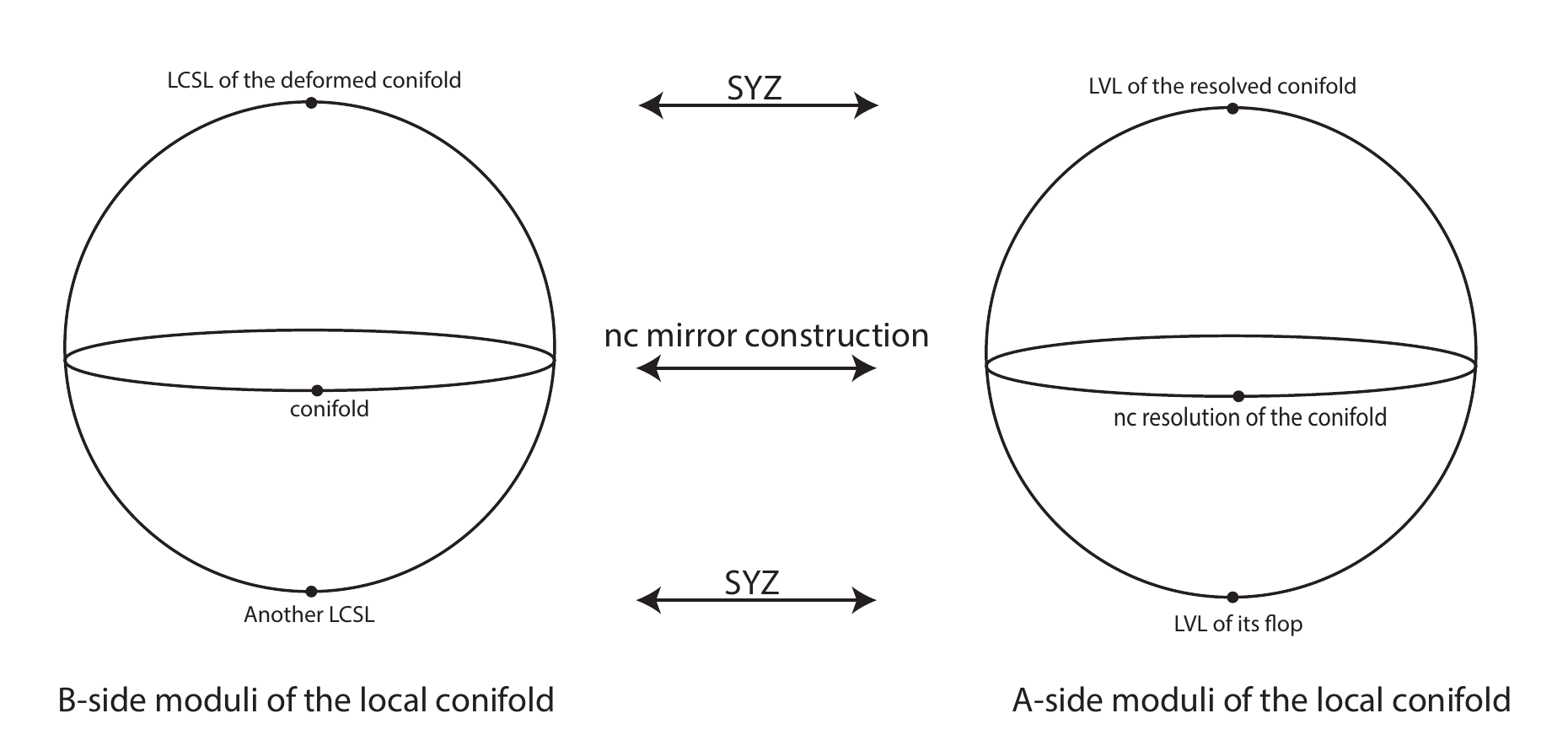}
\caption{The local conifold is self-mirror.  More precisely for the local conifold, a resolution and its flop are just equivalent, and so the upper hemisphere should be identified with the lower hemisphere.}
\label{fig:con-MS-mod}
\end{center}
\end{figure}

We shall prove that stable modules in $D_{\rm nil}^b \mathrm{mod} \, \cA$ with respect to a certain stability condition can be obtained as transformations of special Lagrangians under \eqref{eqn:mainequiv}; as a result the corresponding stability condition on $D^b \mathcal{F}$ is geometric.

\subsection*{Acknowledgement}
We express our gratitude to Matthew Young for drawing our attention to the work of Nagao and Nakajima,
and Yukinobu~Toda for helpful discussions.
S.-C. Lau expresses his gratitude for the AMS-Simons Travel Grant.
The work of H. Hong and S.-T. Yau is substantially supported by Simons Collaboration Grant on Homological Mirror Symmetry.

\section{Review on flops and Bridgeland stability conditions}\label{sec:bsiderev}

In this section, we recall the results by Toda which relate flops with wall-crossings in the space of Bridgeland stability conditions on certain triangulated categories.  For more details and proofs, see \cite{Toda}.

\subsection{Bridgeland stability conditions and crepant small resolutions}

Let $f:\hat{Y}\rightarrow Y$ be a crepant small resolution in dimension three and $C$ the exceptional locus, which is a tree of rational curves $C=C_1\cup\cdots\cup C_N$.

Define the triangulated subcategory $\bD_{\hat{Y}/Y}\subset D^b(\hat{Y})$ to be
\begin{equation}
\bD_{\hat{Y}/Y}:=\{E\in D^b(\hat{Y})\ |\ \mathrm{Supp}(E)\subset C\}.
\label{eq:D/Y}
\end{equation}

Let $^p\Per(\hat{Y}/Y)\subset D^b(\hat{Y})$ ($p=0,-1$) be the abelian categories of perverse coherent sheaves introduced by Bridgeland \cite{Bri}, and 
$$
^p\Per(\bD_{\hat{Y}/Y}):=^p\Per(\hat{Y}/Y)\cap\bD_{\hat{Y}/Y}.
$$

\begin{prop}[\cite{VdB}] The abelian categories $^0\Per(\bD_{\hat{Y}/Y})$ and $^{-1}\Per(\bD_{\hat{Y}/Y})$ are the hearts of certain bounded t-structures on $\bD_{\hat{Y}/Y}$, and are finite-length abelian categories.
The simple objects in $^0\Per(\bD_{\hat{Y}/Y})$ and $^{-1}\Per(\bD_{\hat{Y}/Y})$ are $\{\omega_C[1],\cO_{C_1}(-1),\ldots\\,\cO_{C_N}(-1)\}$ and $\{\cO_C,\cO_{C_1}(-1)[1],\ldots,\cO_{C_N}(-1)[1]\}$ respectively.
\end{prop}

\begin{thm}[\cite{Bri}\cite{Chen-flop}] Let $g:\hat{Y}^{\dagger}\rightarrow Y$ be the flop of $f$, and $\phi:\hat{Y}\dashrightarrow \hat{Y}^{\dagger}$ be the canonical birational map. Then the Fourier-Mukai functor with the kernel $\cO_{\hat{Y}\times_Y \hat{Y}^{\dagger}}\in D^b(\hat{Y}\times \hat{Y}^{\dagger})$ is an equivalence
$$
\Phi_{\hat{Y}\rightarrow \hat{Y}^{\dagger}}^{\cO_{\hat{Y}\times_Y \hat{Y}^{\dagger}}}:D^b(\hat{Y})\overset{\cong}\longrightarrow D^b(\hat{Y}^{\dagger}).
$$
This equivalence restricts to an equivalence $\bD_{\hat{Y}/Y} \overset{\cong}\longrightarrow \bD_{\hat{Y}^\dagger/Y}$ and takes $^0\Per(\hat{Y}/Y)$ to $^{-1}\Per(\hat{Y}^{\dagger}/Y)$.
\end{thm}

Such an equivalence is called \emph{standard} in \cite{Toda}.

Let $\mathrm{FM}(\hat{Y})$ be the set of pairs $(W,\Phi)$, where $W\rightarrow Y$ is a crepant small resolution, and $\Phi:D^b(W)\rightarrow D^b(\hat{Y})$ can be factorized into standard equivalences and the auto-equivalences given by tensoring line bundles. For each $(W,\Phi)\in\mathrm{FM}(\hat{Y})$, there is an associated open subset
$$
U(W,\Phi)\subset\Stab_n(\hat{Y}/Y)
$$
of the space of \emph{normalized} Bridgeland stability conditions on $\bD_{\hat{Y}/Y}$.
A Bridgeland stability condition on $\bD_{\hat{Y}/Y}$ is called \emph{normalized} if the central charge $Z([\cO_x])$ of the skyscraper sheaf at each $x\in C$ is $-1$.

Assume in addition that there is a hyperplane section in $Y$ containing the singular point such that its pullback in $\hat{Y}$ is a smooth surface, Toda proved the following theorem.

\begin{thm}
\cite{Toda} Let $\Stab_n^{\circ}(\hat{Y}/Y)$ be the connected component of $\Stab_n(\hat{Y}/Y)$ containing the standard region $U\left(\hat{Y},\Phi=\mathrm{id}_{D^b(\hat{Y})}\right)$. Define the following union of chambers
$$
\mathcal{M}:=\bigcup_{(W,\Phi)\in\mathrm{FM}(\hat{Y})}U(W,\Phi).
$$
Then $\mathcal{M}\subset\Stab_n^{\circ}(\hat{Y}/Y)$, and any two chambers are either disjoint or equal. Moreover, $\overline{\mathcal{M}}=\Stab_n^{\circ}(\hat{Y}/Y)$.
\end{thm}

In other words, we can obtain the whole connected component $\Stab_n^{\circ}(\hat{Y}/Y)$ from the standard region $U(\hat{Y},\mathrm{id})$ by sequence of flops and tensoring line bundles.

\subsection{The conifold}\label{subsec:exatiyah}

Let $Y=\mathrm{Spec}\ \C[[x,y,z,w]]/(xy-zw)$ and $f:\hat{Y}\rightarrow Y$ be the blowing up at the ideal $(x,z)$. As computed in \cite{Toda},
$$
\Stab_n^{\circ}(\hat{Y}/Y)/\mathrm{Aut}^0(\bD_{\hat{Y}/Y})\cong\mathbb{P}^1-\{\ 3\ \textrm{points}\ \}.
$$

Let $\hat{Y}^{\dagger}\rightarrow Y$  be the blowing up at the other ideal $(x,w)$. Then the three removed points correspond to the large volume limit points of $\hat{Y}$ and $\hat{Y}^{\dagger}$, and the conifold point.

More precisely, $\mathbb{P}^1-\{\ 3\ \mathrm{points}\ \}$ is obtained by gluing the upper and lower half complex planes $\quat,\quat^{\dagger}$, and the real line with the origin removed.
The hearts of the Bridgeland stability conditions in $\quat$ and $\quat^{\dagger}$ are given by $\mathrm{Coh}_{\hat{Y}/Y}$ and $\mathrm{Coh}_{\hat{Y}^{\dagger}/Y}$
respectively.
The heart of the Bridgeland stability conditions on the real line is given by the perverse heart $^0\mathrm{Per}(\bD_{\hat{Y}/Y})\cong^{-1}\mathrm{Per}(\bD_{\hat{Y}^{\dagger}/Y})$.

Let $C,C^{\dagger}$ be the exceptional curves of $\hat{Y}\rightarrow Y,\hat{Y}^{\dagger}\rightarrow Y$ respectively.
Then the equivalence $\bD_{\hat{Y}/Y}\longrightarrow \bD_{\hat{Y}^\dagger/Y}$ satisfies
\begin{enumerate}
\item $\Phi(\cO_C(-1))=\cO_{C^{\dagger}}(-1)[1]$.
\item $\Phi(\cO_C(-2)[1])=\cO_{C^{\dagger}}$.
\item For $x\in C$, the cohomology of $E:=\Phi(\cO_x)\in\bD_{\hat{Y}^{\dagger}/Y}$ vanish except for $H^0(E)=\cO_{C^{\dagger}}$ and $H^{-1}(E)=\cO_{C^{\dagger}}(-1)$.
\end{enumerate}

One can observe the following wall-crossing phenomenon: the skyscraper sheaves $\cO_x\in\bD_{\hat{Y}/Y}$ are stable objects with respect to the stability conditions on the upper half plane $\quat$, but are unstable in $\quat^{\dagger}$. In fact, its image under $\Phi$ is a two term complex $E$ that fits into the following exact triangle:
\begin{equation}\label{eqn:pervpt}
\mathcal{O}_{C^\dagger}(-1)[1] \to E  \to \mathcal{O}_{C^\dagger} \stackrel{[1]}{\to}
\end{equation}
Note that the usual skyscraper sheaf at a point in $C^\dagger$ can be obtained by switching the first and the third terms in \eqref{eqn:pervpt}.

\begin{remark}
It is well-known that if $C$ is a $(-1,-1)$-curve, then the `flop-flop' functor is the same as the inverse of the spherical twist by $\cO_C(-1)$, i.e. $\Phi_{\hat{Y}^{\dagger}\rightarrow \hat{Y}}^{\cO_{\hat{Y}\times_Y \hat{Y}^{\dagger}}}\circ\Phi_{\hat{Y}\rightarrow \hat{Y}^{\dagger}}^{\cO_{\hat{Y}\times_Y \hat{Y}^{\dagger}}}=\mathrm{T}_{\cO_C(-1)}^{-1}$. Proposition \ref{prop:sympl-intro} is the mirror statement of this fact.
\end{remark}

\section{Review on the SYZ mirror of the conifold}\label{sec:revsyz}

SYZ mirror construction for toric Calabi-Yau manifolds was carried out in \cite{CLL} using the wall-crossing techniques of \cite{auroux07}.  The reverse direction, namely SYZ construction for blow-up of $V \times \C$ along a hypersurface in a toric variety $V$ was carried out by \cite{AAK}.  In this section we recall the construction for the conifold $Y = \{(u_1,v_1,u_2,v_2) \in \C^4: u_1v_1 = u_2v_2\}$ as a special case in \cite{CLL,AAK}.  The statement is that $Y - \{u_2v_2=1\}$ is mirror to $Y - (\{u_2=1\} \cup \{v_2=1\})$.  The study motivates the definition of A-flop for Lagrangian fibrations in the next section.

The resolved conifold $\hat{Y} = \mathcal{O}_{\bP^1}(-1) \oplus \mathcal{O}_{\bP^1}(-1)$ is obtained from a small blowing-up of the conifold point $(u_1,v_1,u_2,v_2) = 0$.  It is a toric manifold equipped with a toric K\"ahler form.  We have the $T^2$-action on $\hat{Y}$ given by $(\lambda_1,\lambda_2) \cdot (u_1,v_1,u_2,v_2) = (\lambda_1 u_1, \lambda_1^{-1} v_1, \lambda_2 u_2, \lambda_2^{-1} v_2)$, and we denote the corresponding moment map by $(\mu_1,\mu_2):\hat{Y} \to \R^2$.  Then from the works of Ruan \cite{Ruan-fib}, Gross \cite{Gross-eg} and Goldstein \cite{Goldstein}, there is a Lagrangian fibration
$$ (\mu_1,\mu_2,|zw - 1|): \hat{Y} \to \R^2 \times \R_{\geq 0}.$$
It serves as one of the local models of Lagrangian fibrations which were used by Casta\~no-Bernard and Matessi \cite{CM1,CM2} to build up global fibrations from a tropical base manifold.

\begin{figure}[h]
\begin{center}
\includegraphics[scale=0.3]{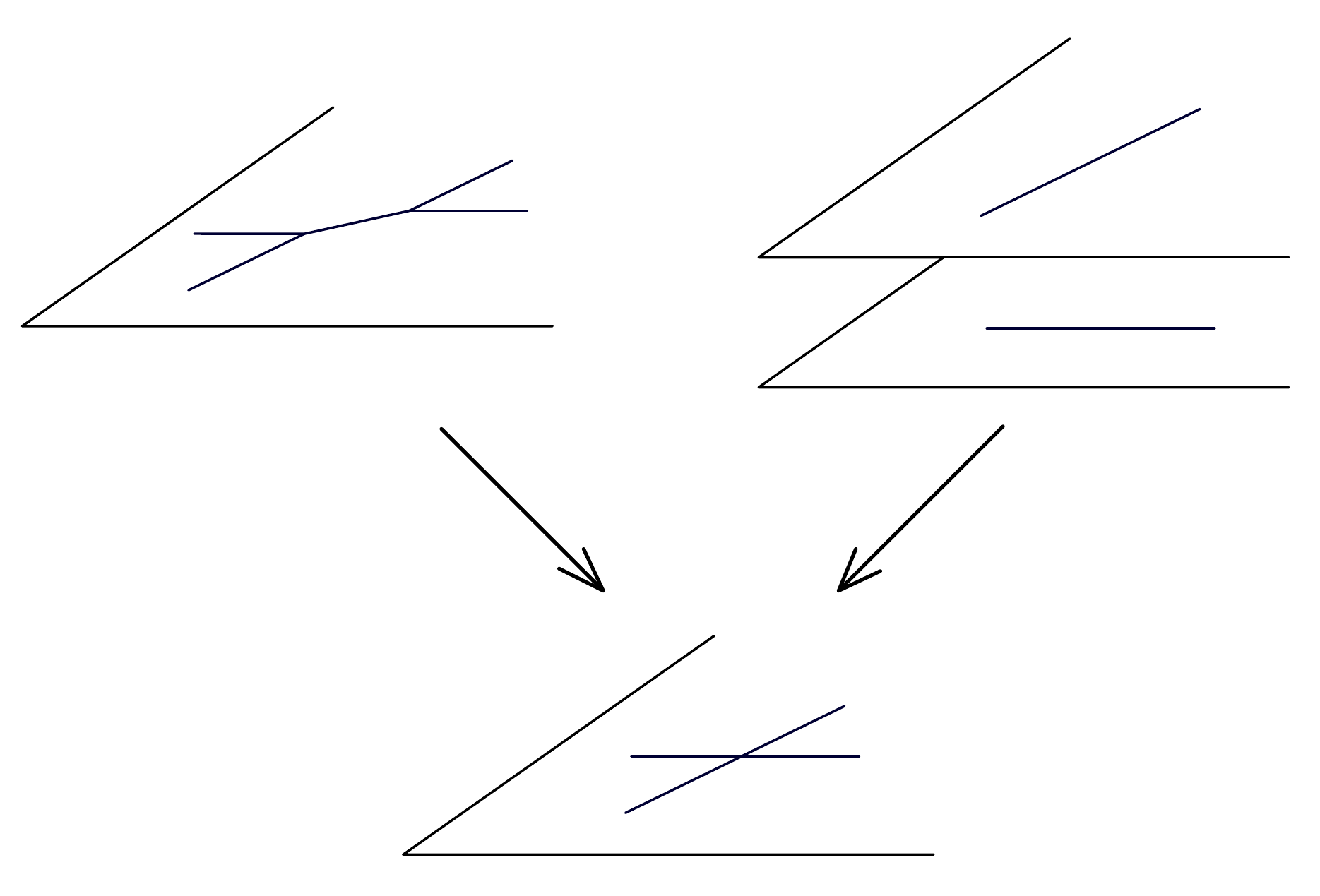}
\caption{The base and discriminant loci of Lagrangian fibrations in conifold transition.}
\label{fig:conifold}
\end{center}
\end{figure}

The discriminant locus of this fibration is contained in the hyperplane 
$$\{(x_1,x_2,x_3) \in \R^2 \times \R_{\geq 0}: x_3=1\},$$ 
see the top left of Figure \ref{fig:conifold}.  This hyperplane is known as the wall for open Gromov-Witten invariants of torus fibers as it contains images of holomorphic discs of Maslov index zero.  By studying wall-crossing of holomorphic discs emanated from infinity divisors (of a compactification of $\hat{Y}$), \cite{CLL} constructed the SYZ mirror of $\hat{Y} - \{zw = 1\}$.

\begin{thm}[A special case in \cite{CLL} and \cite{AAK}]
The SYZ mirror of $\hat{Y} - \{u_2v_2 = 1\}$ is
$$ \{(u_1,v_1,u_2,v_2): u_1,v_1 \in \C, u_2, v_2 \in \C^\times: u_1v_1 = 1+u_2+v_2+qu_2v_2\} $$
where $q = \exp - \big(\textrm{complexified symplectic area of the zero section } \bP^1 \textrm{ of } \hat{Y}\big)$.
\end{thm}
Take the change of coordinates $\tilde{u}_2 = q^{1/2}u_2 + 1/q^{1/2}, \tilde{v}_2 = q^{1/2}v_2 + 1/q^{1/2}$.  (Here we have fixed a square root of $q$.)
Then the equation becomes $u_1v_1 = u_2v_2 + 1-1/q$ and the divisors are $u_2=1/q^{1/2}$ and $v_2=1/q^{1/2}$.  Further rescaling $(u_1,v_1,u_2,v_2)$ by $q^{1/2}$, the SYZ mirror is the deformed conifold 
$$\tilde{Y} = \{(u_1,v_1,u_2,v_2) \in \C^4: u_1v_1 = u_2v_2 + (q-1)\}$$ 
with the divisor $\{(u_2-1)(v_2-1)=0\}$ deleted.
To conclude, we have the mirror pair $\hat{Y} - \{u_2v_2 = 1\}$ and $\tilde{Y} - \{(u_2-1)(v_2-1)=0\}$.

Taking the Atiyah flop of the $(-1,-1)$ curve in $\hat{Y}$ amounts to switching $q$ to $1/q$.  As a result, the mirror of $\hat{Y} - \{zw=1\}$ changes from 
$$\{u_1v_1=u_2v_2+(q-1)\}-\{(u_2-1)(v_2-1) = 0\}$$ 
to 
$$\{u_1v_1=u_2v_2+(1/q-1)\}-\{(u_2-1)(v_2-1) = 0\}$$ 
under flop on $\hat{Y}$.  However changing equation just results in a symplectomorphism.  Thus unlike the flop of a $(-1,-1)$ curve, the mirror flop (of a Lagrangian vanishing sphere in conifold degeneration) does `nothing' to the symplectic manifold.  We need additional geometric structures to detect the mirror flop.  For this local model it is obvious that they can be distinguished by complex structures.  In general we would like to consider geometric structures in the symplectic category.  This will be further studied in the next section.

We can also consider a different relative Calabi-Yau so that Lagrangian spheres can be seen more easily.  First rescale $(u_1,v_1,u_2,v_2)$ so that $\tilde{Y}$ is given as 
$$u_1v_1 - u_2v_2 = q^{1/2} - q^{-1/2}.$$  
Rewrite $\tilde{Y}$ as a double conic fibration,
$$\tilde{Y} = \{(u_1,v_1,u_2,v_2,z) \in \C^5: u_1v_1 = z + q^{1/2}, u_2v_2 = z + q^{-1/2}\}.$$ 
It is equipped with the standard symplectic form from $\C^5$.  If we flop $\hat{Y}$, the mirror $\tilde{Y}$ becomes $\{u_2v_2=z+q^{-1/2}; u_1v_1=z+q^{1/2}\}$.  We take the complement $\tilde{Y} - \{z = c\}$ where $c \in \C - \{-q^{1/2},-q^{-1/2}\}$,   

We have the Lagrangian fibration 
$$
(x_1,x_2,x_3) = (|u_1|^2-|v_1|^2,|u_2|^2-|v_2|^2,|z-c|): \tilde{Y} \to \R^2 \times \R_{\geq 0}
$$
where the boundary divisor is exactly $\{z = c\}$.  The discriminant loci are $\{x_1=0,x_3=|q^{1/2}+c|\}$ and $\{x_2=0,x_3=|q^{-1/2}+c|\}$ contained in the walls $\{x_3=|q^{1/2}+c|\}$ and $\{x_3=|q^{-1/2}+c|\}$ respectively, see the top right of Figure \ref{fig:conifold}.  By \cite[Theorem 11.1]{AAK} (or SYZ in \cite{L13} by Minkowski decompositions), the resulting SYZ mirror is the following.

\begin{thm}[A special case in \cite{AAK} and \cite{L13}]
The SYZ mirror of $\tilde{Y} - \{z = c\}$ is $\hat{Y} - (\{u_2=1\} \cup \{v_2=1\})$.
\end{thm}


Denote $a=-q^{-1/2}$ and $b=-q^{1/2}$, and without loss of generality assume that $a,b$ are real, $c=0$ and $a < b <0$.  Consider the Fukaya category of $\tilde{Y} - \{z=0\}$ generated by the two Lagrangian spheres $S_1$ and $S_2$, where 
\begin{align*} 
S_0 =& \{z=-t, |u_1|=|v_1|,|u_2|=|v_2|: a \leq t \leq b\}, \\
S_1 =& \{z=\exp (t \zeta_1 + (1-t) \zeta_0) \textrm{ for } t \in [0,1], |u_1|=|v_1|,|u_2|=|v_2|\}
\end{align*}
where $\zeta_0 = \log |a| - \pi \consti$ and $\zeta_1 = \log |b| + \pi \consti$.  $S_0$ and $S_1$ are oriented by $dt\wedge d\theta_1\wedge d\theta_2$ where $\theta_1,\theta_2$ are the arguments of $u_1,u_2$ respectively.  They are special Lagrangians and in particular graded by a suitable holomorphic volume form.  (We shall go back to this point in more detail in Section \ref{sec:locmod}.)  Figure \ref{fig:S3} shows $S_0$ in the picture of double conic fibration.

\begin{figure}[h]
\begin{center}
\includegraphics[scale=0.6]{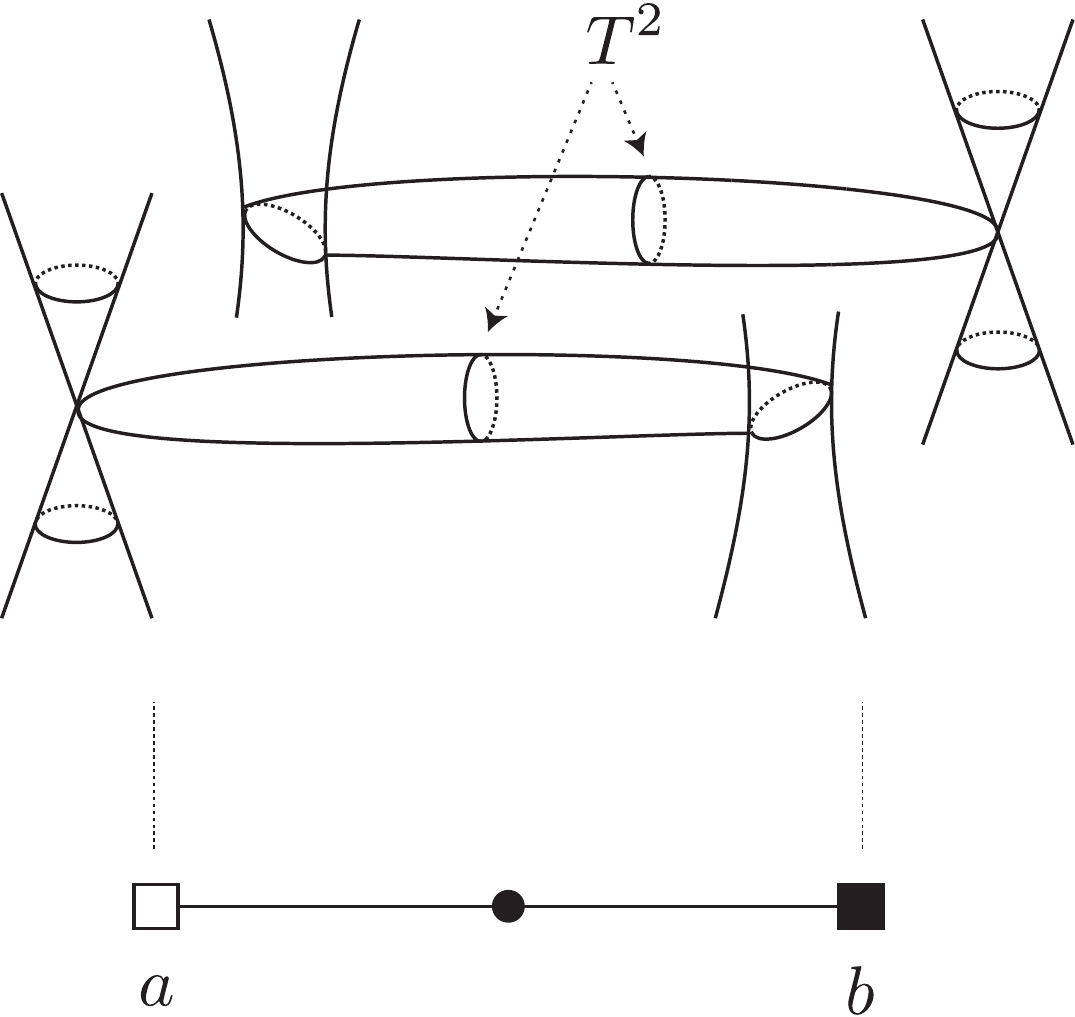}
\caption{Lagrangian $S^3$ seen from the double conic fibration.}
\label{fig:S3}
\end{center}
\end{figure}

Chan-Pomerleano-Ueda \cite{CPU} proved homological mirror symmetry for the mirror pair $(\tilde{Y} - \{z = 0\}, \hat{Y} - (\{u_2=1\} \cup \{v_2=1\})$ making use of the SYZ transformation.  
The result is the following.

\begin{thm}[Theorem 1.2 and 1.3 of \cite{CPU}]\label{thm:citemaincpu}
There is an equivalence between the derived wrapped Fukaya category of $\tilde{Y}_0 := \tilde{Y} - \{z = 0\}$ and the derived category of coherent sheaves of $\hat{Y}_0 := \hat{Y}- (\{u_2=1\} \cup \{v_2=1\})$.
\end{thm}

\begin{remark}
The spheres $S_0$ and $S_1$ here were denoted as $S_1$ and $S_0$ in \cite{CPU} respectively.
\end{remark}

In Section \ref{sec:locmod} and \ref{sec:fordefnc}, $\tilde{Y}_0$ will be denoted as $X_{t=0}$ which appears as a member in a family of symplectic manifolds $X_t$.

Restricting to the Fukaya subcategory consisting of $S_0, S_1$, we have the equivalence between $D^b \langle S_0, S_1 \rangle$ and $\mathcal{D}_{\hat{Y}/Y}$ \eqref{eq:D/Y}.  We will revisit this equivalence in Section \ref{sec:locmod} (see Theorem \ref{thm:cpumod} for more details on the equivalence).  Then we will compare the flop on B-side and the corresponding operation on A-side (to be constructed below) using this.

On the other hand, we can take the approach of \cite{CHL2} to construct the noncommutative mirror of $\tilde{Y}_0$.  From homological mirror symmetry between $\tilde{Y}_0$ and its noncommutative mirror, we obtain stability conditions on the derived Fukaya category generated by $S_0$ and $S_1$ in Section \ref{sec:fordefnc}.  We will show that stable objects are special Lagrangian submanifolds.

\section{A-flop in symplectic geometry}

\subsection{Mirror of Atiyah flop as a symplectomorphism} \label{sec:flop-sympl}

Let $(X,\omega)$ be a symplectic threefold and $S$ a Lagrangian sphere of $X$.  By Weinstein neighborhood theorem, a neighborhood of $S \subset X$ can always be identified symplectomorphically with a neighborhood of $S \subset T^*S$, which can be identified with $\{(u_1,v_1,u_2,v_2)\in \C^4: u_1v_1-u_2v_2=\epsilon\}=\{(u_1,v_1,u_2,v_2,z)\in \C^5: u_1v_1=z+\epsilon, u_2v_2=z\}$ for some $\epsilon>0$, where $\omega$ is given by the restriction of the standard symplectic form on $\C^4$.  This identification is adapted to conifold degeneration at the limit $\epsilon \to 0$.

In other words, we take a \emph{conifold-like chart} in the following sense.

\begin{defn} \label{def:con-chart}
A conifold-like chart around $S$ is $(U,\iota)$, where $U$ is an open neighborhood of $S$ and $\iota: U \hookrightarrow \C^\times \times \C^4$ is a symplectic embedding (where $\C^\times \times \C^4$ is equipped with the standard symplectic form) such that the following holds.
\begin{enumerate}
\item The image of $U$ under the embedding is given by
\begin{equation}\label{eqn:doubleconic}
\left\{
\begin{array}{l}
u_1 v_1 = z-a, \\
u_2 v_2 = z-b
\end{array}\right.
\end{equation}
for some real numbers $a<b$, where $\left|z-\frac{a+b}{2}\right| < R$ for some fixed $R>\frac{b-a}{2}$, and $\left||u_1|^2- |v_1|^2\right| < L$, $\left||u_2|^2- |v_2|^2\right| < L$ for some fixed $L>0$.
Here $z$ is the coordinate of $\C^\times$ and $u_1,v_1,u_2,v_2$ are the coordinates of $\C^4$.  We will also denote the image by $U$ for simplicity.
\item The Lagrangian sphere $\iota(S)$ is given by $\left\{|u_1|=|v_1|,|u_2|=|v_2|,z \in [a,b]\right\}$.
\end{enumerate}
\end{defn}

We will simply identify $U$ with its image under $\iota$.  Let
\begin{align*}
V &= \left\{u_1 v_1 = z-a, u_2 v_2 = z-b, \left|z-\frac{a+b}{2}\right| \leq R-\epsilon, \left| |u_i|^2 - |v_i|^2 \right| \leq L-\epsilon \textrm{ for } i=1,2\right\} \subset U,\\
V' &= \left\{u_1 v_1 = z-a, u_2 v_2 = z-b, \left|z-\frac{a+b}{2}\right| < R-2\epsilon, \left| |u_i|^2 - |v_i|^2 \right| < L-2\epsilon \textrm{ for } i=1,2\right\} \subset V
\end{align*}
for $\epsilon>0$ sufficiently small.  We have a diffeomorphism from $U-V$ to the corresponding open subset of 
$$U^\dagger := \left\{u_1 v_1 = z-b, u_2 v_2 = z-a, \left|z-\frac{a+b}{2}\right| < R, \textrm{ for } i=1,2\right\}$$
defined by $z \mapsto z$, $(u_1,v_1) \mapsto \left(\frac{z-b}{z-a}\right)^{1/2} (u_1,v_1)$, $(u_2,v_2)\mapsto \left(\frac{z-a}{z-b}\right)^{1/2} (u_2,v_2)$.  (We choose a branch of the square root.  It is well-defined since $a,b \notin U-V$.)  By Moser argument, we can cook up a symplectomorphism isotopic to this diffeomorphism.  Thus we have fixed a symplectomorphism $\rho$ from $U-V$ to an open subset of $U^\dagger$.  

In analogous to a flop along a $(-1,-1)$ curve in complex geometry, we define another symplectic threefold $(X^\dagger,\omega^\dagger)$ by gluing $(X - V,\omega)$ with a suitable open subset of $U^\dagger$ by the above symplectomorphism $\rho$ on $U-V$.  By construction (an open subset of) $U^\dagger$ is a conifold chart of $X^\dagger$ around the Lagrangian sphere $S^\dagger \subset U^\dagger$ defined by $\left\{|u_1|=|v_1|,|u_2|=|v_2|,z \in [a,b]\right\}$.  


However, unlike the flop along a $(-1,-1)$ holomorphic sphere, $(X^\dagger,\omega^\dagger)$ is just symplectomorphic to the original $(X,\omega)$, since the gluing map $\rho$ can be extended to $U \to U^\dagger$.  Let
$$\tilde{\psi}_{\pm}: \{(u_1,v_1,u_2,v_2,z) \in \C^5: u_1 v_1 = z-a, u_2 v_2 = z-b\} \to \{u_1 v_1 = z-b, u_2 v_2 = z-a\} $$
be defined by $(u_1,v_1,u_2,v_2,z) \mapsto (\pm\consti u_1,\pm\consti v_1,\pm\consti u_2,\pm\consti v_2,-z+a+b)$ respectively.  It commutes with the $T^2$ action 
$$(\lambda_1,\lambda_2) \cdot (u_1,v_1,u_2,v_2,z) = (\lambda_1 u_1,\lambda_1^{-1} v_1,\lambda_2 u_2,\lambda_2^{-1} v_2,z)$$ 
and hence descends to the symplectic reduction, which is simply rotating the $z$-plane by $\pi$ around $(a+b)/2$.  Let $\psi$ be the restriction of $\tilde{\psi}_+$ to $V'$.  
Then we have a symplectomorphism $U \to U^\dagger$ by interpolating between the gluing map $\rho$ and $\psi$ in the region $R-2\epsilon < \left|z-\frac{a+b}{2}\right| < R-\epsilon , L-2\epsilon < \left| |u_i|^2 - |v_i|^2 \right| < L-\epsilon$.  Namely we take a diffeomorphism which equals to $\psi$ on $V'$, and is given by 
\begin{align*}
z \mapsto& e^{\pi\consti f(|z-(a+b)/2|)} (z-(a+b)/2) + (a+b)/2 \\
(u_1,v_1) \mapsto& \left(\frac{e^{\pi\consti f(|z-(a+b)/2|)} (z-(a+b)/2) + (a-b)/2}{z-a}\right)^{1/2} (u_1,v_1)\\
(u_2,v_2) \mapsto& \left(\frac{e^{\pi\consti f(|z-(a+b)/2|)} (z-(a+b)/2) + (b-a)/2}{z-b}\right)^{1/2} (u_2,v_2)  
\end{align*}
on $U-V$.
Here $f(r)$ is a decreasing function valued in $[0,1]$ which equals to $1$ for $r < R-2\epsilon$ and equals to $0$ for $r > R-\epsilon$.  The square root $z^{1/2}$ is taken for the branch $0<\arg(z)\leq \pi$.  By Moser argument we have a symplectomorphism isotopic to this, and $\rho$ is the restriction to $U-V$.

In conclusion, given a symplectic manifold $(X,\omega)$ and a conifold-like chart around a Lagrangian sphere $S$, we have a symplectomorphism $f^{(X,S)}: (X,\omega) \to (X^{\dagger}, \omega^\dagger)$ by a surgery in analogous to flop in complex geometry.  The operation does not produce a new symplectic manifold because symplectic geometry is too soft.  

If we do the operation twice, we obtain $X^{\dagger\dagger}$ which is canonically identified with $X$ as follows.  $X^{\dagger\dagger}$ is glued from $X-V$ and $U=U^{\dagger\dagger}$ by $\rho^\dagger \circ \rho$.  The composition of $z \mapsto z$, $(u_1,v_1) \mapsto \left(\frac{z-b}{z-a}\right)^{1/2} (u_1,v_1)$, $(u_2,v_2) \mapsto  \left(\frac{z-a}{z-b}\right)^{1/2} (u_2,v_2)$ and $z \mapsto z$, $(u_1,v_1) \mapsto \left(\frac{z-a}{z-b}\right)^{1/2} (u_1,v_1)$, $(u_2,v_2) \mapsto \left(\frac{z-b}{z-a}\right)^{1/2} (u_2,v_2)$ is simply identity.  Hence the gluing $\rho^\dagger \circ \rho = \mathrm{Id}$ and $X^{\dagger\dagger} = X$.
Below we see that doing the above operation twice produces the Dehn twist along the Lagrangian sphere $S$, which induces a non-trivial automorphism on the Fukaya category.

\begin{prop}[same as Proposition \ref{prop:sympl-intro}]
$f^{(X^\dagger,S^\dagger)} \circ f^{(X,S)}: X \to X^{\dagger\dagger}=X$ equals to the inverse of the Dehn twist of $X$ along $S$.
\end{prop}

\begin{proof}

$f^{(X^\dagger,S^\dagger)} \circ f^{(X,S)}:X \to X^{\dagger\dagger}$ is given as follows.  Write $X = X^{\dagger\dagger} = (X-V) \cup_{\mathrm{Id}} U$.
The map is identity on $X-V$.  In $V' \subset U$ it is given by $\psi^2$ which maps $u_i \mapsto -u_i, v_i \mapsto -v_i, z \mapsto z$, and in particular is the antipodal map on the three-sphere 
$$\left\{u_1 v_1 = z-a, u_2 v_2 = z-b, z \in [b,a],|u_i|=|v_i| \textrm{ for } i=1,2\right\} \subset V'.$$  
On $U-V'$ it is isotopic to 
\begin{align*}
z \mapsto& e^{2\pi\consti f(|z-(a+b)/2|)} (z-(a+b)/2) + (a+b)/2\\
(u_1,v_1) \mapsto& \left(\frac{e^{2\pi\consti f(|z-(a+b)/2|)} (z-(a+b)/2) + (b-a)/2}{z-a}\right)^{1/2} (u_1,v_1)\\
(u_2,v_2) \mapsto& \left(\frac{e^{2\pi\consti f(|z-(a+b)/2|)} (z-(a+b)/2) + (a-b)/2}{z-b}\right)^{1/2} (u_2,v_2),
\end{align*}
where $f(r)$ is a decreasing function valued in $[0,1]$ which equals to $1$ for $r < R-2\epsilon$ and equals to $0$ for $r > R-\epsilon$.  The square root $z^{1/2}$ is taken for the branch where $0<\arg(z)\leq 2\pi$.  Thus we see that it is the inverse of the Dehn twist.

\end{proof}

\subsection{Lagrangian fibrations}


We see from the last section that the mirror of the Atiyah flop surgery does not produce a new symplectic manifold unfortunately.  We need additional geometric structures in order to distinguish $X^\dagger$ from $X$.  In this section we consider Lagrangian fibrations.  Conceptually it can be understood as a `real polarization', playing the role of the complex polarization (namely the complex structure) for flop of a $(-1,-1)$ curve.

From now on we identify $X$ and $X^\dagger$ as the same symplectic manifold using the symplectomorphism $f^{(X,S)}$.

Let $\pi:X \to B$ be a Lagrangian torus fibration.  We consider a conifold degeneration of $X$ with a vanishing sphere $S$, such that the Lagrangian fibration around $S$ is like the one on the deformed conifold \cite{Gross-eg,Goldstein}.

\begin{defn} \label{def:con_Lag}
Assume the notations in Section \ref{sec:flop-sympl}.  A Lagrangian fibration $\pi$ is said to be of conifold-like if we have the commutative diagram
$$
\begin{diagram}
	\node{U} \arrow{s,l}{\pi} \arrow{e,t}{\iota} \node{\iota(U)} \arrow{s,r}{\left(|z-c|, \frac{1}{2} (|u_1|^2- |v_1|^2), \frac{1}{2} (|u_2|^2 - |v_2|^2)\right)} \\
	\node{f(U)} \arrow{e,b}{\cong} \node{I \times (-L,L) \times (-L,L)}
\end{diagram}
$$
where $\left|c - \frac{a+b}{2} \right|>R$ and $|c-a| \not= |c-b|$, and $I$ is a certain open interval.
\end{defn}

\begin{thm}[Theorem \ref{thm:fib-intro} in the Introduction] \label{thm:fib}
Given a symplectic threefold $(X,\omega)$ with a Lagrangian fibration $\pi:X\to B$ which is conifold-like around a Lagrangian three-sphere $S \subset X$, there exists a Lagrangian fibration $\pi^\dagger:X \to B$ with the following properties.
\begin{enumerate}
\item $\pi^\dagger$ is also conifold-like around $S$.
\item The images of $S$ under $\pi$ and $\pi^\dagger$ are the same, denoted by $\underline{S}$.  It is a one-dimensional affine submanifold in $B$ away from discriminant locus.
\item $\pi^\dagger = \pi$ outside a neighborhood $V \supset S$ and $V \subset U$.  In particular the affine structures on $B$ induced from $\pi$ and $\pi^\dagger$ are identical away from a neighborhood of $\underline{S}$.
\item From above, there is a canonical correspondence between orientations of regular fibers of $\pi$ and that of $\pi^\dagger$.  Fix an orientation of torus fibers of $\pi$ over the image of $U$, and an orientation of a regular fiber of $\pi|_S$ (which is topologically $T^2$).  Then the induced orientations on $\underline{S}$ through $\pi$ and $\pi^\dagger$ are opposite to each other.
\end{enumerate}
\end{thm}

\begin{proof}
$\pi^\dagger$ is constructed from the symplectomorphism $f:X \to X^\dagger$ given in the last subsection.  Namely we glue the Lagrangian fibration of $X-V$ with the Lagrangian fibration of $U^\dagger$ by $\rho$.  This gives a Lagrangian fibration on $X^\dagger$, and hence on $X$ by the symplectomorphism $f$.  It is constructed directly as follows.

For each fixed $u_1,u_2,v_1,v_2$, take the diffeomorphism $\phi_{u_1,u_2,v_1,v_2}$ on $\{z \in \C: |z-(a+b)/2| < R\}$ defined by 
\begin{equation}
\phi_{u_1,u_2,v_1,v_2}(z)=e^{\pi\consti f\left(\left(\frac{|z-(a+b)/2|}{R}\right)^2+\left(\frac{|u_1|^2- |v_1|^2}{L}\right)^2+\left(\frac{|u_2|^2- |v_2|^2}{L}\right)^2\right)} (z-(a+b)/2) + (a+b)/2
\label{eq:phi}
\end{equation}
where $f(r)$ is a decreasing function valued in $[0,1]$ which equals to $1$ for $r < 1-2\epsilon$ and equals to $0$ for $r > 1-\epsilon$.  Thus $\phi_{u_1,u_2,v_1,v_2}$ is identity on 
$$\left\{(u_1,v_1,u_2,v_2,z) \in U:\left(\frac{|z-(a+b)/2|}{R}\right)^2+\left(\frac{|u_1|^2- |v_1|^2}{L}\right)^2+\left(\frac{|u_2|^2- |v_2|^2}{L}\right)^2 > 1-\epsilon\right\}.$$ 
Define a fibration $U \to I \times (-L,L) \times (-L,L)$ by 
$$\left(\left|\phi_{u_1,u_2,v_1,v_2}\left(z\right)-c\right|, \frac{1}{2} (|u_1|^2- |v_1|^2), \frac{1}{2} (|u_2|^2 - |v_2|^2)\right).$$  
For $|u_1|=|v_1|,|u_2|=|v_2|$, the resulting level curves of $|\phi(0,0,z)-c|$ are depicted in Figure \ref{fig:level-curves}.  Since the fibration is $T^2$-equivariant and any curve on the plane is Lagrangian, it is a Lagrangian fibration by symplectic reduction.  Moreover it agrees with the original Lagrangian fibration $\pi$ on $U-V$.  Hence we can glue this with the original Lagrangian fibration on $X-V$, and obtain another Lagrangian fibration $\pi^\dagger: X \to B$.

By definition $\pi = \pi^\dagger$ away from $V$.  In the neighborhood defined by $\left(\frac{|z-(a+b)/2|}{R}\right)^2+\left(\frac{|u_1|^2- |v_1|^2}{L}\right)^2+\left(\frac{|u_2|^2- |v_2|^2}{L}\right)^2 < 1-2\epsilon$, the fibration is simply 
$$\left(\left|z-(a+b-c)\right|, \frac{1}{2} (|u_1|^2- |v_1|^2), \frac{1}{2} (|u_2|^2 - |v_2|^2)\right)$$
which is also conifold-like around $S$.  The image of $S$ under either $\pi$ and $\pi^\dagger$ is the interval $[b-c,a-c] \times \{0\} \times \{0\}$.  Since the fibration around $S$ is compatible with the symplectic reduction of the $T^2$-action on $(u_1,v_1,u_2,v_2)$, the second and third coordinates $(b_2,b_3)=\left(\frac{1}{2} (|u_1|^2- |v_1|^2), \frac{1}{2} (|u_2|^2 - |v_2|^2)\right)$ serve as the action coordinates of the base of the Lagrangian fibration.  Thus the image $\underline{S} \subset \{b_2=b_3=0\}$ is an affine submanifold.

Since $\pi = \pi^\dagger$ outside $V \subset U$ and every torus fiber of $\pi$ and $\pi^\dagger$ has non-empty intersection in $X - V$, the orientations can be canonically identified.  We have induced orientations on the fibers of $|z-c|$ and also fibers of $|\phi_{u_1,u_2,v_1,v_2}(z)-c|$.  The induced orientation on $\underline{S}$ from $\pi$ (or $\pi^\dagger$ resp.) is such that $\omega(u,v) > 0$ where $u$ is a tangent vector along the orientation of $\underline{S}$ and $v$ (or $v'$ resp.) is a tangent vector along the orientation of a fiber of $|z-c|$ (or a fiber of $|\phi_{u_1,u_2,v_1,v_2}(z)-c|$ resp.).  It follows that $v = -v'$ (up to scaling by a positive number) and hence the two induced orientations on $\underline{S}$ are opposite to each other.
\end{proof}

To distinguish from the usual notion of flop in complex geometry, we call $\pi^\dagger$ to be the \emph{A-flop} of the Lagrangian fibration $\pi$  (where `A' stands for the 'symplectic side' in mirror symmetry).  It is the mirror operation of Atiyah flop.  



In analogous to foliations, we identify two Lagrangian fibrations if they are related by diffeomorphisms as follows.

\begin{defn}
Two Lagrangian fibrations $\pi_1,\pi_2: X \to B$ are said to be equivalent if there exists a symplectomorphism $\Phi: X \overset{\cong}{\to} X$ and a diffeomorphism $\phi: B \overset{\cong}{\to} B$ such that $\phi \circ \pi_1 = \pi_2 \circ \Phi$.
\end{defn}

The following easily follows from construction.

\begin{prop}
If we make different choices of $(U,\iota)$ and the function $f$ in the proof of Theorem \ref{thm:fib}, the resulting Lagrangian fibrations are equivalent.
\end{prop}



\begin{figure}[htb!]
  \centering
   \begin{subfigure}[b]{0.45\textwidth}
   	\centering
    \includegraphics[scale=0.3]{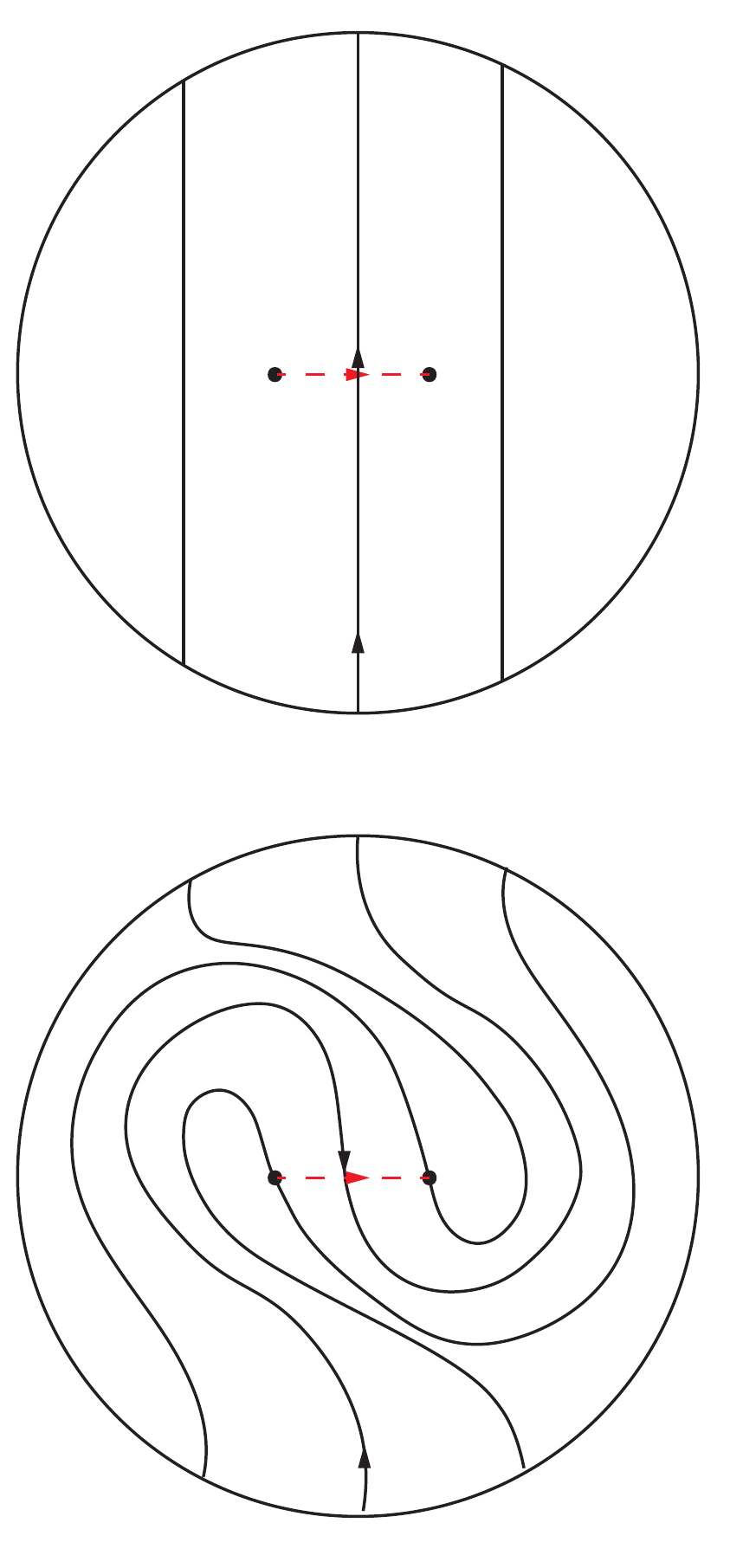}
    \caption{The symplectic reduction of Lagrangian torus fibers before and after the A-flop.}
    \label{fig:level-curves}
   \end{subfigure}
   \hspace{10pt}
   \begin{subfigure}[b]{0.45\textwidth}
   	\centering
    \includegraphics[width=\textwidth]{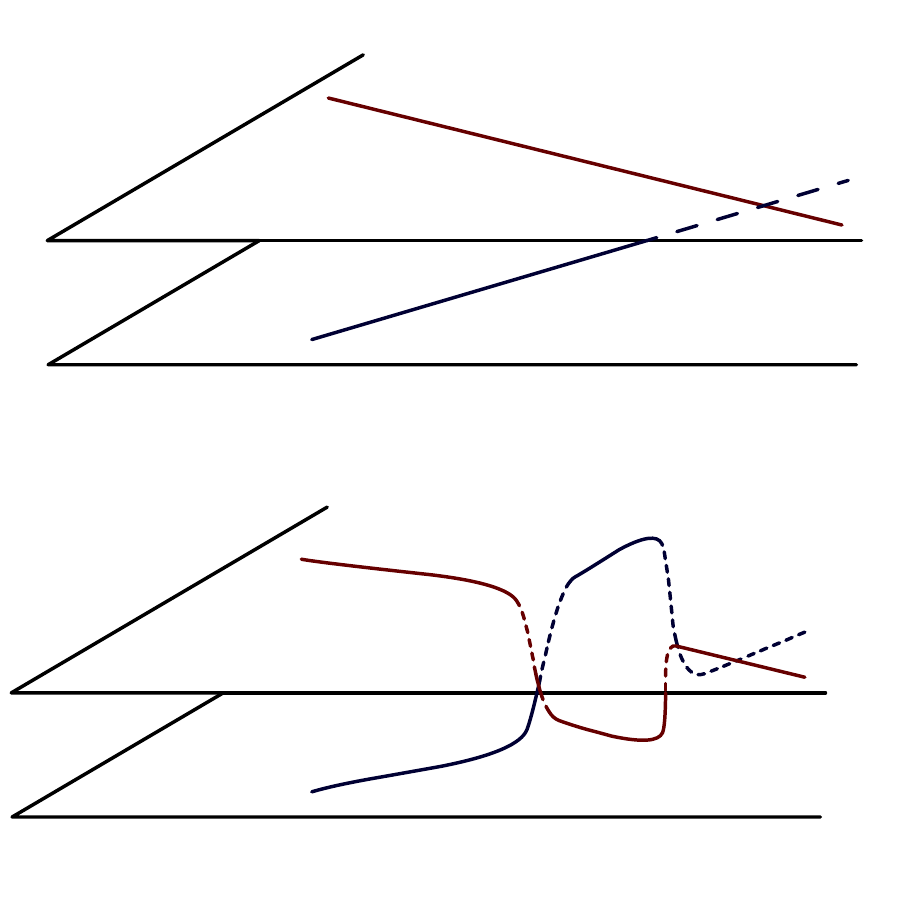}
    \caption{The discriminant locus before and after the A-flop.}
    \label{fig:disc-locus}
   \end{subfigure}
	\label{A-flop of a Lagrangian fibration.}
\end{figure}

\subsection{A-flop on Lagrangian submanifolds}\label{subsec:AflopLag}

We can also consider the effect of A-flop on Lagrangian submanifolds other than torus fibers.  We restrict to the following kind of Lagrangian submanifolds.

\begin{defn} \label{def:T^2-type}
Let $S$ be a Lagrangian sphere in $X$.  A Lagrangian submanifold $L \subset X$ is said to have $T^2$-symmetry around $S$ if there exists a conifold-like chart $(U,(u_1,v_1,u_2,v_2,z))$ around $S$ such that the image of $L \cap U$ under $z$ is a curve and the fiber at each point is given by $|u_i|^2 - |v_i|^2 = c_i(z)$ for some real-valued function $c_i$ on the curve, $i=1,2$.
\end{defn}

Given $L$ with $T^2$-symmetry with respect to a conifold-like chart $(U,(u_1,v_1,u_2,v_2,z))$ around $S$, we define $L^\dagger$ as follows.  Recall the diffeomorphism $\phi_{u_1,u_2,v_1,v_2}$ on $\{z \in \C: |z-(a+b)/2| < R\}$ in Equation \eqref{eq:phi}.  Write the image of $L \cap U$ under $z$ as a level curve $f(z)=0$ of a real-valued function $f$.  Then $L^\dagger$ is given as
$$\left\{(u_1,v_1,u_2,v_2,z) \in U:f(\phi_{u_1,u_2,v_1,v_2}(z))=0, |u_i|^2- |v_i|^2 = c_i(z) \textrm{ for } i=1,2\right\}$$ in $U$ and equals to $L$ outside $U$.  Since the image of $L^\dagger$ is a curve in the symplectic reduction by $T^2$, it is a Lagrangian submanifold with $T^2$-symmetry.  
Note that $L^\dagger$ and $L$ can be topologically different from each other.  

By construction we have
\begin{prop}
$(L^\dagger)^\dagger$ equals to the inverse of the Dehn twist applied to $L$.
\end{prop}

If we have a stability condition $(Z,\mathcal{S})$ on the Fukaya category generated by Lagrangians with $T^2$-symmetry around $S$, then A-flop should give another stability condition $(Z^\dagger,\mathcal{S}^\dagger)$ where $Z^\dagger(L^\dagger) = Z(L)$.  In the next two sections we will restrict to the deformed conifold and carry out this construction explicitly.

\subsection{Examples}
\subsubsection{Deformed conifold} \label{subsubsec:deflocintro}
Consider the deformed conifold $X=\{(u_1,v_1,u_2,v_2,z) \in \C^4 \times \C^\times: u_1 v_1 = z-a, u_2 v_2 = z-b\}$ where $a<b<0$.  (We have taken away the divisor $z=0$.)  We can take the flop of the Lagrangian fibration 
$$\pi=(|u_1|^2-|v_1|^2,|u_2|^2-|v_2|^2,|z|)$$ 
which is special with respect to the holomorphic volume form $\Omega=d\log z \wedge d u_1 \wedge d u_2$.  The base of the fibration is $\R^2\times\R_{>0}$.  The discriminant locus of the fibration is $\{0\} \times \R \times \{|a|\} \cup \R \times \{0\} \times \{|b|\}$.  After the flop, the discriminant locus becomes $\{(0,t,|\phi_{0,t,0,0}(a)|):t\in\R\} \cup \{(t,0,|\phi_{0,t,0,0}(b)|):t\in\R\}$ where $\phi$ is given in Equation \eqref{eq:phi}.  For $t \ll 1$, $\phi_{0,t,0,0}(a)=b$ and $\phi_{0,t,0,0}(b)=a$; for $t$ big enough, $\phi_{0,t,0,0}(a)=a$ and $\phi_{0,t,0,0}(b)=b$.  The base and discriminant locus are shown in Figure \ref{fig:disc-locus}.  The new fibration $\pi^\dagger$ is no longer special with respect to $\Omega$; however it is equivalent to the corresponding special Lagrangian fibration on $X^\dagger=\{(u_1,v_1,u_2,v_2,z) \in \C^4 \times \C^\times: u_1 v_1 = z-b, u_2 v_2 = z-a\}$ with respect to $\Omega^\dagger$ (defined by the same expression as $\Omega$).

We have a family of complex manifolds defined by 
\begin{align*}
u_1 v_1 =& z-\left(\frac{a+b}{2}+\frac{b-a}{2}e^{\pi \consti (1+s)}\right),\\
u_2 v_2 =& z-\left(\frac{a+b}{2}+\frac{b-a}{2}e^{\pi \consti s}\right)\\
\end{align*}
for $s\in [0,1]$ joining $X$ and $X^\dagger$.  They are depicted in Figure \ref{fig:flop_rot}.  Each member has a special Lagrangian fibration defined by the same formula as $\pi$ above.  Before (or after) the moment $s=\frac{1}{2}$, the Lagrangian fibrations are all equivalent.  Approaching the moment $s=\frac{1}{2}$, two singular Lagrangian fibers collide into one and the Lagrangian fibration changes.  Thus $s=\frac{1}{2}$ is the `wall'.  It can also been seen clearly from the base, see the top right of Figure \ref{fig:conifold}.  At the moment $s=\frac{1}{2}$, the two discriminant loci (which are two lines in different directions) collides.

\begin{figure}[h]
\begin{center}
\includegraphics[scale=0.6]{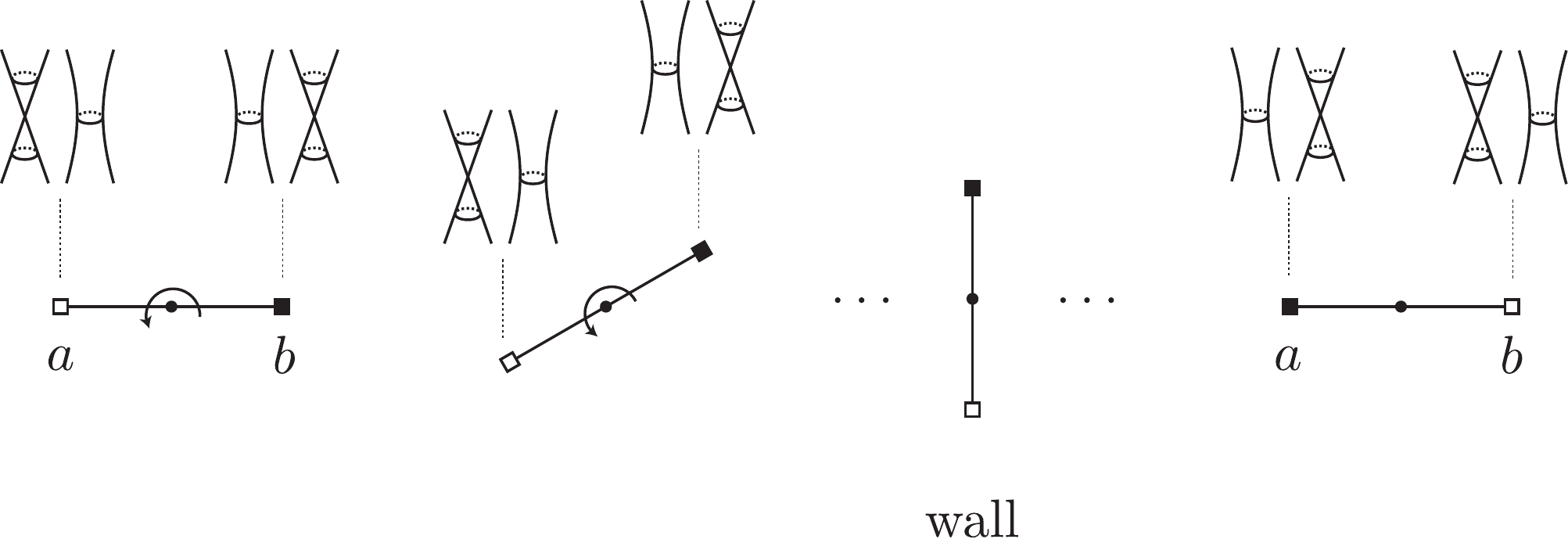}
\caption{A-flop shown in the double conic fibration picture}
\label{fig:flop_rot}
\end{center}
\end{figure}

The torus fibers of $\pi$ and $\pi^\dagger$ are different objects in the Fukaya category.  Namely consider a fiber $T=\{|z|=k,|u_1|=|v_1|,|u_2|=|v_2|\}$ of $\pi$ where $|b|<k<|a|$ and the corresponding fiber $T^\dagger=\{|\phi(z)|=k,|u_1|=|v_1|,|u_2|=|v_2|\}$ of $\pi^\dagger$.  We shall see in Section \ref{sec:locmod} that $T$, which is special Lagrangian with respect to $\Omega$, is a surgery $S_1 \# S_0$ for a morphism in $\mathrm{Mor}(S_1,S_0)$, while $T^\dagger$, which is special Lagrangian with respect to $\Omega^\dagger$, is a surgery $S_0 \# S_1$ for a morphism in $\mathrm{Mor}(S_0,S_1)$.  $S_0,S_1$ are Lagrangian spheres defined by $S_0 = \{z=-t, |u_1|=|v_1|,|u_2|=|v_2|: a \leq t \leq b\}$ and $S_1 = \{z=1+\exp (t \zeta_1 + (1-t) \zeta_0) \textrm{ for } t \in [0,1], |u_1|=|v_1|,|u_2|=|v_2|\}$ where $\zeta_0=\log |a|-\consti\pi$ and $\zeta_1=\log |b|+\consti\pi$.

\subsubsection{Deformed orbifolded conifold}
For $k \ge l \ge 1$, the orbifolded conifold $O_{k,l}$ is the quotient of the conifold $\{u_1v_1=u_2v_2\} \subset \C^4$ by the abelian group $\Z_k \times \Z_l$, where the primitive roots of unity $\zeta_k \in \Z_k$ and $\zeta_l \in \Z_l$ act by
$$
(u_1,v_1,u_2,v_2) \mapsto (\zeta_k u_1,\zeta_k^{-1} v_1,u_2,v_2), \textrm{ and } (x,y,z,w) \mapsto (u_1,v_1,\zeta_l u_2,\zeta_l^{-1} v_2).
$$
In equations
$$
O_{k,l}=\left\{u_1v_1=(z-1)^k, \ u_2v_2=(z-1)^l \right\} \subset \C^5. 
$$ 
It is a toric Gorenstein singularity whose fan is the cone over the rectangle $[0,k] \times [0,l] \subset \R^2$.  For the purpose of constructing a Lagrangian torus fibration with only codimension-two discriminant loci, we shall delete the divisor $\{z=0\} \subset O_{k,l}$ and obtain
$ X_0 = \left\{(u_1,u_2,v_1,v_2,z) \in \C^4 \times \C^\times: u_1v_1=(z-1)^k, \ u_2v_2=(z-1)^l \right\}.$

We shall consider smoothings of $X_0$, which correspond to the Minkowski decompositions of the rectangle $[0,k] \times [0,l]$ into $k$ copies of $[0,1] \times \{0\}$ and $l$ copies of $\{0\} \times [0,1]$ \cite{altmann}.  Explicitly a smoothing is given by
$$
X=\{(u_1,u_2,v_1,v_2,z) \in \C^4 \times \C^\times \ | \  u_1v_1=f(z), \ u_2v_2=g(z)\}$$
where $f(z)$ and $g(z)$ are polynomials of degree $k$ and $l$ respectively, such that the roots $r_i$ and $s_j$ of $f(z)$ and $g(z)$ respectively are pairwise-distinct and non-zero.  For later purpose we shall assume $|r_i|,|s_j|$ are all pairwise distinct.

$X$ admits a double conic fibration $X\rightarrow \C^\times$ by projecting to the $z$-coordinate.  There is also a natural Hamiltonian $T^2$-action on $X$ given by
$(s,t)\cdot (u_1,v_1,u_2,v_2,z):=(su_1,s^{-1}v_1,tu_2,t^{-1}v_2,z)$ for $(s,t) \in T^2 \subset \C^2$.  The symplectic reduction of $X$ by the $T^2$-action is identified with $\C^\times$, the base of the double conic fibration.  Using the construction of Goldstein \cite{Goldstein} and Gross \cite{Gross-eg}, we have the Lagrangian fibration 
\begin{align*}
\pi:X &\rightarrow B:= \R^2 \times (0,\infty)\\
\pi(u_1,v_1,u_2,v_2,z)&=\left(\frac{1}{2}(|u_1|^2-|v_1|^2),\frac{1}{2}(|u_2|^2-|v_2|^2),|z|\right).
\end{align*}
The map to the first two coordinates is the moment map of the Hamiltonian $T^2$-action. 
We denote the coordinates of $B$ by $b = (b_1,b_2,b_3)$.  The discriminant locus is given by the disjoint union of lines
$$\left(\bigcup_{i=1}^k \{b_1=0,b_3=|r_i|\}\right) \cup \left(\bigcup_{j=1}^l \{b_2=0,b_3=|s_j|\}\right) \subset B,$$ 
and the fibers are special Lagrangians in the same phase $\pi/2$ with respect to the volume form $\Omega:=du_1 \wedge du_2 \wedge d\log z$ (\cite[Proposition 3.17]{KL}).

Now let $a=r_1$ and $b=s_1$.  Assume that $|a|\not=|b|$; zero and all other roots $r_i, s_j$ lie outside the disc $|z-(a+b)/2|$.  Let $S_0$ be the Lagrangian matching sphere corresponding to the straight line segment joining $a$ and $b$.  Then the above Lagrangian fibration is conifold-like around $S$.  The flop of this is equivalent to the corresponding Lagrangian fibration on $X^\dagger=\{u_1v_1=f^\dagger(z),u_2v_2=g^\dagger(z)\}$, where $f^\dagger$ and $g^\dagger$ are polynomials with sets of roots $\{s_1,r_2,\ldots,r_k\}$ and $\{r_1,s_2,\ldots,s_l\}$ respectively; and the Lagrangian fibration is 
$$\pi^\dagger(u_1,v_1,u_2,v_2,z)=\left(\frac{1}{2}(|u_1|^2-|v_1|^2),\frac{1}{2}(|u_2|^2-|v_2|^2),|z|\right): X^\dagger \to B.$$  
The Lagrangian fibration $\pi^\dagger$ is no longer special with respect to $\Omega$ on $X$; however it is (equivalent to) a special Lagrangian fibration with respect to $\Omega^\dagger = du_1 \wedge du_2 \wedge d\log z$ on $X^\dagger$.

\subsubsection{Shoen's Calabi-Yau}
Given a compact simple integral affine threefold $B$ with singularities $\Delta$, Casta\~no-Bernard and Matessi \cite{CM1} constructed a symplectic manifold $X$ together with a Lagrangian fibration $X \to B$ inducing the given affine structure.  It is achieved by gluing local models of Lagrangian fibrations around $\Delta$ with the Lagrangian fibration over the affine manifold $B-\Delta$.   In particular their construction can be applied to Shoen's Calabi-Yau \cite{CM2}.  The Lagrangian fibration is conifold-like, and so the mirror flop defined here can be applied.

Shoen's Calabi-Yau is given by the fiber product of two elliptic fibrations on $K3$ surfaces over the base $\bP^1$.  The affine base manifold (which is topologically $\bS^3$) of Shoen's Calabi-Yau was found by Gross \cite[Section 4]{Gross-BB}.  Section 9.2 of \cite{CM2} constructed a conifold degeneration of the affine base forming nine conifold points simultaneously.  

We quickly review their construction here.  Consider the following polyhedral decomposition of $\bS^3$.  Take six copies of triangular prisms 
$$\mathrm{Conv} \{(0,0,0),(0,1,0),(1,0,0),(0,0,1),(0,1,1),(1,0,1)\},$$
three of them are labeled as $\sigma_j$ and three of them are labeled as $\tau_k$ for $j,k \in \Z_3$.  Take nine copies of cubes 
$$\mathrm{Conv} \{(0,0,0),(0,1,0),(1,0,0),(1,1,0),(0,0,1),(1,0,1),(0,1,1),(1,1,1)\}$$ 
and label them as $\omega_{jk}$.  See Figure \ref{fig:ShoenCY-polytope}.  The top triangular face of $\sigma_j$ is glued to the bottom triangular face of $\sigma_{j+1}$ ($j \in \Z_3$), and so topologically $\bigcup_{j=1}^3 \sigma_j$ forms a solid torus.  Similarly do the same thing for $\tau_k$ so that $\bigcup_{k=1}^3 \tau_k$ forms another solid torus.  For the nine cubes, glue the top face of $\omega_{jk}$ with the bottom face of $\omega_{j+1,k}$ for $j \in \Z_3$, and glue the right face of $\omega_{jk}$ with the left face of $\omega_{j,k+1}$ for $k\in\Z_3$.  This topologically forms a two-torus times an interval.  Finally glue the front face of $\omega_{jk}$ with the $j$-th square face of $\tau_k$, and glue the back face of $\omega_{jk}$ with the $k$-th square face of $\sigma_j$.  Here the square faces of $\sigma_j$ and $\tau_k$ are ordered counterclockwisely.  This forms $\bS^3$ as gluing of two solid tori along their boundaries.

\begin{figure}[h]
\begin{center}
\includegraphics[scale=0.5]{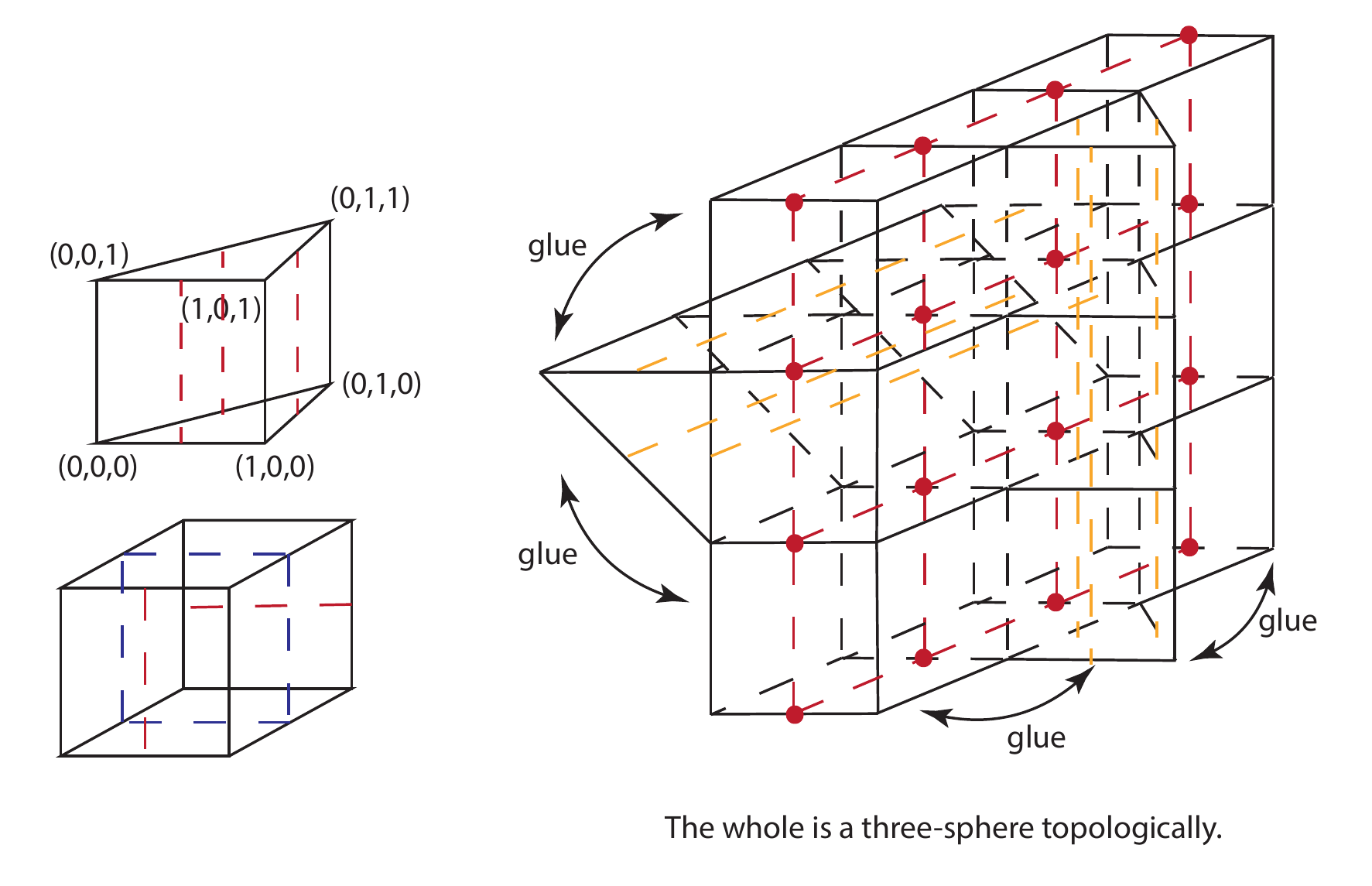}
\caption{Polytopes in the polyhedral decomposition of the affine base of Shoen's CY.}
\label{fig:ShoenCY-polytope}
\end{center}
\end{figure}

The fan structure at every vertex of the polyhedral decomposition is that of $\bP^2 \times \bP^1$.  Together with the standard affine structure of each polytope, this gives $\bS^3$ an affine structure with singularities.  The discriminant locus is given by the dotted lines shown in Figure \ref{fig:ShoenCY-polytope}.  Note that each dotted line in a square face of a prism indeed has multiplicity three.  Thus the discriminant locus is a union of 24 circles counted with multiplicities.  Moreover the dotted lines in cubes form three horizontal and three vertical circles, intersecting with each other at nine points.  These are the nine conifold singularities (which are positive nodes).  

By gluing local models of Lagrangian fibrations around discriminant locus with the Lagrangian fibration from the affine structure away from discriminant locus, \cite{CM2} produced a symplectic manifold which is homeomorphic to the Shoen's Calabi-Yau.  Moreover by using the results on symplectic resolution of Smith-Thomas-Yau \cite{STY} and complex smoothing of Friedman \cite{Friedman} and Tian \cite{Tian}, they showed that the existence of certain tropical two-cycles containing a set of conifold points ensure that the nodes can be simultaneously resolved (and smoothened in the mirror side).  In particular all the nine nodes in this example can be resolved simultaneously.  

In the smoothing the three horizontal and three vertical circles which form part of the discriminant locus are moved apart so that they no longer intersect with each other.  This gives a symplectic manifold $X$ together with a Lagrangian fibration.  The corresponding affine base coincides with the one in the previous work of Gross \cite[Section 4]{Gross-BB}.

The local model for each conifold point in this example is the Lagrangian fibration $(|u_1|^2-|v_1|^2,|u_2|^2-|v_2|^2,|z|)$ on $\{(u_1,v_1,u_2,v_2,z) \in \C^4 \times \C^\times: u_1 v_1 = z-a = u_2 v_2\}$ for $a < 0$; the local model for its smoothing is the fibration defined by the same expression on $\{(u_1,v_1,u_2,v_2,z) \in \C^4 \times \C^\times: u_1 v_1 = z-a, u_2 v_2 = z-b\}$ for $a < b < 0$.  A Lagrangian fibration corresponding to the simultaneous smoothing can be constructed by gluing these local models.  In particular $X$ and the fibration are conifold-like around each of the vanishing spheres corresponding to the nine conifold points.  Hence we can perform A-flop around each of these spheres and obtain new Lagrangian fibrations.  The operation can be understood as link surgery in the base $\bS^3$.

Note that we cannot always keep the circles $A_i,B_j$ in constant levels in the A-flop.  For instance, suppose $A_i$ and $B_j$ are contained in the planes in levels $a_i,b_j$ respectively with $a_1<a_2<a_3<b_1<b_2<b_3$.  (These planes have normal vectors pointing to the right if drawn in Figure \ref{fig:ShoenCY-polytope}.)  Now we perform the A-flop along the vanishing sphere between levels $a_1$ and $b_1$.  The resulting fibration is equivalent to the one with these circles in levels $a_2<a_3<b_1<a_1'<b_2<b_3$ where $a_1'$ is the new level of $A_1$.  At this stage all these circles are still kept in constant levels.  Now let's do the A-flop along the vanishing sphere between levels $a_2$ and $b_3$.  Then the resulting fibration cannot have all these circles in constant levels: if they were in constant levels, then $a_2<b_1<a_1'<b_3<a_2$, a contradiction!


\section{Derived Fukaya category of the deformed conifold}\label{sec:locmod}


In Example \ref{subsubsec:deflocintro}, we consider a path of complex structures on the deformed conifold (with a fixed symplectic form) given by the equations
\begin{equation}\label{eqn:xs01}
X_s = \left\{u_1 v_1 =z-\left(\frac{a+b}{2}+\frac{b-a}{2}e^{\pi \consti (1+s)}\right), u_2 v_2 = z-\left(\frac{a+b}{2}+\frac{b-a}{2}e^{\pi \consti s}\right), z\not=0 \right\}
\end{equation}
for $s \in [0,1]$. ($X_{s=0}$ and $X_{s=1}$ were denoted as $X$ and $X^\dagger$ in \ref{subsubsec:deflocintro}, respectively.)  This deformation of complex structures parametrized by $s$ is SYZ mirror to the flop operation on the resolved conifold.  In the last section we realized this operation as surgery of a Lagrangian fibration.  

In this section, we study the effect of deformation of complex structures (together with holomorphic volume forms) on special Lagrangians. 
This would motivate us to consider A-flop on stability conditions of the derived Fukaya category.


Recall from Section \ref{sec:revsyz} that we have two Lagrangian spheres $S_0$ and $S_1$ in $X_{s=0}$.
Moreover, there is a sequence of Lagrangian spheres $\{S_n : n \in \Z\}$ in $X_{s=0}$ which corresponds to a collection of non-trivial matching paths in the base of the double conic fibration $X_{s=0} \to \C^\times$.
 We depict these spheres in the universal cover of $\C^\times (\ni \!\! z)$ as shown in Figure \ref{fig:spheresss}. 

\begin{defn}
$\mathcal{F}$ is defined to be the full subcategory of $\Fuk(X_{s=0})$ generated by $S_0$ and $S_1$. 
\end{defn}

The main purpose of this section is to prove the following.

\begin{thm}\label{thm:main5cf}
Regular Lagrangian torus fibers of $\pi$ which have non-empty intersection with $S_0$, as well as the Lagrangian spheres $S_i$, are contained in $\mathcal{F}$.
\end{thm}
The theorem follows from Proposition \ref{prop:teqcone} and \ref{prop:smsmt} below.

The torus fibers and spheres $S_i$ are special Lagrangians with respect to the holomorphic volume form
$$ \Omega := d \log \bar{z} \wedge d u_1 \wedge d u_2$$
on $X_{s=0}$. Here we used $d \log \bar{z}$ instead of $d \log z$ to match the ordering of phases on both sides of the mirror\footnote{In order to match the phase inequality in the mirror side, we can either impose the mirror functor to be contravariant, or use the complex structure induced by the conjugate volume form $d \log \bar{z} \wedge d u_1 \wedge d u_2$ like here.  All $S_i$ are still special Lagrangians under this volume form, and we have the phase inequalities $\theta(S_i) > \theta(S_j)$ for $0< i <j$ or $i<j<1$.  This matches the ordering of the phases of stable objects in an exact triangle of the mirror B-side convention.  Namely for an exact triangle $L_1 \to L_1 \# L_2 \to L_2 \stackrel{[1]}{\to}$ where $L_i$ are special Lagrangians, their phases should satisfy $\theta(L_1) \leq \theta(L_1 \# L_2) \leq \theta(L_2)$.}. In particular, we measure the angle in clockwise direction for phases of $S_i$. The diagram in the right side of Figure \ref{fig:spheresss} compares the phases of $S_i$'s, where $S_0$ has the biggest phase in our convention.
In Section \ref{sec:fordefnc} we will see that taking these to be stable objects defines a Bridgeland stability condition on the derived Fukaya category.

Moreover each $S_i$ corresponds to another Lagrangian sphere $S_i^\dagger$ in $\mathcal{F}$, the flop of $S_i$ constructed in Section \ref{subsec:AflopLag}.  The Lagrangian torus fibers of $\pi^\dagger$ and $S_i^\dagger$ are special with respect to the pull-back holomorphic volume form from $X_{s=1}$, and they define another Bridgeland stability condition. In fact, we have $S_i^\dagger = \rho^{-1}(S'_{-i})$ where $\{ S'_{n} : n \in \Z \}$ is the set of new special Lagrangian spheres in $X_{s=1}$ which map to straight line segments by $z$-projection as in Figure \ref{fig:sprimes}.

For later use we orient these spheres as follows.  In conic fiber direction, each $S_i$ restricts to a 2-dimensional torus $\{|u_1| = |v_1|, |u_2| = |v_2|\}$.  We fix the orientation on the fiber torus to be $d \theta_1 \wedge d \theta_2$ where $\theta_i$ are the arguments of $u_i$ respectively.  We orient their matching paths as in the right side of Figure \ref{fig:spheresss}.

\begin{figure}[h]
\begin{center}
\includegraphics[height=1.8in]{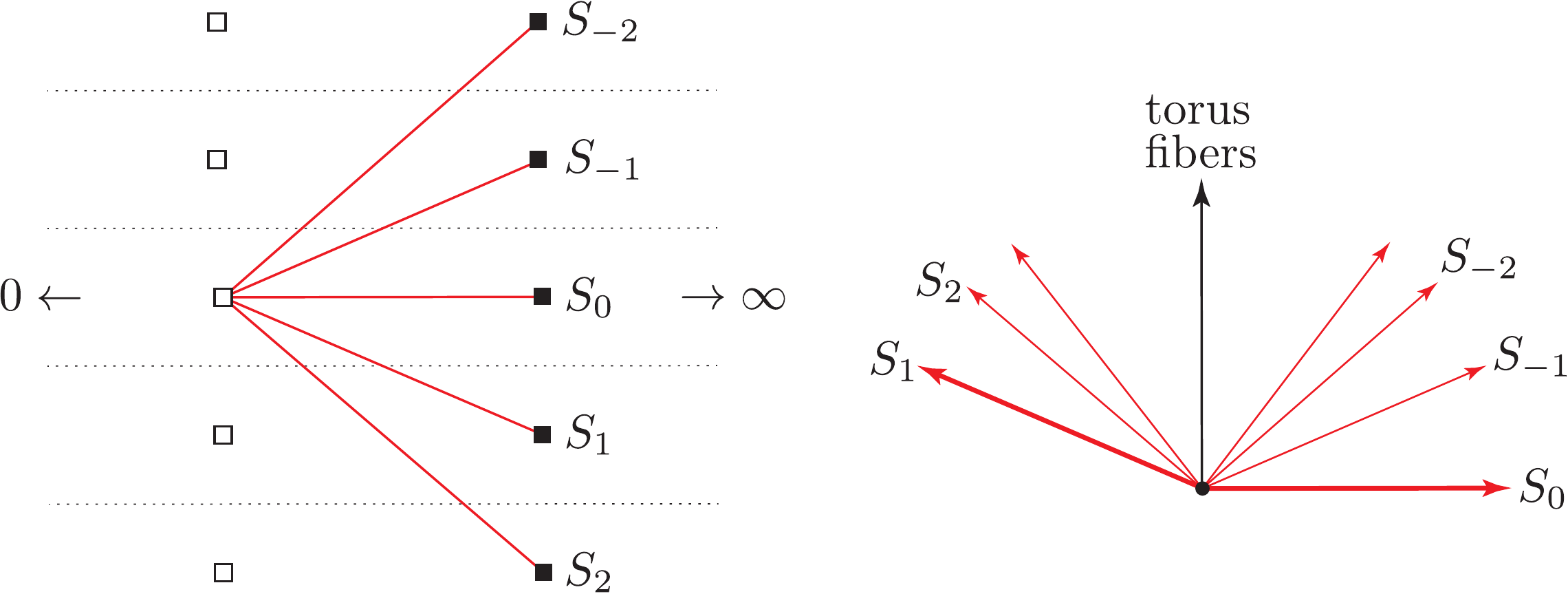}
\caption{Sequence of Lagrangian spheres in $X_{s=0}$}\label{fig:spheresss}
\end{center}
\end{figure}
 
Set $\bL_0:=S_0$ and $\bL_1:=S_1$. They are distinguished objects in the set $\{S_n\}$ of Lagrangian spheres in the sense that they have minimal/maximal slopes (or phases) as well as they generate $D^b \mathcal{F}$. We will study Lagrangian Floer theory of $\bL_0$ and $\bL_1$ intensely in Section \ref{subsec:FCbL}.

%

%
Recall from Section \ref{sec:bsiderev} that $\mathcal{D}_{\hat{Y}/Y}$ is the subcategory of $\mathcal{D}^b (\hat{Y})$ generated by $\mathcal{O}_{C} (-1) [1]$ and $\mathcal{O}_{C}$. \cite{CPU} proved the following equivalence of subcategories of $\Fuk (X_{s=0})$ and $\mathcal{D}^b (\hat{Y})$.

 \begin{thm}\label{thm:cpumod}\cite{CPU}
 There is an equivalence $D^b \mathcal{F} \simeq \mathcal{D}_{\hat{Y}/Y}$,
 sending
 \begin{equation}\label{eqn:equivSi}
 \bL_0 \mapsto \mathcal{O}_C (-1)[1]  \qquad \bL_1 \mapsto \mathcal{O}_C .
 \end{equation}
 \end{thm}
Using the chain model of Abouzaid \cite{Abou11}, they explicitly computed the $A_\infty$-structure of the endomorphism algebras of $\bL_0 \oplus \bL_1$ to conclude that
\begin{equation}\label{eqn:cprcpu1}
\End (\bL_0 \oplus \bL_1) \simeq \End ( \mathcal{O}_{C} (-1)[1] \oplus \mathcal{O}_C).
\end{equation} 
See \cite[Section 5, 7]{CPU} for more details.

In this paper, we shall use either the Morse-Bott model in \cite{FOOO} or pearl trajectories \cite{BC-pearl,Sheridan-CY} to study Lagrangian torus fibers and the noncommutative mirror functor.  They are conceptually easier to understand.

\begin{figure}[h]
\begin{center}
\includegraphics[height=2in]{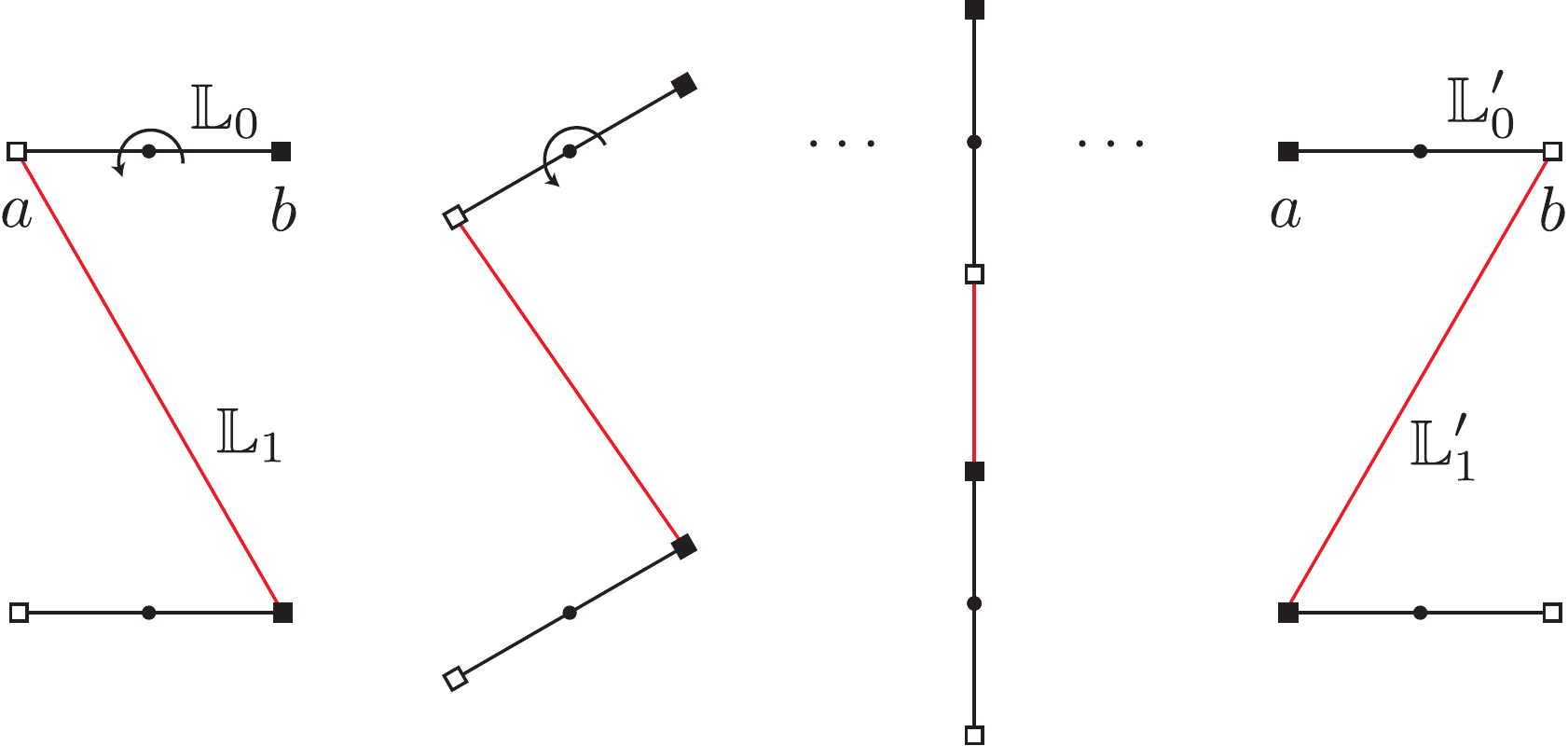}
\end{center}
\caption{Transformation of $\bL_0$ and $\bL_1$ by the symplectomorphism $\rho$}\label{fig:rotdc1}
\end{figure}


The A-flop can be realized by the symplectomorphism $\rho$ from $X_{s=0}$ to $X_{s=1}$ given in \ref{subsubsec:deflocintro} (see Figure \ref{fig:flop_rot}). Figure \ref{fig:rotdc1} shows how $\rho$ acts on $\bL_i$, where the third diagram describes the moment at which $\bL_0$ and $\bL_1$ happen to have the same phases.  
Observe that $X_{s=1}$ \eqref{eqn:xs01} is obtained from $X_{s=0}$ by swapping two sets of coordinates $(u_1,v_1)$ and $(u_2,v_2)$. However, swapping the coordinates is different from the symplectomorphism that gives A-flop, as its effect on $z$-plane shows.

As in Figure \ref{fig:rotdc1}, $\rho$ sends $\bL_0$ and $\bL_1$ to special Lagrangian spheres in $X_{s=1}$ which we denote by $\bL_0'$ and $\bL_1'$ respectively. 
Let $\mathcal{F}'$ denote the Fukaya subcategory of $X_{s=1}$ consisting of $\bL_0'$ and $\bL_1'$. There is a natural functor $\rho_{\ast}: \mathcal{F} \to \mathcal{F}'$ induced by the symplectomorphism $\rho$.
On the other hand, we can repeat the same argument as in the proof of Theorem \ref{thm:cpumod} to see that
$$D^b \mathcal{F}' \simeq \mathcal{D}_{\hat{Y}^\dagger/Y} \qquad \mbox{with} \qquad \bL_0' \mapsto \mathcal{O}_{C^\dagger} (-1)  \quad \mbox{and} \quad  \bL_1' \mapsto \mathcal{O}_{C^\dagger}(-2)[1]$$
Notice that this identification is coherent with the fact that $\bL_0'$ is somewhat similar to the orientation reversal of $\bL_0$, whereas $\bL_1 \stackrel{\rho}{\mapsto} \bL_1'$ can be understood as a change of winding number with respect to $z=0$.

In fact, we have 
\begin{equation}\label{eqn:cprrelsymp}
{\End} (\bL_0' \oplus \bL_1') \simeq {\End} (\bL_0 \oplus \bL_1)
\end{equation}
as two set of objects are related by a symplectomorphism, and
\begin{equation}\label{eqn:cprflopf}
 {\End} \left(\mathcal{O}_{C^\dagger} (-1) \oplus \mathcal{O}_{C^\dagger} (-2)[1]  \right) \simeq {\End} \left(  \mathcal{O}_C (-1)[1] \oplus \mathcal{O}_C \right)
 \end{equation}
due to the flop functor (see \ref{subsec:exatiyah}). It directly implies that the functor $\rho_\ast$  induced by the symplectomorphism is mirror to the flop functor through the identification of A and B side categories via \cite{CPU}. Namely,

 
%
 \begin{prop}
 We have a commutative diagram of equivalences:
 \begin{equation}
 \xymatrix{D^b \mathcal{F} \ar[d]_{\rho_\ast} \ar[rr]^{\simeq} && \mathcal{D}_{\hat{Y}/Y} \ar[d]^{\Phi}  \\
D^b \mathcal{F}'  \ar[rr]^{\simeq}&& \mathcal{D}_{\hat{Y}^\dagger/Y}}
 \end{equation}
 \end{prop}

\begin{proof}
It obviously commutes on the level of objects by the construction.  \eqref{eqn:cprcpu1}, \eqref{eqn:cprrelsymp} and \eqref{eqn:cprflopf} imply that the diagram also commutes on morphism level.
\end{proof}

We shall study how the symplectomorphism $\rho$ or its induced functor $\rho_\ast$ acts on various geometric objects in $D^b \mathcal{F}$. For that, we should  examine what kind of geometric objects are actually contained in $D^b \mathcal{F}$.
%

%

\subsection{Geometric objects of $D^b \mathcal{F}$}\label{subsec:geobjf}

Let us first prove that any torus fiber intersecting $\bL_0$ and $\bL_1$ is isomorphic to a mapping cone $Cone \left(\bL_0 \stackrel{\alpha}{\to} \bL_1\right)$ for some degree-1 element $\alpha \in HF(\bL_0,\bL_1)$ in the derived Fukaya category. In particular this will imply that the category $D^b \mathcal{F}$ contains those torus fibers as objects. 

One can choose the gradings on $\bL_i$ such that  $HF^\ast(\bL_0,\bL_1) = H^\ast (S^1_a)[-1] \oplus H^\ast (S^1_b)[-1]$. Here, we use the Morse-Bott  model, where $S^1_a$ and $S^1_b$ denotes the intersection loci over $z=a$ and $z=b$, respectively. Both $S^1_a$ and $S^1_b$ are isomorphic to a circle. Thus degree 1 elements in $HF(\bL_0,\bL_1)$ are given by linear combinations of (Poincar\`e duals of) fundamental classes of $S^1_a$ and $S^1_b$.  
The cone $Cone (\bL_0 \stackrel{\alpha}{\to} \bL_1)$ can be identified with a boundary deformed object $(\bL_0 \oplus \bL_1, \alpha)$ (see \cite{FOOO} or \cite{Seidel-book}).

Let $L_c:=$ (for $a < c < b$) be a Lagrangian torus that intersects $\bL_0$  at $z=c$. This condition determines $L_c$ uniquely, as components of $L_c$ in double conic fiber direction should satisfy the same equation as those of $\bL_0$ and $\bL_1$. We orient $L_c$ as in Figure \ref{fig:spheresss} in $z$-plane components, and use the standard one (from $d\theta_1 d\theta_2$ as for $S_i$) along the conic fiber directions.
$L_c$ cleanly intersects $\bL_0$ and $\bL_1$ along 2-dimensional tori which we denote by $T_0:= L_c \cap \bL_0$ and $T_1:=L_c \cap \bL_1$. One can  see that $CF(L_c,\bL_0)=C^\ast (T_0)$ and $CF(\bL_1,L_c) = C^\ast (T_1)$ for suitable choice of a grading on $L_c$. Similarly, $CF(\bL_0, L_c) = C^\ast (T_0)[-1]$ and $CF(L_c,\bL_1) = C^\ast (T_1)[-1]$. 

Let $U_{\rho_1,\rho_2,\rho_z}$ be a unitary flat line bundle on $L_c$ whose holonomies along circles in the double conic fibers are $\rho_1$ and $\rho_2$ and that along the circle in $z$-plane is $\rho_z$.

\begin{lemma}
If $(\rho_1,\rho_2) \neq (1,1)$, then 
\begin{equation}\label{eqn:vanish12}
HF(\bL_0, (L_c,U_{\rho_1,\rho_2,\rho_z})) =0, \quad HF(\bL_1, (L_c,U_{\rho_1,\rho_2,\rho_z})) =0.
\end{equation} 
\end{lemma}
\begin{proof}
One can simply use the Morse-Bott model for each of cohomology groups in \eqref{eqn:vanish12}.  Each of this group is simply a singular cohomology of the intersection loci, equipped with twisted differential. Since the intersection loci are 2-dimensional torus in the double conic fiber, the twisting is determined by $(\rho_1,\rho_2)$. Here we only have classical differential, as there is no holomorphic strip between $\bL_i$ and $T_c$. One can easily check that the cohomology vanishes if the twisting is nontrivial.

Alternatively, one can perturb Lagrangians to have transversal intersections as in Figure \ref{fig:tcone_pert} to see that the Floer differential has coefficients $\rho_1-1$ and $\rho_2-1$, which are nonzero for nontrivial $(\rho_1,\rho_2)$. 
\end{proof}
The lemma implies that $(L_c,U_{\rho_1,\rho_2,\rho_z})$ has no Floer theoretic intersection with $\bL_0$ or $\bL_1$ unless $\rho_1=\rho_2=1$.
From now on, we will only consider flat line bundles of the type $U_{0,0,\rho_z}$ on Lagrangian torus fibers, which will be written as $U_{\rho_z}$ instead of $U_{0,0,\rho_z}$ for notational simplicity. Let $P_0:=PD[T_0] \in CF^\ast (L_c,\bL_0)=C^\ast (T_0)$ and $P_1:=PD[T_1] \in CF^\ast (\bL_1, L_c)=C^\ast(T_1)$. We also set $\alpha_a:=PD[S^1_a] \in CF^\ast (\bL_0,\bL_1)$ and $\alpha_b:=PD[S^1_b] \in CF^\ast (\bL_0,\bL_1)$. Notice that $\deg \alpha_a = \deg \alpha_b =1$ whereas $\deg P_0 = \deg P_1 =0$.

\begin{figure}[h]
\begin{center}
\includegraphics[height=2.2in]{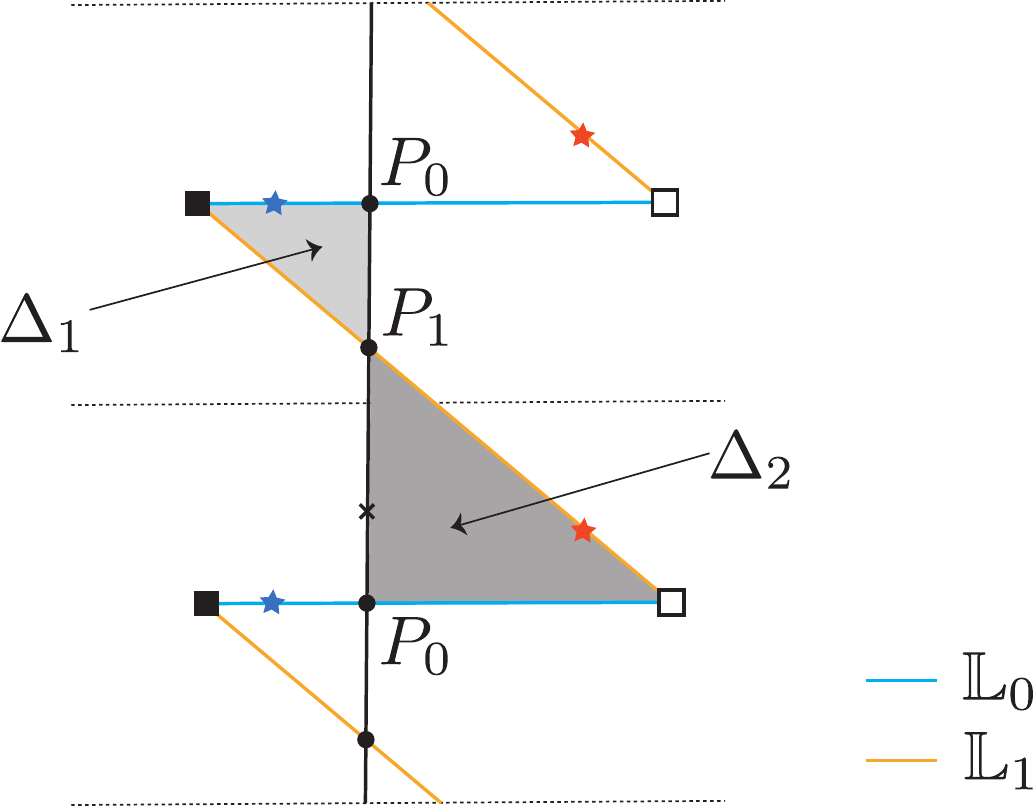}
\end{center}
\caption{Two triangles contributing to $m_1$ and $m_2$}\label{fig:tcone}
\end{figure}

\begin{lemma}\label{lem:ratioalpha}
$P_0 \in CF^0 ( (L_c,U_{\rho_z}), (\bL_0 \oplus \bL_1, \alpha) )$ and $P_1 \in CF^0 ( (\bL_0 \oplus \bL_1, \alpha), (L_c,U_{\rho_z}) )$ are cycles with respect to $m_1^{0,\alpha}$ and $m_1^{\alpha,0}$ respectively if and only if $\alpha$ is given as $\lambda_a \alpha_a + \lambda_b \alpha_b \in CF^1(\bL_0,\bL_1)$ with $(\lambda_a:\lambda_b) = (T^{\omega(\Delta_2)} \rho_z : T^{\omega(\Delta_1)} )$\footnote{It is harmless to put $T=e^{-1}$ since only finitely many polygons contribute to $A_\infty$-structures. Nevertheless we will keep the notation $T$ to highlight contributions from nontrivial holomorphic polygons.} where $\Delta_1$ and $\Delta_2$ are triangles shaded in Figure \ref{fig:tcone}.
\end{lemma}

\begin{proof}
We will prove the lemma for $P_1$ only, and the proof for $P_0$ is similar. We pick a point $\times$ as in Figure \ref{fig:tcone} for representative of $U_{\rho_z}$ so that when boundary of a holomorphic polygon passes this point, the corresponding $m_k$-operation is multiplied by $\rho_z^{\pm1}$ depending on the orientation. (More precisely, the point $\times$ represent 2-dimensional subtorus in $L_c$ lying over this point, which is called a hyper-torus and used to fix the gauge of a flat line bundle in \cite{CHLtoric}.)

 Observe that two holomorphic triangles $\Delta_1$ and $\Delta_2$ shown in Figure \ref{fig:tcone} contribute to the following operations:
 \begin{equation}\label{eqn:m1bm2}
m_2(\lambda_a \alpha_a,P_1) = \lambda_a T^{\omega(\Delta_1)} \bar{P}_0, \qquad m_2(\lambda_b \alpha_b,P_1) = - \rho_z \lambda_b T^{\omega(\Delta_2)} \bar{P}_0.
\end{equation}
where $\bar{P}_0$ is $PD[T_0]$ regarded as an element of $CF(\bL_0, L_c) = C^\ast (T_0) [-1] \subset CF((\bL_0 \oplus \bL_1, \alpha), (L_c, U_{\rho_z}))$ (note that $\deg \bar{P}_0 = 1$).  We do not provide the precise sign rule here since it is not crucial in our argument. Indeed we can assume that two operations in \eqref{eqn:m1bm2} produces outputs with the opposite signs by replacing $\lambda_b$ to $-\lambda_b$ if necessary.

Therefore we see that
$$m_1^{\alpha,0}(P_1) =\sum_k m_k (\alpha, \cdots, \alpha, P_1) = \left(\lambda_a T^{\omega(\Delta_1)} - \rho_z \lambda_b T^{\omega(\Delta_2)} \right) \bar{P}_0=0$$
if and only if $\lambda_a$ and $\lambda_b$ have the ratio as given in the statement.
\end{proof}

We next prove that $P_0$ and $P_1	$ in Lemma \ref{lem:ratioalpha} give isomorphisms between two objects $(L_c, U_{\rho_z})$ and $Cone (\bL_0 \stackrel{\alpha}{\to} \bL_1)$ where $\alpha$ is chosen as in Lemma \ref{lem:ratioalpha}. Here, it is enough to present the ratio between $\lambda_a$ and $\lambda_b$,
as the mapping cone does not depend on the scaling of $\alpha$ by an element in $\C^\times$ (or in $\Lambda\setminus \{0\}$ if we do not substitute $T$ by $e^{-1}$).


%
%
%

\begin{prop}\label{prop:teqcone}
We have $(L_c, U_{\rho_z}) \cong Cone (\bL_0 \stackrel{\alpha}{\to} \bL_1)$ in the derived Fukaya category of $X_{s=0}$ where $\alpha=\lambda_a \alpha_a + \lambda_b \alpha_b$ is chosen as in Lemma \ref{lem:ratioalpha}. 
\end{prop}

\begin{proof}
Let us fix $\lambda_a$ and $\lambda_b$ to be $\rho_z T^{\omega(\Delta_2)} $ and $T^{\omega(\Delta_1)}$ for simplicity.
We claim that
$$m_2^{\alpha,0,\alpha} ( P_1,P_0) =C(\one_{\bL_0} + \one_{\bL_1}), \quad m_2^{0,\alpha,0} ( P_0,P_1) =C \one_{L_c}$$
for some common constant $C$. (One should rescale $P_0$ and $P_1$ to get strict identity morphisms.)
To see this, pick generic points ``$\star$" on $\bL_0$, $\bL_1$ as in Figure \ref{fig:tcone}, whose number of appearance in the boundary of holomorphic discs determines the coefficient of $\one_{\bL_i}$. The same triangles in the proof of Lemma \ref{lem:ratioalpha} now contribute as
$$m_2^{\alpha,0,\alpha} ( P_1,P_0)  = \lambda_a T^{\omega(\Delta_1)} \one_{\bL_1} + \rho_z \lambda_b T^{\omega(\Delta_2)} \one_{\bL_0} = \rho_z T^{\omega(\Delta_1)+\omega(\Delta_2)} \left(\one_{\bL_0} + \one_{\bL_1} \right)$$
Here two contributions add up contrary to \eqref{eqn:m1bm2}. In fact, the relative signs are completely determined by $z$-directions since all the Lagrangians share the other directions, and one can use the sign rule due to Seidel \cite{Seidel-g2} for $z$-plane components.

It is easy to check that the computation does not depend on the choice of generic points $\star$ (it is essentially because $P_0$ and $P_1$ are cycles). Likewise, $\Delta_2$ contributes to
$$m_2^{0,\alpha,0} ( P_0,P_1) = \rho_z \lambda_b T^{\omega(\Delta_2)} \one_{L_c} = \rho_z T^{\omega(\Delta_1)+\omega(\Delta_2)} \one_{L_c}.$$

In particular, Proposition \ref{prop:teqcone} implies the following exact sequence in the derived Fukaya category
$$ \bL_1 \to (L_c, U_{\rho_z}) \to \bL_0 \stackrel{[1]}{\to}.$$

\end{proof}


\begin{remark}
Analogously, the following gives an exact triangle in $\mathcal{D}_{\hat{Y}/Y}$:
$$ \mathcal{O}_C \to \mathcal{O}_y \to \mathcal{O}_{C} (-1)[1] \stackrel{[1]}{\to}$$
for a point $y$ in $C$, or equivalently $\mathcal{O}_y \cong Cone(\mathcal{O}_C (-1) [1] \to \mathcal{O}_C)$ for some degree 1 morphism. Note that $\mathcal{O}_y$ is a SYZ mirror of one of torus fibers $L_c$ (with a flat line bundle $U_{\rho_z}$). Proposition \ref{prop:teqcone} and the above exact triangle shows that the equivalence \eqref{eqn:equivSi} sends torus fibers to skyscraper sheaves over points in $C$. 
\end{remark}

By symmetric argument
(or by the triangulated structure on $D^b\Fuk (X_{s=0})$), one also has
$$ Cone(\bL_0 [1] \stackrel{\beta}{\to} L_c) \cong \bL_1$$
in $D^b \Fuk (X_{s=0})$ where $\deg (\beta)= 1$. Here, $[1]$ can be though of as taking orientation reversal of the $z$-component of $\bL_0$ (or, more precisely, such change of grading).
A similar statement holds true for other Lagrangian spheres $\{S_m : m \in \Z\}$ (Figure \ref{fig:spheresss}). 


%

\begin{figure}[h]
\begin{center}
\includegraphics[height=3in]{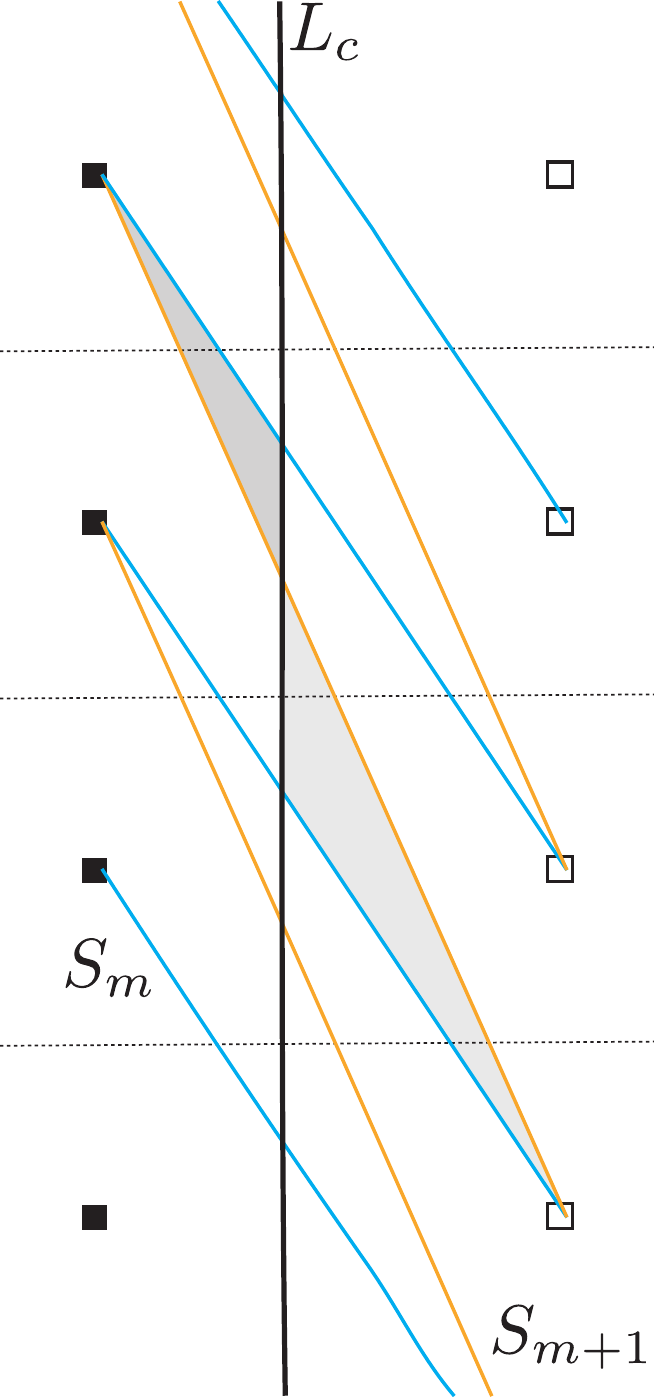}
\end{center}
\caption{Triangles contributing to $m_1$ and $m_2$ on $CF(S_m[1] \oplus S_{m+1}, L_c)$ and $CF(L_c,S_m[1] \oplus S_{m+1})$}\label{fig:smt}
\end{figure}

\begin{prop}\label{prop:smsmt} Lagrangian sphere $S_i$ for any $i$ can be obtained from taking successive cones from $S_0$ and $S_1$. More precisely, one has the following:
\begin{itemize}
\item for $m \geq 1$,
$$S_{m+1} \cong Cone (L_c \to S_{m}),$$
\item for $n \leq 0$,
$$S_{n-1} \cong Cone (S_{n}[1] \to L_c [1]).$$
\end{itemize}
\end{prop}

\begin{proof}
We only prove the first identity, and the proof for the second can be performed in a similar manner.
One can easily check that $L_c = Cone (S_{m} [1] \stackrel{\alpha}{\to} S_{m+1})$ (for $n \leq 0$) by the same argument as in the proof of Lemma \ref{lem:ratioalpha}, where $\alpha$ is a degree 1 morphism from $S_{n-1}$ to $S_n [1]$. The contributing pair of triangles are as shown in Figure \ref{fig:smt}, and hence we should take into account the relative areas of these two triangles together with the location of $c\in (a,b)$, when we choose $\alpha$. We omit the details as it is completely parallel to the proof of Lemma \ref{lem:ratioalpha}.

From $L_c = Cone (S_{m}[1] \stackrel{\alpha}{\to} S_{m+1})$, we have an exact triangle in the derived Fukaya category $S_{m+1} \to L_c \to S_{m} [1] \stackrel{[1]}{\to}$ or equivalently $S_{m} \to S_{m+1} \to L_c  \stackrel{[1]}{\to}$, which implies $S_{m+1}  = Cone (L_c \to S_m)$ for some degree 1 morphism.


\end{proof}

We conclude that $D^b \mathcal{F}$ contains a sequence of Lagrangian spheres $\{S_n: n \in \Z\}$ and $\mathbb{P}^1\setminus\{0,\infty\}$-family of Lagrangian tori parametrized by $(c,\rho_z)$, where two missing points $0$ and $\infty$ are presumably corresponding to two singular torus fibers. Indeed, $D^b \mathcal{F}$ contains the cones $Cone(\bL_0 \stackrel{\lambda_b \alpha_b}{\to} \bL_1)$, $Cone(\bL_0 \stackrel{\lambda_a \alpha_a}{\to} \bL_1)$, and we believe that they are isomorphic to two singular fibers that pass through $z=a$ and $z=b$, respectively. Although a similar argument as in the proof of Proposition \ref{prop:teqcone} seem to go through, we do not spell this out here due to technical reasons.

Notice that the above Lagrangian submanifolds are all special, thus they are expected to be stable objects in the Fukaya category. Later in Section \ref{sec:fordefnc}, we will study their transformations into noncommutative resolution of the conifold $Y$ to give stable quiver representations.

\begin{remark}
By the equivalence in Theorem \ref{thm:cpumod}, Lagrangians spheres correspond to line bundle on the exceptional curve $C$ (or their shifts) as follows:
\begin{equation*}
\begin{array}{lclcl}
S_{m} &\mapsto& \mathcal{O}_C (m-1)  &\mbox{for}& m \geq 1 \\
S_{n} &\mapsto& \mathcal{O}_C (n-1) [1]  &\mbox{for}& n \leq 0.
\end{array}
\end{equation*}
One can easily check this comparing the cone relations in Proposition \ref{prop:smsmt} and exact sequences consisting of line bundles and skyscraper sheaves on $C$.
\end{remark}

\subsection{Mirror to perverse point sheaves}
In this section, we describe how torus fibers (intersecting $\bL_0$ and $\bL_1$) are affected by A-flop. We will see that they behave precisely in the same way as skyscraper sheaves supported at points in $C (\subset \hat{Y})$. Note that points in $\hat{Y}$ are mirror to torus fibers in SYZ point of view, and those in $C$ are mirror to torus fibers that intersect $\bL_i$. Thus it is natural to expect that those torus fibers are transformed to unstable objects (i.e. non-special Lagrangians) which can be written as  mapping cones analogous to \eqref{eqn:pervpt}.
\begin{figure}[h]
\begin{center}
\includegraphics[height=3in]{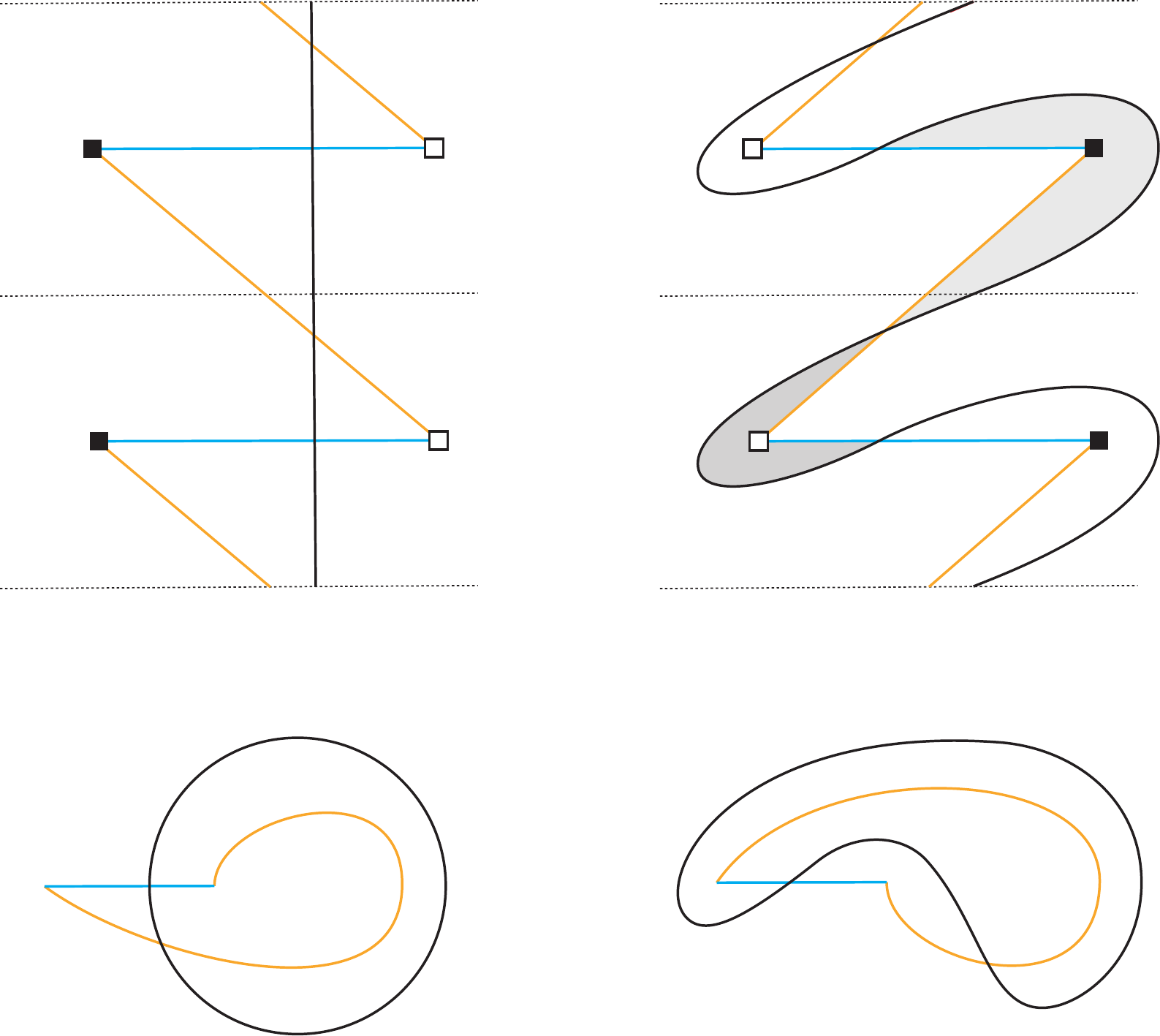}
\end{center}
\caption{Transformation of torus fibers under $\rho$}\label{fig:pervpta}
\end{figure}

\begin{prop}\label{prop:sesconeaf}
The functor $\rho_\ast$ sends $(L_c,\rho_z)$ to $Cone (\bL_0' \stackrel{\alpha'}{\to} \bL_1')$ for a degree $1$ morphism $\alpha'$.
\end{prop}

\begin{proof}
The proof is the same as that of Lemma \ref{lem:ratioalpha}, except that $\alpha'$ in $HF(\bL_0',\bL_1')$ now represent outer (non-convex) angle in the $z$-plane picture.
Altenatively, since $\rho_\ast$ sends cones to cones, and $\rho_\ast (\bL_0)= \bL_0'$, $\rho_\ast(\bL_1)=\bL_1'$, we have
$$\rho_\ast (L_c,\rho_z) = \rho_\ast (Cone (\bL_0 \stackrel{\alpha}{\to} \bL_1) ) = Cone (\rho (\bL_0) \stackrel{\rho_\ast (\alpha)}{\to} \rho(\bL_1)) = Cone (\bL_0' \stackrel{\alpha'}{\to} \bL_1')$$
where $\alpha'$ is a degree one morphism from $\bL_0'$ to $\bL_1'$, and occupies outer angles (after $z$-projection) as in Figure \ref{fig:pervpta}.


\end{proof}

\begin{remark}
Note that we have an exact triangle
$$\bL_1' \to \rho_\ast (L_c,\rho_z) \to \bL_0' \stackrel{[1]}{\to}$$
from Proposition \ref{prop:sesconeaf}. Thus $\bL_1'$ can be thought of as a subobject of $\rho_\ast (L_c,\rho_z)$. $\bL_1'$ has bigger phase than $\rho_\ast (L_c,\rho_z)$, which is another way to explain unstability of $\rho_\ast (L_c,\rho_z)$. 
\end{remark}

Likewise, flop sends most of Lagrangians spheres in $\{S_m : m\in \Z\}$ to non-special objects. 
In fact, it is easy to see from the picture that $S_0$ and $S_1$ are the only spheres in this family that remain special after A-flop. We conclude that the equivalence $\rho_\ast$ does not preserve the set of special Lagrangians, and hence the A-flop can be thought of as a nontrivial change of holomorphic volume form, while keeping the symplectic structure as its induced from a symplectomorphism. We will revisit this point of view in \ref{subsec:bridgelandfuk}.

\section{Non-commutative mirror functor for the deformed conifold and stability conditions}\label{sec:fordefnc}

So far, we have studied A-flop on the smoothing $X_{s=0}$ \eqref{eqn:xs01} of the conifold mostly in SYZ perspective comparing with its SYZ mirror, the resolve conifold (taken away a divisor) .
In this section we will consider a certain quiver algebra as another mirror to $X_{s=0}$, which is well-known to be a noncommutative crepant resolution of the conifold. 
The relation between noncommutative resolution and commutative one will be explained later (see Remark \ref{rem:noncommresc}).  The mirror quiver category will enable us to study stability conditions more explicitly.

There is a natural way to obtain the above quiver as a formal deformation space of the object $\mathbb{L}=\bL_0 \oplus \bL_1$ (recall $\bL_0 = S_0$ and $\bL_1=S_1$ are Lagrangian spheres with maximal/minimal phases) in the Fukaya category.  By the result in \cite{CHL}, such a construction comes with an $A_\infty$-functor from a Fukaya category to the category of quiver representations. We will construct geometric stability conditions using the functor, and examine A-flop on these stability conditions.

We begin with an explicit computation of the $A_\infty$-structure on $CF(\bL,\bL)$, which is crucial to describe formal deformation space of $\bL$.

\subsection{Floer cohomology of $\mathbb{L}$}\label{subsec:FCbL}
As our Lagrangian $\bL$ is given as a direct sum, $CF(\bL,\bL)$ consists of four components:
$$CF(\bL,\bL) = CF(\bL_0,\bL_0) \oplus CF(\bL_1,\bL_1) \oplus CF(\bL_0,\bL_1) \oplus CF(\bL_1, \bL_0). $$
The first two components are both isomorphic to the cohomology of the three-sphere as graded vector spaces, and hence have degree-0 and degree-3 elements only. 
These elements will not be used for formal deformations.  We only take degree-1 elements for deformations, so that the $\Z$-grading is preserved.

Recall that $\bL_0$ and $\bL_1$ intersect along two disjoint circles.
There are several computable models for $CF(\bL_0,\bL_1)$ and $CF(\bL_1,\bL_0)$ provided in \cite{Abou11}.  Explicit computation was given in \cite{CPU} using one of these models, which we spell out here.  

\begin{thm}(See \cite[Theorem 7.1]{CPU}.)\label{thm:sumainfs} The $A_\infty$-structure on $CF(\bL,\bL)$ are given as follows. As vector spaces, 
\begin{equation*}
\begin{array}{l}
CF(\bL_i,\bL_i) = \Lambda \langle \one_{\bL_i} \rangle \oplus \Lambda \langle [\pt]_{\bL_i} \rangle \,\,\,\, \mbox{for}\,\,\, i=0,1 \\
CF(\bL_0,\bL_1) = \Lambda \langle X \rangle \oplus  \Lambda \langle Z \rangle \oplus  \Lambda \langle \bar{Y} \rangle \oplus  \Lambda \langle \bar{W} \rangle \\
$$CF(\bL_1,\bL_0) =  \Lambda \langle Y \rangle \oplus  \Lambda \langle W \rangle \oplus  \Lambda \langle \bar{X} \rangle \oplus  \Lambda \langle \bar{Z} \rangle.
\end{array}
\end{equation*}
with degrees of generators given as
\begin{equation*}
\begin{array}{l}
\deg \one_{\bL_i} = 0, \quad \deg [\pt]_{\bL_i} = 3, \\
\deg X= \deg Y = \deg Z = \deg W = 1,\\
\deg \bar{X} = \deg \bar{Y} = \deg \bar{Z} =\deg \bar{W} =2.
\end{array}
\end{equation*}
We have $m_1\equiv0$ and $m_{\geq 4}\equiv0$. The only nontrivial operations are 
\begin{equation*}
\begin{array}{l}
-m_2(X, \bar{X} )= -m_2(Z,\bar{Z} ) = m_2 (\bar{Y},Y) = m_2(\bar{W},W) = [\pt]_{\bL_0}, \\
m_2(\bar{X},X )= m_2(\bar{Z},Z ) = -m_2 (Y,\bar{Y}) = -m_2(W,\bar{W}) = [\pt]_{\bL_1}, \\
m_3(X,Y,Z) =- m_3(Z,Y,X)= \bar{W}, \quad m_3(Y,Z,W) =- m_3(W,Z,Y)= \bar{X},\\
m_3(Z,W,X) =- m_3(X,W,Z)= \bar{Y}, \quad m_3(W,X,Y) =- m_3(Y,X,W)= \bar{Z}.
\end{array}
\end{equation*}
and those determined by the property of the unit $\one_{\bL_i}$.

\end{thm}

In what follows, we take an alternative way to compute $A_\infty$-structure on $CF(\bL,\bL)$ hiring \emph{pearl trajectories}, which is more geometric in the sense that it shows explicitly the holomorphic disks (attached with Morse trajectories) contributing to the $A_\infty$-operations. This will also help us to have geometric understanding of various computations to be made later, although most of the proof will rely on algebraic arguments.

First we choose a generic Morse function $f_i$ on $\bL_i$ with minimum and maximum only for $i=0,1$. We denote these critical points by  $\one_{\bL_i}, [\pt]_{\bL_i}$ by obvious analogy, where $\deg \one_{\bL_i} = 0$ and $\deg\, [\pt]_{\bL_i} =3$. Then $CF(\bL_i,\bL_i)$ is defined to be the Morse complex of $f_i$, which is nothing but the 2-dimensional vector space generated by $\one_{\bL_i}, [\pt]_{\bL_i}$.

The Morse trajectories of $f_i$ are described as follows. Recall that the two Lagrangians $\bL_0$ and $\bL_1$ intersect along two disjoint circles which we denoted by $S^1_a$ and $S^1_b$ (lying over $z=a$ and $z=b$, respectively). We want to argue that generically, there is a unique gradient flow line in each $\bL_i$ that runs from $S^1_a$ to $S^1_b$ and vice versa. We may assume that there are no critical points on $S^1_a$ or $S^1_b$ by genericity.

Let $W^- \left( S^1_b \right)$ be the unstable manifold of $S^1_b$ with respect to the Morse function $f_i$ on $\bL_i$, namely
$$W^{-} \left( S^1_b \right) = \left\{ x\in \bL_0 \, : \, \varphi^t (x) \in S^1_b\,\,\, \mbox{for some} \,\,\, t\geq0 \right\} $$
where $\varphi^t$ is the flow of $\nabla f_i$. Including the maximum ($\one_{\bL_i}$ in our notation), the unstable manifold of $S^1_b$ is topologically a disk that bounds $S^1_b$. Now observe that two circles $S^1_a$ and $S^1_b$ form a Hopf link in $\bL_i$. Therefore, generically $S^1_a$ intersects the disk $W^- \left(S^1_b \right) \cup \{max\}$ at one point as shown in Figure \ref{fig:hopf}. Therefore we see that there is a unique trajectory flowing from $S^1_a$ to $S^1_b$. This is the only property of the Morse functions $f_0$ and $f_1$ which we will use later.


\begin{figure}[h]
\begin{center}
\includegraphics[height=2in]{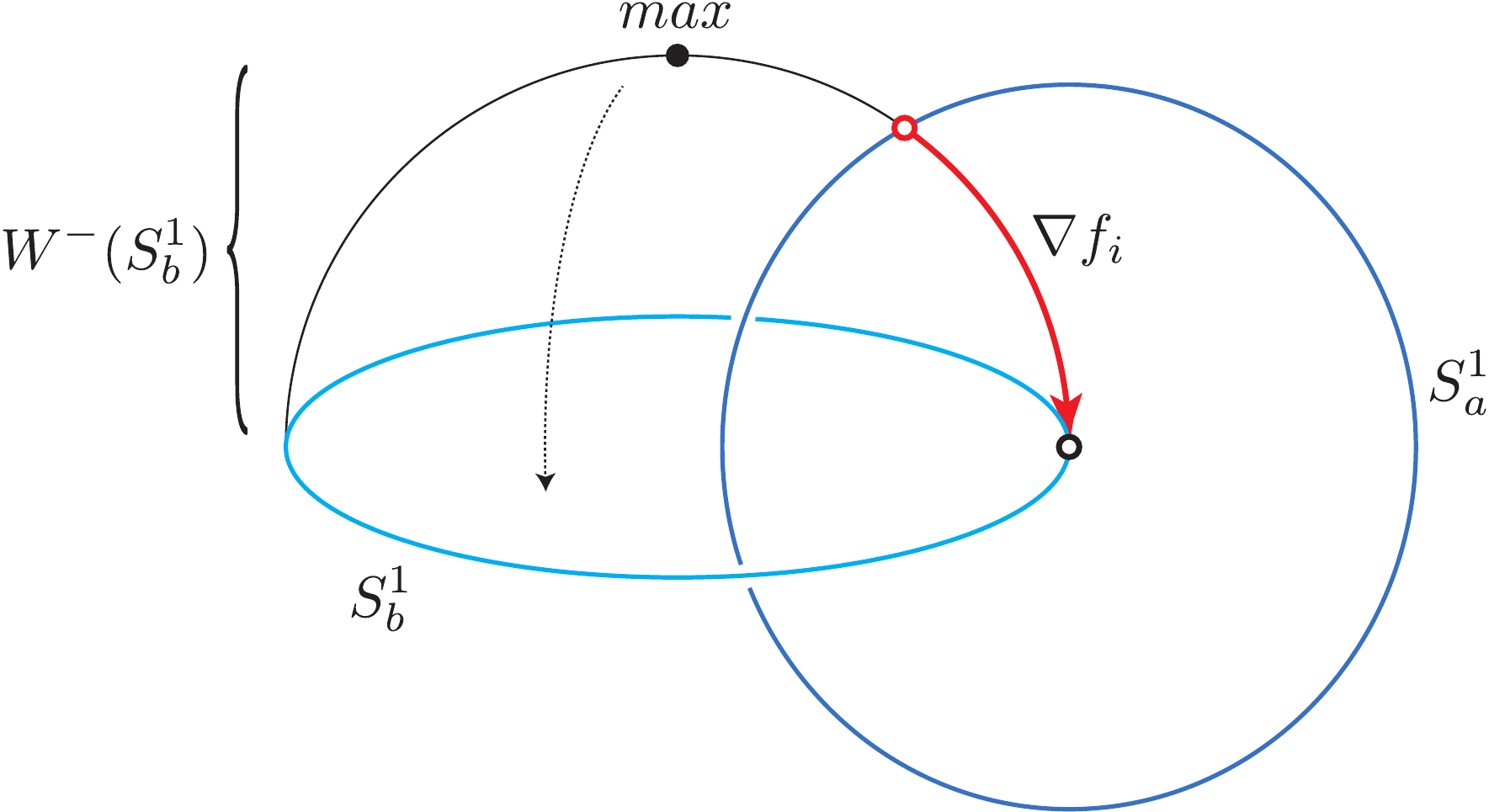}
\end{center}
\caption{A trajectory of the Morse function $f_i$ from $S^1_a$ to $S^1_b$}\label{fig:hopf}
\end{figure}

%

For the other components of $CF(\bL,\bL)$, we perturb $\bL_1$ in double conic fiber direction as in Figure \ref{fig:pearl1}, so that $\bL_0$ and $\bL_1$ intersect each other transversely at four different points after perturbation. Therefore, both $CF(\bL_0,\bL_1)$ and $CF(\bL_1,\bL_0)$ are generated by these four points, which we denote as follows.
$$CF(\bL_0,\bL_1) = \Lambda \langle X \rangle \oplus  \Lambda \langle Z \rangle \oplus  \Lambda \langle \bar{Y} \rangle \oplus  \Lambda \langle \bar{W} \rangle$$
$$CF(\bL_1,\bL_0) =  \Lambda \langle Y \rangle \oplus  \Lambda \langle W \rangle \oplus  \Lambda \langle \bar{X} \rangle \oplus  \Lambda \langle \bar{Z} \rangle$$
with degrees of generators given as
$$\deg X= \deg Y = \deg Z = \deg W = 1, \qquad \deg \bar{X} = \deg \bar{Y} = \deg \bar{Z} =\deg \bar{W} =2.$$ 
Here, $X$ and $\bar{X}$ are represented by the same point, but regarded as elements in $CF(\bL_0,\bL_1)$ and $CF(\bL_1,\bL_0)$ respectively, and similar for $Y,Z,W$. In fact, they can be thought of as Poincare dual to each other. Obviously, they are all cycles (i.e. $m_1$-closed) since opposite strips (pairs of strips on cylinders in Figure \ref{fig:pearl1}) cancel pairwise. Therefore $CF(\bL,\bL)$ comes with a trivial differential.

Now we are ready to spell out $A_\infty$-algebra structure on $CF(\bL,\bL)$ in terms of the above model. Recall from \cite{BC-pearl,Sheridan-CY} that $A_\infty$-operation counts the configurations which consist of several holomorphic disks (pearls) joined by gradient trajectories as shown in Figure \ref{fig:pearl1}.  The constant disk at $X$ (and $\bar{X}$) attached with flows to $[\pt]_{\bL_i}$ contributes as $m_2(X, \bar{X} )= [\pt]_{\bL_0}$. Similarly, we have
$$m_2(X, \bar{X} )= m_2(Z,\bar{Z} ) = -m_2 (\bar{Y},Y) = -m_2(\bar{W},W) = - [\pt]_{\bL_0},$$
$$m_2(\bar{X},X )= m_2(\bar{Z},Z ) = -m_2 (Y,\bar{Y}) = -m_2(W,\bar{W}) =  [\pt]_{\bL_1}.$$
Other $m_2$'s are either determined by properties of units $1_{\bL_i}$, or zero by degree reason. 

\begin{figure}[h]
\begin{center}
\includegraphics[height=2.2in]{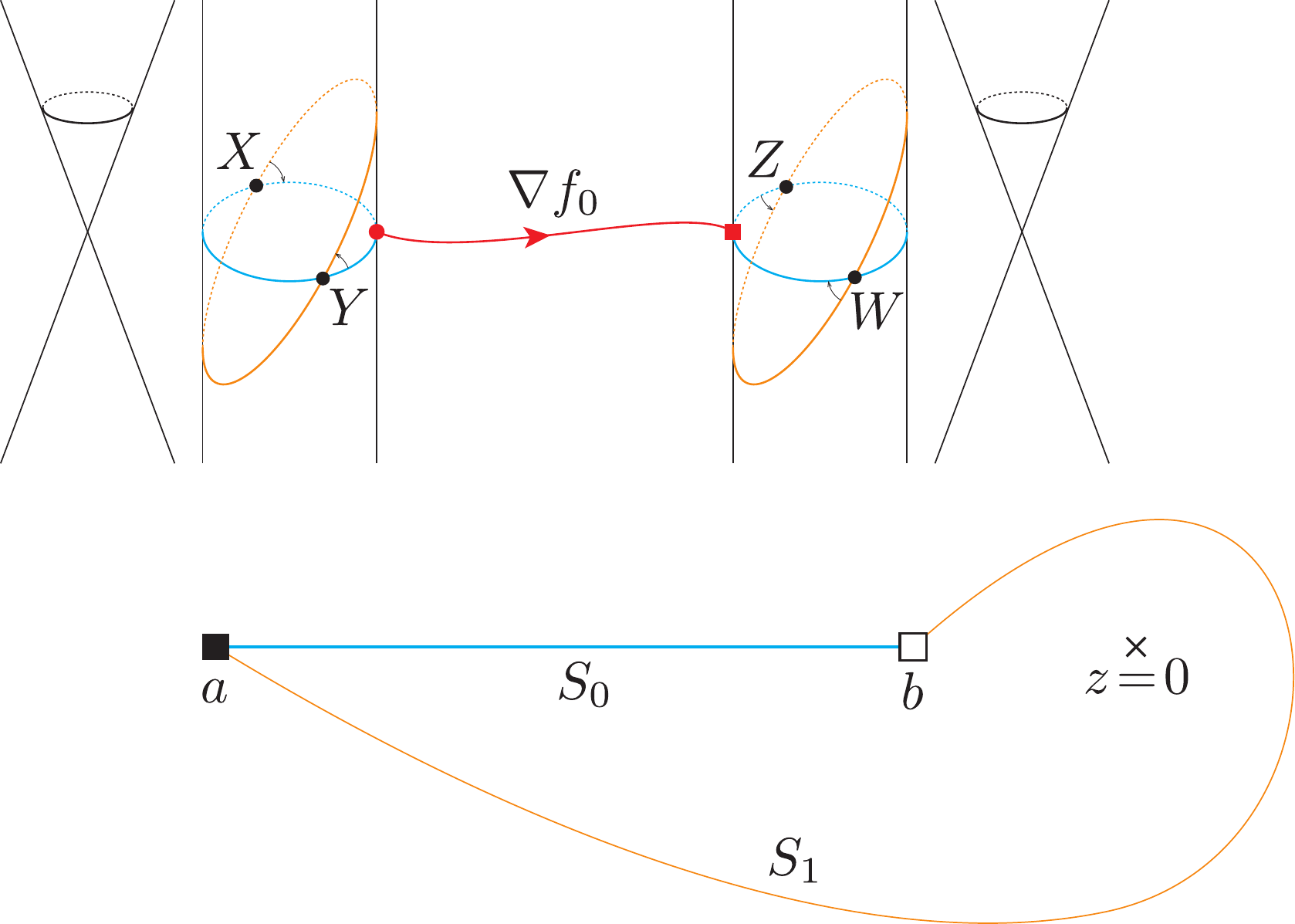}
\end{center}
\caption{Pearl trajectory contributing to $m_3(X,Y,Z)$}\label{fig:pearl1}
\end{figure}

Computation of $m_3$ involves more complicated pearl trajectories. We give an explicit picture for one of those trajectories, and the rest can be easily found in a similar way. In Figure \ref{fig:pearl1}, one can see a pearl trajectory consisting of two bigons connected by a gradient flow, which contributes to $m_3(X,Y,Z)$ with output $\bar{W}$. The red colored connecting flow in Figure \ref{fig:pearl1} is precisely the gradient trajectory in Figure \ref{fig:hopf}, and hence the corresponding moduli is isolated. The pearl trajectory degenerates into a Morse tree when perturbing $\bL_0$ back to the original position (see Figure \ref{fig:varmod}), which one of models in \cite{Abou11} takes into account.

\begin{figure}[h]
\begin{center}
\includegraphics[height=2in]{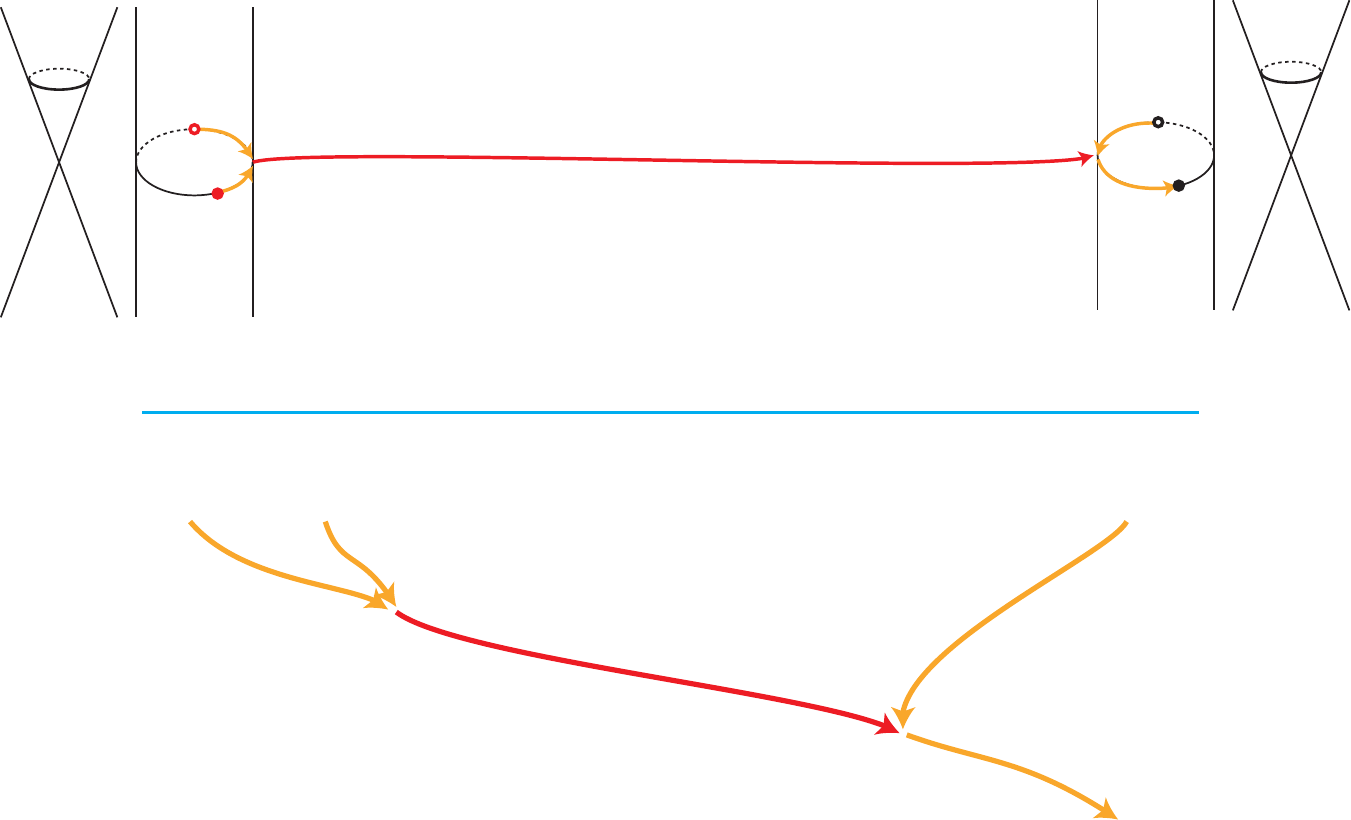}
\end{center}
\caption{Morse trees for $m_3(X,Y,Z)$}\label{fig:varmod}
\end{figure}


Consequently, we have the following complete list of nontrivial $m_3$-operations:
$$m_3(X,Y,Z) =- m_3(Z,Y,X)= \bar{W}, \quad m_3(Y,Z,W) =- m_3(W,Z,Y)= \bar{X},$$
$$ m_3(Z,W,X) =- m_3(X,W,Z)= \bar{Y}, \quad m_3(W,X,Y) =- m_3(Y,X,W)= \bar{Z}.$$
We remark that the symplectic area of a pearl trajectory becomes zero after taking limit back to the original clean intersection situation (so that it degenerates into a Morse tree), which explains why there are no $T$ appearing in the above computations. If one wants to keep working with the perturbed picture, one can simply rescale generators so that the coefficients of $m_3$ to be still $1$.

Note that the $A_\infty$-structure computed in this way precisely  coincides with the one given in Theorem \ref{thm:sumainfs}.

\subsection{Construction of mirror from $\bL$}

As in \cite{CHL2}, we take formal variables $x,y,z,w$ and consider the deformation parameter $b=xX + yY + zZ + wW$, and solve the following Maurer-Cartan equation:
\begin{equation}\label{eqn:MCbL}
 m_1(b) + m_2 (b,b) + m_3(b,b,b) + \cdots = 0.
\end{equation}
where $m_k$ with inputs involving $x,y,z,w$ are defined simply by pulling out the coefficient of Floer generators to the front, i.e.,
$$m_k ( x_1 X_1, \cdots, x_k X_k) = (-1)^\ast x_k x_{k-1} \cdots x_1 m_k (X_1, \cdots, X_k).$$
Here $(-1)^\ast$ is determined by usual Koszul sign convention, and in particular, is positive when $x_i X_i$ is one of $\{xX,yY,zZ,wW\}$.

\begin{remark}
In \cite{CHL2}, ``weak" Maurer-Cartan equations were mainly considered (i.e. the right hand side of \eqref{eqn:MCbL} replaced by $\lambda \cdot 1_\bL$), but in our example, weak Maurer-Cartan equation cannot have solutions unless $\lambda=0$ due to degree reason.
\end{remark}

By Theorem \ref{thm:sumainfs}, the Maurer-Cartan equation \eqref{eqn:MCbL} is equivalent to
$$( zyx-xyz) \bar{W} + ( wzy- yzw) \bar{X} + (  xwz-zwx) \bar{Y} + ( yxw- wxy) \bar{Z} =0.$$
Therefore, $b$ is a solution of \ref{eqn:MCbL} if and only if $(x,y,z,w)$ is taken from the path algebra (modulo relations) $\mathcal{A}$ of the following quiver with the potential:
\vspace{0.5cm}
\begin{equation*}
\xymatrix{Q: &v_0 \cdot \ar@/^0.7pc/[rr]^{z} \ar@/^1.5pc/[rr]^{x}&&\ar@/^0.7pc/[ll]^{y} \ar@/^1.5pc/[ll]^{w} \cdot v_1 &&\Phi=(xyzw)_{cyc} - (wzyx)_{cyc}}
\end{equation*}
where the vertex $v_i$ corresponds to the object $\bL_i$ and arrows correspond to degree 1 morphisms between $\bL_i$'s (or their associated formal variables). See \cite[Section 6]{CHL2} for more details. The quiver algebra has the following presentation:
\begin{equation} \label{eq:mir-A}
\mathcal{A}:= \dfrac{\Gamma Q}{ \langle \partial_x \Phi,\partial_y \Phi,\partial_z \Phi,\partial_w \Phi \rangle} = \dfrac{\Gamma Q}{\langle xyz-zyx, yzw-zwy,zwx-xwz, wxy-yxw\rangle}.
\end{equation}

\begin{remark}\label{rem:noncommresc}
$(Q,\Phi)$ is well-known to be the noncommutative crepant resolution of the conifold (see for instance \cite{Bergh}).  In fact, one can easily check that the subalgebra of $\cA$ consisting of loops based at one of vertices is isomorphic to the function algebra of the conifold. For e.g., if we set $\alpha = xy, \beta= xw, \gamma = zy, \delta=zw$, then the relations among $x,y,z,w$ force $\alpha,\beta,\gamma,\delta$ to commute with each other and to satisfy $\alpha \delta = \beta \gamma$.
\end{remark}

By \cite{CHL2}, we have an (trianglulated) $A_\infty$-functor from the Fukaya category to $D^b {\rm Mod} \, \mathcal{A}$ (here, $D^b {\rm Mod} \, \mathcal{A}$ denotes dg-enhanced triangulated category)
$$ \widetilde{\Psi}^\bL : D^b \mathcal{F} \to D^b {\rm Mod} \, \mathcal{A}.$$
In the rest of the section, we will use the degree shift $\Psi:= \widetilde{\Psi}^\bL [3]$ instead of $ \widetilde{\Psi}^\bL$ itself, in order to have the images of the geometric objects in $D^b \mathcal{F}$ lying in the standard heart ${\rm Mod} \, \mathcal{A}$ of $D^b {\rm Mod} \, \mathcal{A}$. This point will be clearer after the computation of $\Psi(S_i)$ in the next section. We will also see that the functor is an equivalence onto a certain subcategory of $D^b {\rm Mod}\, \mathcal{A}$.

%




\subsection{Transformation of $\bL_0$ and $\bL_1$ and their central charges}
We first compute the images of $\bL_0$ and $\bL_1$ themselves under $\Psi$. $\Psi(\bL_0)$ is simply a chain complex over $\mathcal{A}$ given by $CF(\bL,\bL_0)$, which is a direct sum
\begin{equation*}
\begin{array}{lcl}
CF(\bL,\bL_0) &=& CF(\bL_0,\bL_0) \oplus CF(\bL_1,\bL_0) \\
&=& \langle \one_{\bL_0}, [\pt]_{\bL_0} \rangle \oplus \langle Y,W,\bar{X},\bar{Z} \rangle
\end{array}
\end{equation*} 
with a differential $ m_1^b$.
Recall that $m_1^b (p) = \sum_k m_k (b,\cdots,b,p)$.

We have already made essential computations for $m_1^b$ (see Theorem \ref{thm:sumainfs}).
Using the previous computation, the list of $m_1^b$ acting on the generators is given as follows:
\begin{equation*}
\begin{array}{lcl}
m_1^b (\one_{\bL_0}) &=& y Y + w W \\
m_1^b (Y) &=& xw \bar{Z}  -  zw \bar{X}  \\
m_1^b (W) &=& zy \bar{X}  - xy \bar{Z}  \\
m_1^b (\bar{X}) &=& x [\pt]_{\bL_0} \\
m_1^b (\bar{Z}) &=& z [\pt]_{\bL_0} \\
m_1^b ([\pt]_{\bL_0}) &=& 0
\end{array}
\end{equation*}

Therefore $[\pt]_{\bL_0}$ is the only nontrivial class in the cohomology. Moreover, $x [\pt]_{\bL_0}$ and $z [\pt]_{\bL_0}$ are zero in the cohomology, and hence we obtain a finite dimensional representation of $(Q,\Phi)$ over $\C$. 
Consequently,  $m_1^b$-cohomology of $\Psi (\bL_0) \in D^b {\rm Mod}\, \mathcal{A}$ is 1-dimensional vector space (over $\C$) supported over the vertex $v_0$, generated by the class $[\pt]_{\bL_0}$. We remark that this vector space sits in degree 0 part due to the shift $\Psi = \WT{\Psi}^\bL [3]$.

Likewise, $\Psi (\bL_1)$ gives a $1$-dimensional $\C$-vector space supported over $v_1$ after taking $m_1^b$-cohomology. In particular, the image of $\Psi$ lies in a subcategory $D^b {\rm mod}\, \mathcal{A}$ consisting of objects with finite dimensional cohomology.

\begin{thm}\label{thm:loccatid}
$\Psi :  D^b \mathcal{F} \to D^b {\rm mod} \,\mathcal{A}$ is a fully faithful embedding, which sends $\bL_0$ an $\bL_1$ to their corresponding vertex simples.
Moreover, it is an equivalence onto $D^b_{\rm nil} {\rm mod} \, \mathcal{A}$, the full subcategory of $D^b{\rm mod} \, \mathcal{A}$ consisting of objects with nilpotent cohomologies.
\end{thm}

\begin{proof}
As  $D^b_{\rm nil} {\rm mod} \, \mathcal{A}$ is a full subcategory of $D^b {\rm mod} \,\mathcal{A}$, it suffices to prove that 
$$ \Psi : D^b \mathcal{F} \to D^b_{\rm nil} {\rm mod} \,\mathcal{A}$$
is an equivalence. Since the image of generators $\bL_0$ and $\bL_i$ are vertex simples which are nilpotent over $\mathcal{A}$, $\Psi$ lands on $D^b_{\rm nil} {\rm mod} \,\mathcal{A}$.

We  prove that the morphism level functor on $\hom (\bL_i, \bL_j)$ ($i,j=0,1$) induces isomorphisms of cohomology groups. Without loss of generality, it is enough to consider $\hom (\bL_0,\bL_0)$ and $\hom(\bL_0,\bL_1)$. By \cite[Theorem 6.10]{CHL2}, we know that both
$$ \Psi_1 : \hom (\bL_0, \bL_0) \to \hom (\Psi (\bL_0), \Psi (\bL_0))$$
$$ \Psi_1 : \hom (\bL_0, \bL_1) \to \hom (\Psi (\bL_0), \Psi (\bL_1))$$
induce injective maps on the level of cohomology. Thus, it is enough to check that 
$$\dim {\rm Ext} (\Psi (\bL_0), \Psi (\bL_0)) = 2 \qquad \mbox{and} \qquad \dim {\rm Ext} (\Psi (\bL_0), \Psi (\bL_1)) = 4.$$
On the other hand, it is known by \cite{VdB} (see \cite{Szen} also) that $\bD^b_{\rm nil} {\rm mod}\,\mathcal{A}$ is equivalent to $\bD_{\hat{Y}/Y}$ \eqref{eq:D/Y} with vertex simples  corresponding to $\mathcal{O}_{C}$ and $\mathcal{O}_C (-1)[1]$. Therefore, the computation of endomorphisms of $\mathcal{O}_C \oplus \mathcal{O}_C (-1)[1]$ due to \cite[Section 5]{CPU} finishes the proof.
\end{proof}
From now on, we take $D^b_{\rm nil} {\rm mod} \, \mathcal{A}$ to be the target category of $\Psi$.

Let $z_i$ be the central charge of $\bL_i$, namely $z_i:=\int_{\bL_i} \Omega$.
We define a central charge on quiver representations as follows. For a representation $V:=(V_0,V_1)$ of $(Q,\Phi)$ (with some maps between $V_0$ and $V_1$ which we omitted),
\begin{equation}\label{eqn:zstabquiv}
Z(V):= z_0 \dim V_0 + z_1 \dim V_1.
\end{equation}
We define a dimension vector of a representation $V$ by $\underline{\dim} (V) := (\dim V_0, \dim V_1) \in \Z_{\geq 0}^2$ for later use.
For general objects in $D^b {\rm mod} \mathcal{A}$, $V_i$ above should be replaced by the corresponding cohomology.

\begin{prop}\label{prop:centpresf}
The object level functor $\Psi_0 : Obj \left(D^b \mathcal{F} \right) \to Obj \left(D^b_{\rm nil} {\rm mod} \mathcal{A} \right)$ is a central charge preserving map.
\end{prop}
\begin{proof}
The statement is obviously true for $\bL_0$ and $\bL_1$ as they are mapped to modules with 1-dimensional cohomology supported at the corresponding vertices. Since the central charges on both sides are additive (i.e. they are the maps from the $K$-groups) and $\Psi$ is a triangulated functor, the statement directly follows.
\end{proof}

In particular, special Lagrangians $\bL_0$ and $\bL_1$ are sent to simple and hence stable objects on quiver side. 

\subsection{Stables on quiver side}
Set $\zeta_i$ to be the argument of $z_i$ taken in  $(0,\pi]$  for $i=0,1$. Since $\bL_0$ and $\bL_1$ are special Lagrangians, $\zeta_i$ is nothing but the phase of $\bL_i$. According to our convention (see the discussion below Theorem \ref{thm:main5cf}), we have $\zeta_0 > \zeta_1$. 

Nagao and Nakajima used the following notion of stability for the abelian category ${\rm mod} \mathcal{A}$. 

\begin{defn}\label{defn:abelnnst} 
We define stability of quiver representations of $(Q,\Phi)$ as follows.
\begin{enumerate}
\item
We define the phase function $\zeta$ on ${\rm mod}\,\mathcal{A}$ by
$$\zeta(V) = \zeta_0 \dim V_0+ \zeta_1 \dim V_1 \in \R$$
for $V \in {\rm mod}\,\mathcal{A}$.
\item
An object $V$ of ${\rm mod} \mathcal{A}$ is said to be stable (semistable resp.) if for any subobject $W$ of $V$,
$$\zeta (W) < \zeta ( V) \quad (\zeta(W) \leq \zeta(V) \,\,\it{resp.}).$$
\end{enumerate}
\end{defn}

It is elementary to check that $\arg Z(V) \in (0,\pi]$ for $V \in {\rm mod}\, \cA$  induces the equivalent stability on the abelian category ${\rm mod} \mathcal{A}$ as Definition \ref{defn:abelnnst}. In fact, one can easily check that two quantities have the same ordering relations. 
\begin{lemma}
For $V ,W \in {\rm mod} \, \mathcal{A}$, $\zeta (V) \leq \zeta (W)$ if and only if $\arg Z(V) \leq \arg Z(W)$.
\end{lemma}
The proof is elementary, and we omit here. 
In particular, the set of stable objects remains the same even if we use $\arg Z(V)$ in place of $\zeta(V)$. 
\begin{remark}\label{rmk:brstabnil}
One advantage of using $Z(V)$ (rather than $\zeta$) is that it can be lifted to a Bridgeland stability condition on 
$D^b_{\rm nil} {\rm mod}\,\mathcal{A}$. In fact, one can check that it is equivalent to the perverse stability on $\bD_{\hat{Y}/Y}$ via the identification $D^b_{\rm nil} {\rm mod}\, \cA \cong \bD_{\hat{Y}/Y}$. \cite[Remark 4.6]{NN} gives a correspondence between stable objects in the heart of each category.
\end{remark}

We briefly review the classification of stable objects in ${\rm mod} \cA$  following \cite{NN}. 
We first set up the notation as follows. We define $\mathcal{A}$-module $V_\pm (m)$ by $(\C^m, \C^{m \pm 1})$ together with the maps corresponding arrows given as in the left two columns in \eqref{eqn:cpmmaps} where right (left, resp.) arrows are $x,z$ ($y, w$, resp.).
Likewise, $V_\pm^\dagger (m)$ denotes $\mathcal{A}$-module $(\C^m, \C^{m \mp 1})$ which can be visualized as the right two columns in \eqref{eqn:cpmmaps}. Up to isomorphism, one can assume that all arrows act by the identity map from $\C$ to itself.

\begin{equation}\label{eqn:cpmmaps}
\xymatrix@R=0.7pt@C=0.5pt{
&& \C   \\
\C \ar[urr] \ar[drr] && \\
&& \C\\
\C \ar[rru]& \\
\vdots&& \vdots  \\
\C \ar[rrd] &&  \\
&&\C \\
\C \ar[urr] \ar[drr] && \\
&& \C  \\
&V_+ (m)&}
\qquad\qquad
\xymatrix@R=0.7pt@C=0.5pt{
\C \ar[rrd]& \\
&& \C \\
\C \ar[rru] \ar[rrd] && \\
&& \C\\
\vdots&& \vdots  \\
&&\C \\
\C \ar[rrd]\ar[rru]&& \\
&& \C  \\
\C \ar[rru]&&  \\
&V_{-} (n)&}
\qquad \qquad
\xymatrix@R=0.7pt@C=0.5pt{
&& \C  \ar[dll] \\
\C && \\
&& \ar[ull] \ar[dll]\C\\
\C & \\
\vdots&& \vdots  \\
\C &&  \\
&&\C \ar[ull] \ar[dll]\\
\C && \\
&& \C \ar[ull] \\
&V_+' (m)&}
\qquad \qquad
\xymatrix@R=0.7pt@C=0.5pt{\C & \\
&& \C \ar[ull] \ar[dll] \\
\C && \\
&& \ar[ull]\C\\
\vdots&& \vdots  \\
&&\C \ar[dll] \\
\C && \\
&& \C \ar[ull] \ar[dll] \\
\C &&  \\
&V_-' (n)&}
\qquad \qquad
\end{equation}

We are now ready to state the classification result by Nagao-Nakajima.

\begin{thm}\cite[Theorem 4.5]{NN}\label{thm:nncl}
Stable modules $W$ in ${\rm mod} \mathcal{A}$ are classified as follows:
\begin{enumerate}

\item when $\zeta_0 > \zeta_1$
\begin{itemize}
\item $V_+ (m)$ ($m \geq 1$)
\item $\mathcal{A}$-module $W$ with $\underline{\dim} W =(1,1)$ parametrized by $\hat{Y}$
\item $V_- (n)$ ($n \geq 0$)
\end{itemize}

\item when $\zeta_0< \zeta_1$
\begin{itemize}
\item $V_+ ' (m)$ ($m \geq 1$)
\item $\mathcal{A}$-module $W$ with $\underline{\dim} W =(1,1)$ parametrized by $\hat{Y}^\dagger$
\item $V_-' (n)$ ($n \geq 0$)
\end{itemize}

\end{enumerate}
In particular, $\{\zeta_1=\zeta_2\}$ gives a wall, and the wall structure consists of only two chambers.
\end{thm}

In order to precisely match the pictures in \cite{NN}, one should locate the vertex $v_1$ to the left in the quiver diagram. For instance, the vertex simple at the left vertex in \cite{NN} corresponds to $\mathcal{O}_C$ whereas in our case $v_1$ (sitting on the right) represents the Lagrangian $\bL_1(=S_1)$ which is mirror to $\mathcal{O}_C$. See \cite[Remark 4.6]{NN}.

\begin{remark}
Modules with dimension vector $(1,1)$ in (1) and (2) of Theorem \ref{thm:nncl} can be presented as
\begin{equation}\label{eqn:ccparay}
\xymatrix{ \C \ar@/^0.7pc/[rr]^{z} \ar@/^1.5pc/[rr]^{x}&&\ar@/^0.7pc/[ll]^{y} \ar@/^1.5pc/[ll]^{w} \C}
\end{equation}
for some $(x,y,z,w)\in \C^4$. Note that \eqref{eqn:ccparay} is nilpotent if and only if either $x=z=0$ or $y=w=0$.  
Notice that \eqref{eqn:ccparay} has three dimensional deformation (scaling actions of arrows $x,y,z,w$ up to overall rescaling) whereas $V_{\pm} (k)$ is  rigid for all $k$. Later, we will see that \eqref{eqn:ccparay} is mirror to a Lagrangian torus fibers whose first Betti number is 3, and $V_{\pm} (k)$ is mirror to a Lagrangian sphere.
\end{remark}

In what follows, we shall show that transformations of special Lagrangians in $D^b \mathcal{F}$ by our functor $\Psi$ recovers all the stable representations which are nilpotent. 

\subsection{Transformation of special Lagrangians in $D^b \mathcal{F}$ before/after flop}\label{subsec:bef/aft}

We next compute the transformation of other geometric objects in $D^b \mathcal{F} = D^b\langle \bL_0, \bL_1 \rangle$ under the mirror functor $\Psi$ induced by $\bL=\bL_0 \oplus \bL_1$. Recall from \ref{subsec:geobjf} that this category contains Lagrangian spheres and torus fibers  (intersecting spheres) as geometric objects.  

We begin with a torus fiber $L_c$ ($a <c <b$) in $X$ that intersects each of $\bL_0$ and $\bL_1$ along $T^2$ at $z=c$. (see \ref{subsec:geobjf}).
Suppose $L_c$ is also equipped with a flat line bundle $U$ whose holonomy along a circle in $z$-direction is $\rho$. Recall that we only consider $U$ with trivial holonomies along both of double conic fiber directions as otherwise they would not belong to the category.

\begin{lemma}\label{lem:imtoruspsi}
 The transformation of $(L_c ,\rho)$ by $\Psi$ is the representation of $(Q,\Phi)$ (after taking cohomology) given as
\begin{equation}\label{eqn:imtorusfib}
\xymatrix{ \C \ar@{..}[dd]  \ar@<0.4ex>[rr]_{\lambda_1} \ar@<-0.3ex>[rr]^{\lambda_2}&& \C \ar@{..}[dd] \\
&& \\
v_0 \cdot \ar@/^0.7pc/[rr]^{z} \ar@/^1.5pc/[rr]^{x}&&\ar@/^0.7pc/[ll]^{y} \ar@/^1.5pc/[ll]^{w} \cdot v_1 }
\end{equation}
where $[\lambda_1: \lambda_2]$ parametrizes the exceptional curve $C$ in $\hat{Y}$.
and the maps in the other direction (i.e. actions of $y$ and $w$) are zero.
\end{lemma}

\begin{proof}
Recall from Proposition \ref{prop:teqcone} that $(L_c,\rho)$ is a mapping cone $Cone(\bL_0 \stackrel{\alpha}\to \bL_1)$ for nonzero $\alpha \in HF^1(\bL_0,\bL_1)$ uniquely determined up to scaling. i.e. the following defines an exact triangle in $D^b \mathcal{F}$
\begin{equation}\label{eqn:sesrevisit}
 \bL_1 \to (L_c, \rho_z) \to \bL_0 \stackrel{[1]}{\to}.
\end{equation}

Since the functor $\Psi$ is a triangulated equivalence and $\alpha \neq 0$, $\Psi (L_c,\rho_z)$ is also a nontrivial extension of $\Psi(\bL_1)$ and $\Psi(\bL_0)$.  It is elementary exercise to show that all nontrivial extensions of these two representations which are nilpotent should be of the form given in \eqref{eqn:imtorusfib}. Moreover, since we have $HF^1 (\bL_0, \bL_1) \cong {\rm Ext}^1 (\Psi(\bL_0), \Psi(\bL_1))$ by the morphism level functor of $\Psi$, we see that $\Psi(L_c,\rho)$ gives all possible extensions as $\alpha$ varies, or equivalently $c$ and $\rho$ (in $(L_c,\rho)$) vary. Note that the family of such $(L_c,\rho)$ precisely parametrizes points in the exceptional curve $C$ by SYZ mirror construction due to \cite{CPU}.
\end{proof}

We remark that the cones 
$Cone(\bL_0 \stackrel{\lambda_a \alpha_a}{\to} \bL_1)$ and $Cone(\bL_0 \stackrel{\lambda_b \alpha_b}{\to} \bL_1)$ in $D^b \mathcal{F}$
(which are supposedly singular torus fibers)
are sent to the representations of the same form with one of $\lambda_i$ being zero, which together with $\Psi (L_c,\rho)$ completes the $\mathbb{P}^1$-family of stable representations.

\subsubsection*{Geometric argument}
We provide a more geometric computation of $\Psi (L_c,\rho)$ making use of pearl trajectory model introduced in \ref{subsec:FCbL}.
As shown in Figure \ref{fig:tcone_pert}, $L_c$ intersect $\bL_0 \cup \bL_1$ at eight different points after perturbation. Let 
\begin{equation*}
\begin{array}{l}
\bL_0 \cap L_c :=\{a_{00}, a_{01}, a_{10}, a_{11} \} \\
\bL_1 \cap L_c:=\{b_{00},b_{01},b_{10},b_{11} \}
\end{array}
\end{equation*}
(see Figure \ref{fig:tcone_pert}) where $\deg a_{ij} = i+j+1$ and $\deg b_{ij} = i+j$. Note that the degrees for generators in $\bL_0 \cap L_c$ are shifted by $1$ due to Floer theoretic grading from intersections in $z$-direction.  

\begin{figure}[h]
\begin{center}
\includegraphics[height=2.5in]{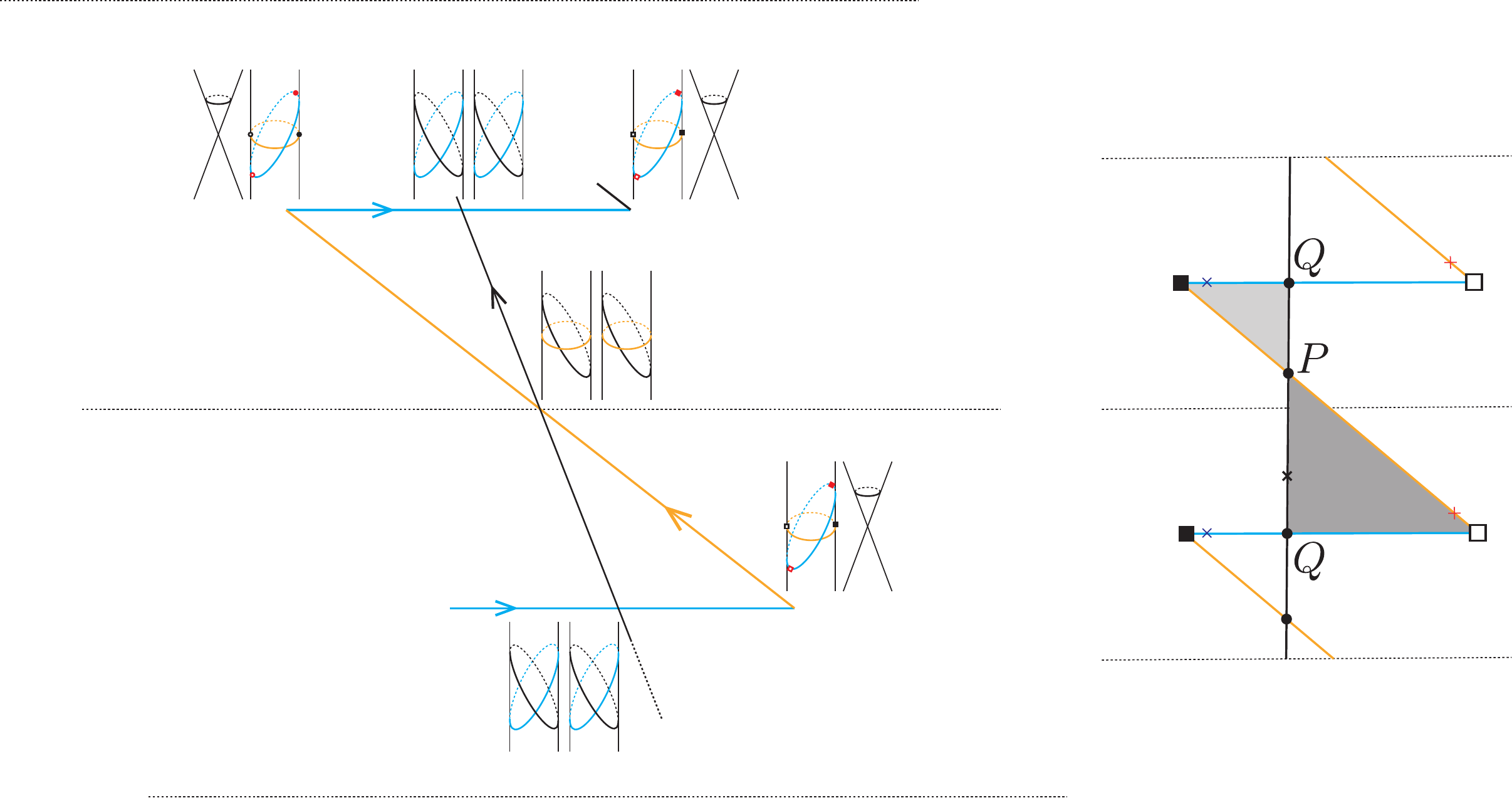}
\end{center}
\caption{Perturbation of $L_c, \bL_i$ and polygons contributing to $m_1^b$ on $CF(\bL,(L_c,\rho))$ }\label{fig:tcone_pert}
\end{figure}

$m_1^b$ for the above generators can be computed as follows:
\begin{equation*}
\begin{array}{lcl}
m_1^b (b_{00}) &=& (T^{\omega(\Delta_1)} x-\rho T^{\omega(\Delta_2)}z) a_{00} + xy b_{01} + zy b_{10}\\
m_1^b (b_{01}) &=& (\rho T^{\omega(\Delta_2)} z- T^{\omega(\Delta_1)}x) a_{01} -zw b_{11} \\
m_1^b (b_{10}) &=& (\rho T^{\omega(\Delta_2)} z- T^{\omega(\Delta_1)}x) a_{10} + xy b_{11} \\
m_1^b (b_{11}) &=& (T^{\omega(\Delta_1)} x- \rho T^{\omega(\Delta_2)} z) a_{11}
\end{array}
\,\,
\begin{array}{lcl}
m_1^b (a_{00}) &=&  yx a_{01} + wz a_{10} \\
m_1^b (a_{01}) &=&  -wz a_{11}\\
m_1^b (a_{10}) &=&  yx a_{11}\\
m_1^b (a_{11}) &=& 0
\end{array}
\end{equation*}
Here, the coefficients of the form $(T^{\omega(\Delta_1)} x-\rho T^{\omega(\Delta_2)} z)$ in front of $a_{ij}$ (in $m_1^b (b_{ij})$) is from the pair of holomorphic polygons projecting to the shaded triangles $\Delta_1$ and $\Delta_2$ in $z$-plane shown in Figure \ref{fig:tcone_pert}. (They are the same triangles appearing in the proof of Lemma \ref{lem:ratioalpha}.) Other terms are contributed by pearl trajectories consisting of two $2$-gons joined by a gradient flow, which have the same shape as the one contributing to $m_3$ drawn in Figure \ref{fig:pearl1}.

Therefore the $m_1^b$-cohomology is generated by $[a_{11}]$ as $\mathcal{A}$-module, and we have $T^{\omega(\Delta_1)} x[a_{11}] = \rho T^{\omega(\Delta_2)} z[a_{11}]$ (since their difference is $m_1^b (b_{11})$). Moreover, these are the only nontrivial scalar multiplication since $wz[a_{11}]=yx[a_{11}]=0$. In particular, we see that $\lambda_1,\lambda_2$ in \eqref{eqn:imtorusfib} satisfy $\lambda_1:\lambda_2=\rho T^{\omega(\Delta_2)} : T^{\omega (\Delta_1)}$.


\vspace{0.5cm}
On the other hand, $\Psi$ transforms Lagrangian spheres $S_k$ into the following stable representations.

\begin{prop}\label{prop:imsk}
The images of spheres $\{S_k : k \in \Z\}$ in $D^b \mathcal{F}$ (after taking cohomology) are given as follows:
\begin{enumerate}
\item For $m \geq 1$, $\Psi (S_m) = V_+ (m)$,
\item For $n \leq 0$, $\Psi (S_n) = V_- (|n|)$.
\end{enumerate}
where $V_\pm (k)$ are as in (left two columns of) \eqref{eqn:cpmmaps}.
\end{prop}

\begin{proof}
We will only prove (1), and the proof of (2) can be done in a similar manner. The statement is true for $m=1$ by Theorem \ref{thm:loccatid}. We will proceed by induction. Let us assume that it is true for $m$. By Proposition \ref{prop:smsmt}, we have
$$S_{m+1} \cong Cone (L_c \stackrel{\alpha}{\to} S_{m})$$
for some $\alpha$ which implies the exact triangle $S_{m} \to S_{m+1} \to L_c \stackrel{[1]}{\to}$ in $D^b \mathcal{F}$. Since $\alpha$ is a nonzero element in the Floer cohomology, $S_{m+1}$ is a nontrivial extension of $S_m$ and $L_c$. 
We see that $\Psi(S_{m+1})$ is an extension of $V_{+} (m)$ and $\Psi(L_c)=\xymatrix{ \C \ar@<0.4ex>[r]_{\lambda_1} \ar@<-0.3ex>[r]^{\lambda_2} & \C}$. 
For simplicity, we choose a suitable $L_c$ such that $\lambda_1 = -\lambda_2$, and hence, after rescaling two arrows act as $id$ and $-id$ respectively.

On the other hand, we already know one nontrivial extension of these two modules, which is nothing but $V_{+} (m+1)$. To see this, observe that the map $\sigma :V_+ (m) \to V_{+} (m+1)$ defined by
\begin{equation*}
 \sigma: 
 \begin{array}{l}
 e_i \mapsto e_i + e_{i+1}, \\
  f_j \mapsto f_j + f_{j+1}
  \end{array}
\end{equation*}
is an injective $\mathcal{A}$-module map, where $e_i$ (resp. $f_j$) denotes the standard basis of $V_+(m)$ spanning the $i$-th (resp. $j$-th) component $\C$ over $v_0$ (resp. $v_1$) for $1\leq i \leq m-1$ (resp. $1 \leq j \leq m$). See \eqref{eqn:basischoosevm} below.
\begin{equation}\label{eqn:basischoosevm}
\xymatrix@R=0.7pt@C=0.5pt{
&& \C=\langle f_1 \rangle    \\
 \langle e_1 \rangle =\C \ar[urr]^{x} \ar[drr]_{z} && \\
&& \C=\langle f_2 \rangle \\
\langle e_{2} \rangle =\C \ar@<-0.4ex>[rru]& \\
\vdots&& \vdots  \\
\langle e_{m-2	} \rangle = \C \ar@<0.5ex>[rrd]  &&  \\
&&\C = \langle f_{m-1} \rangle  \\
\langle e_{m-1} \rangle = \C \ar[urr]^{x} \ar[drr]_{z} && \\
&& \C =\langle f_{m} \rangle  \\
&V_+ (m)&}
\end{equation}
Now the cokernel of $\sigma$ is spanned by $[e_1]$ and $[f_1]$, and
$x [e_1] = [x e_1] = [f_1]$ and $z [e_1] = [z e_1] = [f_2] = -[f_1]$ since $f_1 + f_2$ is in the image of $\sigma$. Therefore, we have
$$ 0 \to V_+ (m) \stackrel{\sigma}{\to} V_{+} (m+1) \stackrel{	}{\to} \left( \xymatrix{ \C \ar@<0.4ex>[r]_{-id} \ar@<-0.3ex>[r]^{id} & \C} \right) \to 0.$$

Moreover, $V_+ (m+1)$ is the only nontrivial extension of $V_+ (m)$ and $ \xymatrix{ \C \ar@<0.4ex>[r]_{-id} \ar@<-0.3ex>[r]^{id} & \C}$ since
$$\dim \,{\rm Ext} \left( \xymatrix{ \C \ar@<0.4ex>[r]_{-id} \ar@<-0.3ex>[r]^{id} & \C}, V_+(m) \right) \cong \dim \, HF^1 (L_c, S_m) = 1.$$
(Here, we  used the fact that $\Psi$ is an equivalence, and the induction hypothesis that $\Psi(S_m) = V_+ (m)$.) We conclude that $S_{m+1}$ should map to $V_+ (m+1)$ by $\Psi$.

\end{proof}

\subsubsection*{Geometric argument}
Alternatively, one can compute the image of spheres under $\Psi$ directly by holomorphic disk counting (or more precisely computing $m_1^b$ on $CF(\bL,S_i)$). We give a brief sketch of the computation for $S_m$ for $m \geq 1$. On $z$-plane, the projection of $S_m$ intersects the interval $(a,b)$ $(m-1)$-times (not including the end points $a$ and $b$). Here, we perturb $\bL_0, \bL_1$ and $S_m$ along the fiber direction as in the proof of Lemma \ref{lem:imtoruspsi} so that they mutually intersect transversely.

Let us denote these $m-1$ points in $(a,b)$ by $c_1, c_2, \cdots, c_{m-1}$ as shown in Figure \ref{fig:smmm}.
These are the only locations where one has the highest degree intersections (i.e. degree $3$ elements in $CF(\bL,S_m)$ that map to degree 0 elements by $\Psi = \WT{\Psi}^{\bL} [3])$. 
Cohomology long exact sequence tells us that it is enough to consider these elements, as cohomologies of $\Psi(S_0)$ and $\Psi(S_1)$ are supported only at this degree.
We remark that the intersection $\bL \cap S_m$ occurring at $z=a$ and $z=b$ only produces degree 0 and 1 elements in the Floer complex, which are not in the highest degree.
\begin{figure}[h]
\begin{center}
\includegraphics[height=2.7in]{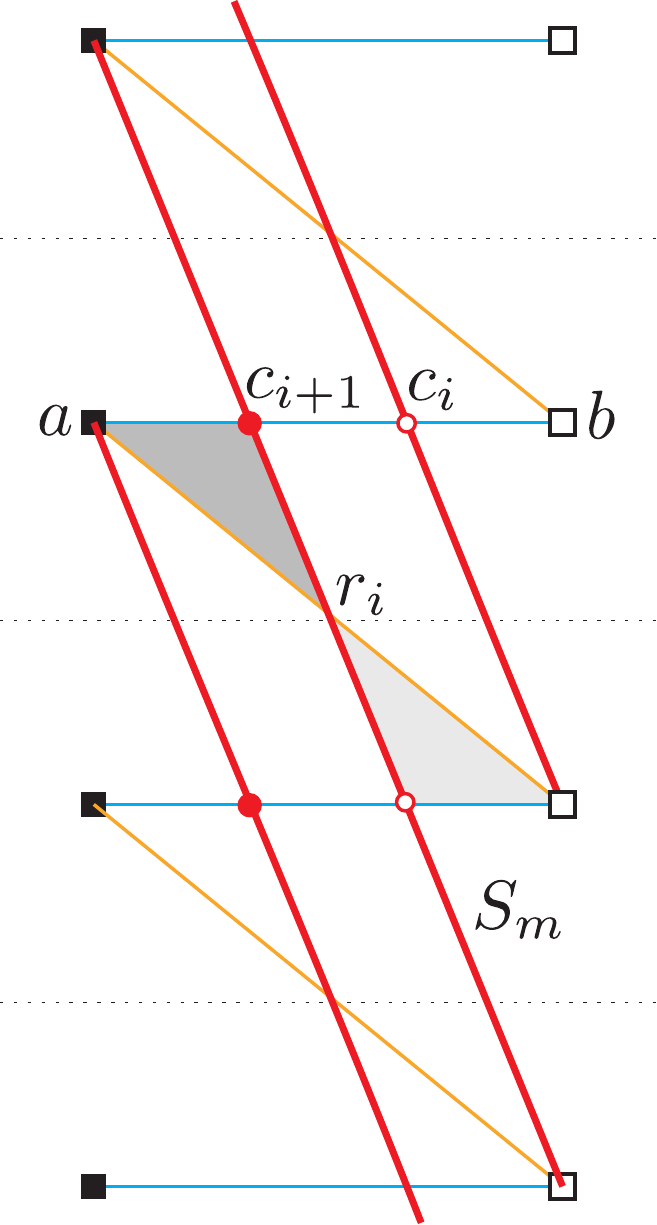}
\end{center}
\caption{Contributing holomorphic polygons to (the differential of) $\Psi(S_m)$}\label{fig:smmm}
\end{figure} 

Denote these highest degree intersection points by
$$\tilde{c}_1, \tilde{c}_2, \cdots, \tilde{c}_{m-1}$$
where $\tilde{c}_i$ projects down to $c_i$. (Obviously there is only one highest degree intersection point over each $\tilde{c}_i$ after perturbation.) There are pairs of triangles whose $z$-projections are given as shaded region in Figure \ref{fig:smmm}. Each of these pair contributes to $m_1^b$ with the same input, say $r_i$ as in the figure, and gives rise to
\begin{equation}\label{eqn:tricici1}
m_1^b (r_i)= z \tilde{c}_i - x \tilde{c}_{i+1}.
\end{equation}
Also, there are pearl trajectories with a similar shape to the ones contributing to formal deformation ($m_3$) of $\bL$, which lies over the interval $[a,b]$. 
These trajectories induce
\begin{equation}\label{eqn:looptrivf}
m_1^b (p_i) = \pm yx \tilde{c}_i, \quad m_1^b (q_i) = \pm wz \tilde{c}_i
\end{equation}
where $p_i$ and $q_i$ are the two degree 2 intersection points lying over $c_i$.

Set $e_i:=[\tilde{c}_i]$ which belong to the $v_0$-component of the resulting quiver representation . \eqref{eqn:tricici1} implies
$$ z e_i = x e_{i+1},$$
and we denote this element by $f_i$ which belongs to $v_1$-component. We also set $f_1 : = x e_1$ and $f_{m} := z e_{m-1}$. Combining \eqref{eqn:tricici1} and  \eqref{eqn:looptrivf} implies all other actions of $\cA$ are trivial. Therefore the resulting cohomology has precisely the same as $V_+ (m)$ as a $\cA$-module.

\subsubsection*{Effect of A-flop}
The images of $\bL_0=S_0$ and $\bL_1=S_1$ under the symplectomorphism $\rho : X_{s=0} \to X_{s=1}$ gives another Lagrangian spheres $S_0'$ and $S_1'$ (see Figure \ref{fig:rotdc1}). In addition, we have a new sequence of special Lagrangian spheres $\left\{S_k' :k\in\Z \right\}$ depicted in Figure \ref{fig:sprimes}. Readers are warned that $S_k'$ is not a image of $S_k$ under A-flop unless $k=0$ or $k=1$. Note that phases of other spheres lie between those of $S_0'$ and $S_1'$ due to our choice of gradings (or orientations) as in Figure \ref{fig:sprimes}. (Recall that we measure the phase angles in clockwise direction.)

We perform the same mirror construction making use of $\bL_0'=S_0'$ and $\bL_1'=S_1'$ which obviously produce the same quiver with the potential. Only difference is that now the ordering of the phases of $\bL_0'$ and $\bL_1'$ are switched. Namely, in this case, we have $\zeta_0' < \zeta_1'$ where $\zeta_i'$ is a phase of $\bL_i'$. Thus one can naturally expect to obtain stable representations in (2) of Theorem \ref{thm:nncl} by applying the resulting functor $\Psi' : D^b \langle S_0' ,S_1' \rangle \to D^b_{\rm nil} {\rm mod} \, \mathcal{A}$ to $\left\{S_k' \right\}$ and new torus fibers (which can be represented as straight vertical lines in Figure \ref{fig:sprimes}).

\begin{figure}[h]
\begin{center}
\includegraphics[height=2.7in]{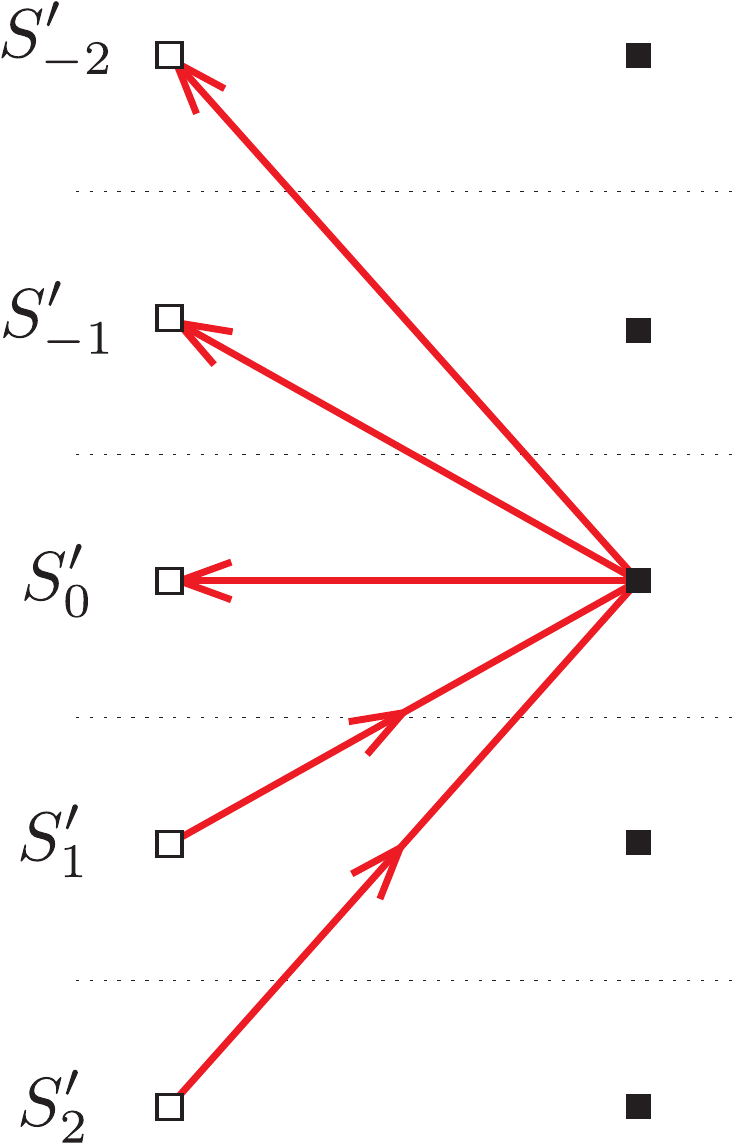}
\end{center}
\caption{Special Lagrangian spheres $S_k'$ in $X_{s=1}$}\label{fig:sprimes}
\end{figure}

\begin{prop}
Torus fibers intersecting $S_0'$ and $S_1'$ are transformed under $\Psi'$ into
\begin{equation*}
\xymatrix{ \C  & \C   \ar@<0.4ex>[l]_{\lambda_1} \ar@<-0.3ex>[l]^{\lambda_2} }
\end{equation*}
after taking $m_1^b$-cohomology, where $[\lambda_1:\lambda_2]$ parametrizes the exceptional curve $C^\dagger$ in $\hat{Y}^\dagger$.
The images of spheres $\{S_k' : k \in \Z\}$ in $D^b \mathcal{F}'$ (after taking cohomology) are given as follows:
\begin{enumerate}
\item For $m \geq 1$, $\Psi (S_m') = V_+^\dagger (m)$,
\item For $n \leq 0$, $\Psi (S_n') = V_-^\dagger (|n|)$.
\end{enumerate}
where $V_\pm^\dagger (k)$ are as in (right two columns of) \eqref{eqn:cpmmaps}.
\end{prop}

The proof is essentially the same as that of Proposition \ref{prop:imsk}, and we will not repeat here.

\subsection{Bridgeland stability on $D^b \mathcal{F}$}\label{subsec:bridgelandfuk}
As mentioned in Remark \ref{rmk:brstabnil}, $D^b_{\rm nil} {\rm mod}\, \cA$ admits a Bridgeland stability through the isomorphism $D^b_{\rm nil} {\rm mod}\, \cA \cong \bD_{\hat{Y}/Y}$. Therefore, $D^b \mathcal{F}$ also admits a Bridgeland stability condition $(Z,\mathcal{S})$ by pulling-back the one on $D^b_{\rm nil} {\rm mod} \, \mathcal{A}$ via the equivalence $\Psi$. By Proposition \ref{prop:centpresf}, we see that the pull-back stability on $D^b \mathcal{F}$ is geometric in the sense that its central charge is given by the period $\int \Omega_{s=0}$.
Moreover, the discussion in \ref{subsec:bef/aft} tells us that the special Lagrangian spheres $S_k$ and tori (that intersects spheres) are stable objects in the heart. 

After applying A-flop, we consider the subcategory $\mathcal{F}'$ generated by $S_0'$ and $S_1'$ in $X_{s=1}$. By the same reason, $D^b \mathcal{F}'$ admits a Bridgeland stability condition whose stable objects (in the heart) are special Lagrangian spheres $S_k'$'s and new torus fibers in $X_{s=1}$.
Note that their inverse image under $\rho:X_{s=0} \to X_{s=1}$ are $\{S_i^\dagger\}$, the A-flop of the spheres $\left\{S_i \right\}$ in $X_{s=0}$, which were discussed in Section \ref{subsec:AflopLag}. 

By pulling back this stability condition on $D^b \mathcal{F}'$ via $\rho_\ast$, we get a new stability condition $(Z^\dagger,S^\dagger)$ on $D^b \mathcal{F}$ (whose central charge comes from $\rho^\ast \Omega_{s=1}$) with the set of stable objects 
$$\mathcal{S}^\dagger= \{L^\dagger : L \,\, \mbox{is stable with respect to} \,\, (Z,\mathcal{S})\}.$$
This proves Theorem \ref{thm:aflopstab}.


\bibliographystyle{amsalpha}
\bibliography{geometry}

\end{document}